\documentclass[12pt,reqno]{article}
\usepackage{mathtools}
\usepackage{amsfonts,amsmath,amscd,amssymb,amsthm}
\usepackage{fullpage}
\usepackage{txfonts}
\usepackage{cleveref}
\usepackage{autonum}

\newcommand{\N}{\mathbb{N}}
\newcommand{\Z}{\mathbb{Z}}

\newcommand{\R}{\mathbb{R}}
\newcommand{\C}{\mathbb{C}}
\newcommand{\E}{\mathbb{E}}

\newcommand{\lam}{\lambda}

\newcommand{\supp}{\text{supp}}

\newcommand{\Tr}{\text{\rm Tr}}
\newcommand{\diag}{{\rm{diag}}}

\newtheorem{info}{}
\newtheorem{theorem}[info]{Theorem}
\newtheorem{corr}[info]{Corollary}

\newtheorem{lem}[info]{Lemma}
\newtheorem{prop}[info]{Proposition}
\newtheorem{remark}[info]{Remark}
\numberwithin{info}{section}

\numberwithin{equation}{section}
\renewcommand{\[}{\begin{equation}}
	\renewcommand{\]}{\end{equation}}
\makeatletter

\g@addto@macro\normalsize{%
	\setlength\abovedisplayskip{5pt}
	\setlength\belowdisplayskip{5pt}
	\setlength\abovedisplayshortskip{4pt}
	\setlength\belowdisplayshortskip{4pt}}
\makeatother

\renewcommand{\cal}{\mathcal}

\newcommand{\Lam}{\Lambda}
\newcommand{\sq}{\sqrt}
\newcommand{\sign}{\Sign}

\renewcommand{\sq}{\sqrt}

\newcommand{\To}{\Rightarrow}
\renewcommand{\Im}{\mathrm{Im}}
\renewcommand{\Re}{\mathrm{Re}}

\renewcommand{\Re}{\mathrm{Re}}

\renewcommand{\H}{\mathbb{H}}
\newcommand{\pf}{\mathrm{pf}}
\newcommand{\Sign}{\mathrm{Sign}}

\renewcommand{\k}{\kappa}
\newcommand{\pv}{\mathrm{p.v}}

\newcommand{\Ai}{\mathrm{Ai}}

\newcommand{\arcos}{\mathrm{arccos}}
\newcommand{\arcosh}{\mathrm{arcosh}}
\DeclarePairedDelimiter{\ceil}{\lceil}{\rceil}
\newcommand{\linSpan}{\mathrm{span}}

\begin{document}
	\title{Gaussian Multiplicative Chaos for Gaussian Orthogonal and Symplectic Ensembles}
	
	\author{
		Pax Kivimae\thanks{Department of Mathematics, Northwestern University. Email: kivimae@math.northwestern.edu}
	}
	
	\maketitle
	
	\begin{abstract}
		We study the characteristic polynomials of both the Gaussian Orthogonal and Symplectic Ensembles. We show that for both ensembles, powers of the absolute value of the characteristic polynomials converge in law to Gaussian multiplicative chaos measures after normalization for sufficiently small real powers. The main tool is a new asymptotic relation between the fractional moments of the absolute characteristic polynomials of Gaussian Orthogonal, Unitary, and Symplectic Ensembles.
	\end{abstract}
	
	\section{Introduction}
	
	The Gaussian Orthogonal/Unitary/Symplectic Ensembles are probability measures on the space of $N \times N$ Symmetric/Hermitian/Quaternionic-Hermitian matrices with density
	\begin{equation}\label{eq:defEnsemble}
		\P_\beta(A)=\frac{1}{Z_{N,\beta}}\exp(-N\beta \Tr(A^2)),
	\end{equation}
	where here $\beta=1/2/4$, respectively, and $Z_{N,\beta}$ is a normalization constant. We will write $A_{N,\beta}$ for a matrix sampled from the measure \eqref{eq:defEnsemble}.
	
	The focus of this paper will be the study of the (absolute) characteristic polynomial of matrices sampled from these ensembles in the large $N$ limit. Our understanding of the behavior of such random characteristic polynomials for these (and other) matrix ensembles has seen much progress in recent years (see, for example \cite{GUEGMC} and the references therein). Much of this progress has been motivated by conjectures on these processes due to Fyodorov, Hiary, Keating, and Simm in \cite{fyodorovdistribution,hiary,fyodorovfreezing}, who predict a detailed picture of the extremal values of these processes by relating these characteristic polynomials to objects from the theory of log-correlated fields. Key to this picture is the foundational work of Fyodorov, Khoruzhenko, and Simm \cite{FyoGUE}, which shows that for $\beta=2$, the normalized logarithm of the characteristic polynomial converges (in a suitable sense) to a log-correlated Gaussian random field, expanding the picture derived by Hughes, Keating, and O'Connell \cite{GFFCircular} in the Unitary case. More explicitly, in \cite{FyoGUE} they construct a centered Gaussian field $X$ on $(-1,1)$ with covariance kernel
	\begin{equation}
		\mathbb E X(x) X(y) = -\frac{1}{2} \log |2(x-y)|.
	\end{equation}
	Due to the divergence of this covariance kernel on the diagonal, $X$ does not exist as a random function, but instead, they show that it may be constructed as a random variable valued in a suitable space of distributions. They then proceed to show that the normalized logarithm of $|\det(A_{N,2}-xI)|$ converges in law to $X$ with respect to the topology given by a suitable Sobolev norm.
	
	A key tool used to understand the geometric properties of log–correlated fields (such as $X$) is their associated (GMC) Gaussian multiplicative chaos measures, introduced by Kahane \cite{kahane}, which are roughly given by regularizing the exponential of the field. For more background on log-correlated fields and GMC measures, the reader is invited to consult \cite{GMCintro}.
	
	In view of the results of \cite{FyoGUE, GFFCircular}, one thus expects that the asymptotic behavior of the characteristic polynomials of these ensembles and their powers should be described by a family of GMC measures, and indeed it is from this perspective that \cite{fyodorovdistribution,fyodorovfreezing} were able to formulate their conjectures. Much progress has been achieved in establishing this picture rigorously. In particular, for $\beta=2$, Berestycki, Webb, and Wong \cite{GUEGMC} establish that small enough powers of the absolute value of the characteristic polynomial can be described in the large $N$ limit by a Gaussian multiplicative chaos (GMC) measure. Explicitly, they construct a family of random measures $\mu_{\alpha}$, for $\alpha\in (0,2)$, which may be heuristically written as
	\[\label{eq:GMCdef}\mu_{\alpha}(dx)=\exp\left(\alpha X(x) - \frac{\alpha^2}{2}\E[X(x)^2]\right)dx.\]
	They then show that when viewed as random measures on $(-1,1)$, the sequence
	\[\frac{|\det(A_{N,2}-xI)|^{\alpha}}{\E[|\det(A_{N,2}-xI)|^{\alpha}]}dx\]
	converges in law to $\mu_{\alpha}$ with respect to the topology of weak convergence of measures. Our main result will be an analog of this result for the Gaussian Orthogonal and Symplectic ensembles (i.e., $\beta=1,4$).
	
	\begin{theorem}
		\label{GMCtheorem}
		For $\alpha\in (0,1/\sq{2})$, we have as $N\to \infty$ that
		\[\frac{|\det(A_{2N,1}-xI)|^{\alpha}}{\E[|\det(A_{2N,1}-xI)|^{\alpha}]}dx\To \mu_{\sq{2}\alpha}(dx)\]
		in law with respect to the topology of weak convergence of measures. 
		
		\noindent	Similarly, for $\alpha\in (0,\sq{2})$ we have as $N\to \infty$ that
		\[\frac{|\det(A_{N,4}-xI)|^{\alpha}}{\E[|\det(A_{N,4}-xI)|^{\alpha}]}dx\To \mu_{\alpha/\sq{2}}(dx)\]
		in law with respect to the topology of weak convergence of measures.
	\end{theorem}
	
	We note that some analogous results are known outside of the case of Gaussian ensembles. In particular, Webb \cite{WebbL2} showed that suitably small powers (more precisely powers in the $L^2$-regime) of the characteristic polynomial of a Haar distributed random unitary matrix converges in law to a certain GMC measure on the unit circle, after a suitable normalization. This result was later extended by Nikula, Saksman, and Webb \cite{WebbL1} to the larger $L^1$-regime, which is believed to be optimal. Additionally, Forkel and Keating \cite{ForkelKeating} have shown a similar result for the characteristic polynomials of matrices sampled from the Haar measures on both the orthogonal and symplectic groups.
	
	In all of these cases though, the results rely on precise asymptotics for fractional moments of characteristic polynomials (for example, see Sections 3 and 4 of \cite{GUEGMC}). Depending on the case, this becomes equivalent to obtaining asymptotics for certain Toeplitz, Hankel, or Toeplitz+Hankel determinants associated with certain measures with Fischer-Hartwig singularities (see \cite{RHconjectures} for further background). These asymptotics have been computed in many cases through the use of Riemann-Hilbert methods (see, for example, \cite{Krav}). These methods though, are specific to the case of $\beta=2$, with no clear generalization to the $\beta=1,4$ cases.
	
	Here we take a different approach to prove Theorem \ref{GMCtheorem}. Instead of calculating the fractional moments directly (which appear to still be inaccessible to current methods), we instead relate fractional moments of $|\det(A_{2N,1}-xI)|$ and $|\det(A_{N,4}-xI)|$ to those of $|\det(A_{2N,2}-xI)|$. For notational convenience in stating our result, here and elsewhere, we will denote
	\[D_{N,\beta}(x)=\det(A_{N,\beta}-xI).\]
	Our main tool in proving Theorem \ref{GMCtheorem} is then the following asymptotic relation, which may also be of independent interest.
	
	\begin{theorem}
		\label{main-theorem}
		Let $m\ge 1$, and for each $1\le i\le m$, let us take $\alpha_i>0$ and $\lam_i\in (-1,1)$ such that $-1<\lam_m<\cdots <\lam_1<1$. Let $\cal{W}$ be a smooth, compactly-supported function that coincides with a polynomial on a neighborhood of $[-1,1]$. Then we have that
		\begin{align}\label{identity}
			\E\Bigg[e^{\Tr(\cal{W}(A_{2N,1}))}&\prod_{i=1}^{m}|D_{2N,1}(\lam_i)|^{\alpha_i}\ \Bigg]
			\times \E\Bigg[e^{\Tr(2\cal{W}(A_{N,4}))}\prod_{i=1}^{m}|D_{N,4}(\lam_i)|^{2\alpha_i}\Bigg] \\
			&= \E\Bigg[e^{\Tr(2\cal{W}(A_{2N,2}))}\prod_{i=1}^{m}|D_{2N,2}(\lam_i)|^{2\alpha_i}\Bigg](1+O(N^{-1/6})).
		\end{align}
		Moreover the error term is uniform over compact subsets of $\{\lam_i\}_{i=1}^{m}\in \{(\lam_1,\cdots \lam_m)\in (-1,1)^m:\lam_1>\cdots >\lam_m\}$.
	\end{theorem}
	
	\begin{remark}
		Theorem \ref{main-theorem} may be compared to the following relation given by Baik and Rains \cite{baik} (also noted in \cite{RHconjectures}). Denote by $O_N,U_N$, and $S_N$, a random matrix sampled from the Haar measure on $\mathrm{SO(N)}$, $\mathrm{U(N)}$, and $\mathrm{Sp(N)}$, respectively. Then for a large class of functions $g$  (including integrable functions), we have the following relation
		\[\E [\det(g(O_{2N+2}))]\E [\det(g(S_{2N}))]=\E[ |\det(g(U_{2N}))|^2].\]
	\end{remark}
	
	For other occurrences of GMC measures in random matrix theory, see \cite{GMConDisk, ChiCircle}. We also mention the work of Lambert, Ostrovsky, and Simm \cite{Lambert}, where they establish the convergence of a regularized version of $\log(|D_{N,\beta}(x)|)$ to a log-correlated field for arbitrary $\beta>0$. Additionally, we also mention the recent work of Bourgade, Mody, and Pain \cite{bour}, where they establish normalized fluctuations of $\log(|D_{N,\beta}(x)|)$ at a finite collection of points. In the work of Claeys, Fahs, Lambert, and Webb \cite{clae}, a similar convergence to a GMC measure is shown for the exponential of the eigenvalue counting function, from which they derive optimal bounds on eigenvalue fluctuations. Lastly, for some beautiful relationships between log-correlated fields, random matrices, and the behavior of special functions in number theory, one may consult \cite{KatzSarnakBook,fyodorovfreezing,fyodorovdistribution}, and the references therein.
	
	Lastly, we also remark that the restriction to the even case for $\beta=1$ in Theorem \ref{GMCtheorem} appears to be a technical artifact of our method of proof, which relies crucially on Theorem \ref{main-theorem}. Similarly, we note that the parameter ranges for $\alpha$ in Theorem \ref{GMCtheorem} amount to only a strict subset of the $L^2$-phase, which occurs for $\alpha<\sq{\beta}$. This restriction as well is a limitation of our method and is due to having to rely on the convergence of the $\beta=2$ at larger powers, as is explained below. 
	
	\subsection{Construction of GMC and sketch of the proof of Theorem \ref{GMCtheorem} from Theorem \ref{main-theorem}} \label{outline1}
	We now outline our method to obtain Theorem \ref{GMCtheorem} from Theorem \ref{main-theorem}. Details are given in Section \ref{GMC section}. We begin by recalling the basics of the construction of the GMC measure \eqref{eq:GMCdef}. We first recall a sequence of Gaussian processes on $(-1,1)$, denoted $X_M$, such that a.s. $X_M\To X$ as $M\to \infty $ in the distributional sense. This decomposition comes from the identity (see, for example, Appendix C of \cite{Log-Chebyshev}) that for $x,y\in (-1,1)$, we have that
	\[-\frac{1}{2}\log(2|x-y|)=\sum_{k=1}^{\infty}\frac{1}{k}T_k(x)T_k(y),\]
	where $T_k$ is $k$-th Chebyshev polynomial of the first kind defined by the relation $T_k(\cos(\theta))=\cos(k\theta)$ for $\theta\in (0,2\pi]$. The polynomial $T_k$ is also the $k$-th orthogonal polynomial with respect to the measure $\mu_{as}$, where $\mu_{as}(x)=\frac{2}{\pi}(1-x^2)^{-1/2}I(|x|<1)$ denotes the (shifted) arcsine distribution. In particular, from this, we have the following (formal) distributional identity
	\[X(x)=\sum_{k=1}^{\infty}\frac{Z_k}{\sq{k}}T_k(x),\]
	where $(Z_k)_k$ is a sequence of i.i.d standard Gaussians random variables. We define \[X_M(x)=\sum_{k=1}^{M}\frac{Z_k}{\sq{k}}T_k(x),\]
	which is the formal projection of $X$ onto polynomials of degree-$M$ in $L^2(\mu_{as})$.
	
	If we define $\mu_{M,\alpha}(dx)=\exp(\alpha X_M(x)-\frac{\alpha^2}{2}\E [X_M(x)^2])dx$, it is shown in \cite{GUEGMC} that for $\alpha\in (0,2)$, the measures $\mu_{M,\alpha}$ converge weakly in law to a nontrivial limiting measure, which will be our $\mu_{\alpha}$. In addition, for $\alpha\in (0,\sq{2})$, it is shown that $\mu_{M,\alpha}$ converges to $\mu_{\alpha}$ in the $L^2$-norm as well.
	
	We now define
	\[\mu_{N,\beta,\alpha}(dx)=\frac{|D_{N,\beta}(x)|^{\alpha}}{\E[|D_{N,\beta}(x)|^{\alpha}]}dx.\]
	In addition, we introduce processes $X_{N,M,\beta}$, similar to $X_M$, but relating instead to $\log(|D_{N,\beta}(x)|)$. As the eigenvalues of $A_{N,\beta}$ may live outside of $(-1,1)$, we first must smoothly extend $T_k$ to a compactly-supported function. In particular, let us fix some $\epsilon>0$ and choose, for each $k$, a function $\tilde{T}_k(x)$ which is smooth and of compact support, and such that $\tilde{T}_k(x)=T_k(x)$ on $(-1-\epsilon,1+\epsilon)$. We define
	\[X_{N,M,\beta}(x)=-\sum_{k=1}^{M}\frac{2}{k}\Tr(\tilde{T}_k(A_{N,\beta}))T_k(x),\;\;\;\;\mu_{N,M,\beta,\alpha}(dx)=\frac{e^{\alpha X_{N,M,\beta}(x)}}{\E e^{\alpha X_{N,M,\beta}(x)}}dx.\]
	
	For a measure $\mu$ on $(-1,1)$, and a function $f$, let us denote $\mu(f):=\int_{-1}^1 f(x)\mu(dx)$. We note that by a standard argument (see Section 4 of \cite{kallenberg}), there is no loss of generality to replace the weak convergence in Theorem \ref{GMCtheorem} with vague convergence. That is, it suffices to show for any fixed continuous compactly-supported function $\varphi:(-1,1)\to
	[0,\infty)$, and for $\alpha<\sq{\beta/2}$, that we have that
	\[\mu_{N,\beta,\alpha}(\varphi)\To \mu_{\alpha\sq{2/\beta}}(\varphi)\label{eqn-to-show}\]
	in law as $N\to \infty$ for $\beta=1,4$ (where $N$ only runs over even numbers if $\beta=1$).
	
	For $\beta=1,2,4$ and $\alpha>0$ and $M$ fixed we have that \[\mu_{N,M,\beta,\alpha}(\varphi)\To \mu_{M,\alpha \sq{2/\beta}}(\varphi)\]
	in law as $N\to \infty$. This result is a consequence of the classical asymptotics for linear statistics given in \cite{Johansson}. In particular, this follows immediately by repeating the proof of the $\beta=2$ case given in Section 6 of \cite{GUEGMC}, as the results of \cite{Johansson} hold for all $\beta=1,2,4$. Noting the inequality,
	\[\frac{1}{4}\E(\mu_{N,\beta,\alpha}(\varphi)-\mu_{\alpha\sq{2/\beta}}(\varphi))^2\le\E(\mu_{N,\beta,\alpha}(\varphi)-\mu_{N,M,\beta,\alpha}(\varphi))^2+\]\[\E(\mu_{N,M,\beta,\alpha}(\varphi)-\mu_{M,\alpha\sq{2/\beta}}(\varphi))^2+\E(\mu_{M,\alpha\sq{2/\beta}}(\varphi)-\mu_{\alpha\sq{2/\beta}}(\varphi))^2,\]
	we see that when taking the limit $N\to \infty$ and then the limit $M\to \infty$, the second and third terms on the right-hand side vanish, so to demonstrate (\ref{eqn-to-show}) it suffices to show that for $\beta=1,4$, we have that
	\[\lim_{M\to\infty}\lim_{N\to\infty}\E(\mu_{N,\beta,\alpha}(\varphi)-\mu_{N,M,\beta,\alpha}(\varphi))^2=0, \label{MN-limit}\]
	where again $N$ only ranges over even numbers if $\beta=1$.
	
	One should note that in the case of $\beta=2$, the left-hand side of (\ref{MN-limit}) is shown to vanish in \cite{GUEGMC}. The key step will then be to define an additional pair of measures
	\[\nu_{N,\alpha}(dx)=\frac{|D_{2N,1}(x)|^{\alpha}|D_{N,4}(x)|^{2\alpha}}{\E \big[|D_{2N,1}(x)|^{\alpha}\big]\E \big[|D_{N,4}(x)|^{2\alpha}\big]}dx,\;\;\;\nu_{N,M,\alpha}(dx)=\frac{e^{\alpha X_{2N,M,1}(x)}e^{2\alpha X_{N,M,4}(x)}}{\E e^{\alpha X_{2N,M,1}(x)}\E e^{2\alpha X_{N,M,4}(x)}}dx,\]
	where $A_{2N,1}$ and $A_{N,4}$ are sampled independently. We observe that $\E_{A_{2N,1}}[\nu_{N,\alpha}(\varphi)]=\mu_{N,4,2\alpha}(\varphi)$ an $\E_{A_{N,4}}[\nu_{N,\alpha}(\varphi)]=\mu_{2N,1,\alpha}(\varphi)$, where $\E_{A_{2N,1}}$ and $\E_{A_{N,4}}$ denote expectation with respect to $A_{2N,1}$ and $A_{N,4}$, respectively. A similar result holds for $\nu_{N,M,\alpha}$. A key observation is that Jensen's inequality implies that
	\[\E(\mu_{N,4,\alpha}(\varphi)-\mu_{N,M,4,\alpha}(\varphi))^2=\E(\E_{A_{2N,1}}(\nu_{N,\alpha}(\varphi)-\nu_{N,M,\alpha}(\varphi)))^2\le \E(\nu_{N,\alpha}(\varphi)-\nu_{N,M,\alpha}(\varphi))^2.\]
	With a similar observation at $\beta=1$, this implies that to show (\ref{MN-limit}), it suffices to show that for $\alpha<1/\sq{2}$
	\[\lim_{M\to\infty}\lim_{N\to\infty}\E(\nu_{N,\alpha}(\varphi)-\nu_{N,M,\alpha}(\varphi))^2=0.\label{Jensen-Limit}\]
	In addition, it will suffice to show that
	\[\lim_{M\to\infty}\lim_{N\to\infty}\E(\nu_{N,\alpha}(\varphi)-\nu_{N,M,\alpha}(\varphi))^2=
	\lim_{M\to\infty}\lim_{N\to\infty}\E(\mu_{2N,2,2\alpha}(\varphi)-\mu_{2N,M,2,2\alpha}(\varphi))^2\label{GMC-intro-final},\]
	as the latter is shown to vanish in Proposition 2.9 of \cite{GUEGMC}. Expanding both sides of (\ref{GMC-intro-final}), we will see that all the integrands coincide pointwise as $N\to\infty$ by Theorem \ref{main-theorem}, reducing the remaining work to showing that the error term is suitably uniform. Details of this are given in Section \ref{GMC section}.

	\subsection{Outline of the proof of Theorem \ref{main-theorem}}\label{outline2}
	We now outline the method of proof of Theorem \ref{main-theorem}. Let us choose some $m\ge 1$ and $(\cal{W},\{\lam_i\}_i,\{\alpha_i\}_i)$ satisfying the conditions of Theorem \ref{main-theorem}. We will denote the polynomial coinciding with $\cal{W}$ on $(-1,1)$ as $\cal{W}_0$, and denote the difference as $\cal{E}=\cal{W}-\cal{W}_0$. We will omit the data $(N,\cal{W},\{\lam_i\}_i,\{\alpha_i\}_i)$ from the notation when such a choice is clear from the context. We will reduce the proof of Theorem \ref{main-theorem} to the computation of the asymptotics of the determinant of an $N$-by-$N$ matrix $\Delta_N:=\Delta_N(\cal{W},\{\lam_i\}_i,\{\alpha_i\}_i)$, defined in (\ref{delta_def}). More specifically, we show that
	\[\frac{\E[e^{\Tr(\cal{W}(A_{2N,1}))}\prod_{i=1}^{m}|D_{2N,1}(\lam_i)|^{\alpha_i}] \E[e^{\Tr(2\cal{W}(A_{N,4}))}\prod_{i=1}^{m}|D_{N,4}(\lam_i)|^{2\alpha_i}]}{\E[e^{\Tr(2\cal{W}(A_{2N,2}))}\prod_{i=1}^{m}|D_{2N,2}(\lam_i)|^{2\alpha_i}]}=\sq{\det(\Delta_N)}\label{Ratio-Form}.\]
	This reduces the proof of Theorem \ref{main-theorem} to the computation of $\det(\Delta_N)$.
	
	This representation was essentially observed as Remark 2.4 of \cite{Stoj} and will essentially follow (up to normalization constants) from the representations presented in \cite{beke}. This matrix $\Delta_N$ originally appeared in \cite{WidomRelation}, and was used to express the correlation functions of Orthogonal and Symplectic ensembles, and is shown there to be quite sparse. In particular in the case $\cal{W}=0$, we have that $\Delta_N$ coincides with the identity outside of an $(m+1)$-by-$(m+1)$ block. In the case of $\cal{W}\neq 0$ we will show that $\Delta_N$ is the identity outside a $(m+\max(\deg(\cal{W}_0)-1,1))$-by-$(m+\max(\deg(\cal{W}_0)-1,1))$ block, up to an error term which we show is exponentially small in Section \ref{Exponential Error Section}.
	
	In either case, the entries of this distinguished block may be expressed explicitly in terms of the orthogonal polynomials with respect to the measure \[w_N(x):=w(x;N,\cal{W},\{\lam_i\}_i,\{\alpha_i\}_i)=e^{-2Nx^2}e^{2\cal{W}(x)}\prod_{i=1}^{m}|x-\lam_i|^{2\alpha_i}.\label{w-def}\]
	The asymptotics of the orthogonal polynomials with respect to this measure is known in the case of $\cal{W}=0$ by \cite{Krav}, and $\cal{W}\neq 0$ by \cite{GUEGMC}, and are recalled in detail in Section \ref{Asymptotics-Section}. With this information, we are able to compute the asymptotics of this block matrix to sufficient order to compute the first-order behavior of $\det(\Delta_N)$. This is done in Section \ref{Integral Section}.
	
	We remark that computations of $\Delta_N$ for other measures form a key step in the proof of universality for the correlations functions of various symmetric and symplectic ensembles \cite{Universality,UniversalityLaguerre,UniversalityEdge,Stoj}.
	
	\subsection{Structure of the paper}
	
	In Section \ref{preliminary-section}, we prove (\ref{Ratio-Form}) and reduce the asymptotic computation of $\det(\Delta_N)$ to Propositions \ref{Exponential-Error-Prop} and \ref{W-Integral-Proposition}. In addition, we provide a modification of Theorem \ref{main-theorem}, Proposition \ref{Main-prop-Merging-2}, which applies in the case of (subcritical) merging singularities and reduces its proof to Proposition \ref{merging-W-Integral-Proposition}. In Section \ref{GMC section} we complete the proof of Theorem \ref{GMCtheorem} which by the argument given above is reduced to demonstrating (\ref{GMC-intro-final}). The bulk of this statement will follow from Theorem \ref{main-theorem} and Proposition \ref{Main-prop-Merging-2}. In Section \ref{Asymptotics-Section}, we provide asymptotics for the orthogonal polynomials with respect to $w_N$, while in Section \ref{Exponential Error Section} we employ these asymptotics to prove Proposition \ref{Exponential-Error-Prop}, which effectively allows us to neglect all but finitely many entries of $\det(\Delta_N)$. In Section \ref{Integral Section}, we prove Proposition \ref{W-Integral-Proposition}, which consists of computing the highest order behavior of certain integrals of orthogonal polynomials, by employing classical methods of steepest descent to the asymptotics of Section \ref{Asymptotics-Section}. Furthermore, in Section \ref{merging section}, we will provide a proof of Proposition \ref{merging-W-Integral-Proposition}. The methods of this section will be similar to those of Section \ref{Integral Section}, and indeed the cases of Proposition \ref{W-Integral-Proposition} and Proposition \ref{merging-W-Integral-Proposition} have significant overlap.
	
	\subsection{Acknowledgments} The author was partially supported by the grants NSF DMS-1653552, NSF DMS-1502632, and NSF DMS-2202720 while completing this work. The author would like to thank their advisor Antonio Auffinger for several discussions related to this project and for many comments on earlier versions of this manuscript. In addition, the author would like to thank an anonymous reviewer for finding numerous typos and errors earlier drafts.
	
	\section{Preliminary Results\label{preliminary-section}}
	In this section, we will prove (\ref{Ratio-Form}). In addition, we will analyze the terms of $\Delta_N$ to reduce the computation of $\det(\Delta_N)$ to a sequence of integrals and quantities given in terms of the orthogonal polynomials of $w_N$.
	
	We will assume henceforth that $N>m+\deg(\cal{W}_0)$. We define the unnormalized expected characteristic polynomial as
	\[\Phi_{N,\beta}(\cal{W},\{\lam_i\}_i,\{\alpha_i\}_i)=\frac{1}{N!}\int_{\R^N}|\Delta_N(x)|^{\beta}\prod_{i=1}^{N}(e^{-N\beta x_i^2+\cal{W}(x_i)}\prod_{j=1}^{m}|x_i-\lam_j|^{\alpha_j})dx^N,\]
	where here $\{\lam_i\}_i=(\lam_1,\cdots \lam_m)$, $\{\alpha_i\}_i=(\alpha_1,\cdots \alpha_m)$, and
	$\Delta_N(x)=\prod_{1\le i<j\le N}(x_i-x_j)$
	denotes the Vandermonde determinant. We note that (see Chapter 1 of \cite{LogGas})
	\[\Phi_{N,\beta}(\cal{W},\{\lam_i\}_i,\{\alpha_i\}_i)=Z_{N,\beta}\E e^{\Tr(\cal{W}(A_{N,\beta}))}\prod_{i=1}^{m}|\det(A_{N,\beta}-\lam_i I)|^{\alpha_i}.\]
	We will need the following relation.
	\begin{lem}
		\label{Normalization-Lemma}
		For all $N\ge 1$, we have that $Z_{2N,1}Z_{N,4}=2^{2N}Z_{2N,2}$.
	\end{lem}
	\begin{proof}
		The values of $Z_{N,\beta}$ are known exactly as
		\[Z_{N,\beta}=(2N)^{-\beta\frac{N(N-1)}{4}-N/2}(N!)^{-1}\beta^{-\frac{N}{2}-\beta\frac{N(N-1)}{4}}(2\pi)^{N/2}\prod_{j=1}^{N}\frac{\Gamma(1+j\beta/2)}{\Gamma(1+\beta/2)}.\]
		This
		follows from equation 1.163 of \cite{LogGas}, up to rescaling. In view of this, we may rewrite $Z_{2N,1}Z_{N,4}/Z_{2N,2}$ as
		\[\frac{(4N)^{\frac{-2N(2N-1)}{4}-N}(2N)^{-N(N-1)-N/2}}{(4N)^{-\frac{2N(2N-1)}{2}-N}}\frac{4^{-N/2-N(N-1)}\Gamma(2)^{2N}}{2^{-N-N(2N-1)}\Gamma(3)^{N}}\frac{(2\pi)^{N/2}}{\Gamma(3/2)^{2N}N!}\frac{\prod_{j=1}^{2N}\Gamma(1+j/2)\prod_{j=1}^{N}\Gamma(1+2j)}{\prod_{j=1}^{2N}\Gamma(1+j)}.\label{eqn:ignore-1235}\]
		We observe that
		\[\frac{(4N)^{\frac{-2N(2N-1)}{4}-N}(2N)^{-N(N-1)-N/2}}{(4N)^{-\frac{2N(2N-1)}{2}-N}}=2^{N^2-\frac{N}{2}},\quad \frac{4^{-N/2-N(N-1)}\Gamma(2)^{2N}}{2^{-N-N(2N-1)}\Gamma(3)^{N}}=1.\]
		We may also obtain that
		\[\frac{\prod_{j=1}^{2N}\Gamma(1+j/2)\prod_{j=1}^{N}\Gamma(1+2j)}{\prod_{j=1}^{2N}\Gamma(1+j)}=\prod_{j=1}^{N}\frac{\Gamma(j+1/2)\Gamma(1+j)}{\Gamma(2j)}=\]\[N!\prod_{j=1}^{N}\frac{\Gamma(j+1/2)\Gamma(j)}{\Gamma(2j)}=N!\pi^{N/2}\prod_{j=1}^{N}2^{1-2j}=N!\pi^{N/2}2^{-N^2},\]
		where in the second to last equality we have employed the Legendre duplication formula (see 6.1.18 of \cite{AS}) \[\Gamma(z)\Gamma(z+1/2)/\Gamma(2z)=2^{1-2z}\pi.\]
		In particular we may rewrite (\ref{eqn:ignore-1235}) as
		\[2^{-N/2}\pi^{N/2}\frac{(2\pi)^{N/2}}{\Gamma(3/2)^{2N}}=\frac{\pi^{N}}{\Gamma(3/2)^{2N}}=2^{2N},\]
		where in the last step we used that $\Gamma(3/2)=\frac{1}{2}\sq{\pi}$. As (\ref{eqn:ignore-1235}) coincides with $Z_{2N,1}Z_{N,4}/Z_{2N,2}$ we obtain the desired result.
	\end{proof}
	From this, we see that
	\[\frac{\E[e^{\Tr(\cal{W}(A_{2N,1}))}\prod_{i=1}^{m}|D_{2N,1}(\lam_i)|^{\alpha_i}] \E[e^{\Tr(2\cal{W}(A_{N,4}))}\prod_{i=1}^{m}|D_{N,4}(\lam_i)|^{2\alpha_i}]}{\E[e^{\Tr(2\cal{W}(A_{2N,2}))}\prod_{i=1}^{m}|D_{2N,2}(\lam_i)|^{2\alpha_i}]}=\]
	\[2^{2N}\frac{\Phi_{2N,1}(\cal{W},\{\lam_i\}_i,\{\alpha_i\}_i)\Phi_{N,4}(2\cal{W},\{\lam_i\}_i,\{2\alpha_i\}_i)}{\Phi_{2N,2}(2\cal{W},\{\lam_i\}_i,\{2\alpha_i\}_i)}.\label{int-21}\]
	We now recall a classical representation of these integrals. For simplicity, for the remainder of the section, we will fix $(\cal{W},\{\lam_i\}_i,\{\alpha_i\}_i)$ and denote, as above, \[w_N(x)=w(x;N,\cal{W},\{\lam_i\}_i,\{\alpha_i\}_i).\]
	We define a family of $N$-by-$N$ matrices defined for $1\le i,j\le N$ as
	\[[M_N^2(\cal{W},\{\lam_i\}_i,\{\alpha_i\}_i)]_{ij}=\int x^{i+j-2}w_N(x)dx,\]
	\[[M^4_N(\cal{W},\{\lam_i\}_i,\{\alpha_i\}_i)]_{ij}=\frac{1}{2}\int ((x^{i-1})'(x^{j-1})-(x^{i-1})(x^{j-1})')w_N(x)dx,\]
	\[[M^1_{N}(\cal{W},\{\lam_i\}_i,\{\alpha_i\}_i)]_{ij}=\frac{1}{2}\int \int x^{j-1}y^{i-1}\;\sign(x-y)\sq{w_N(x)w_N(y)}dxdy.\]
	
	We note that the latter two matrices are antisymmetric, while the first is symmetric. The relationship between these matrices and the various $\Phi_{N,\beta}$ is classical.
	\begin{lem}
		\label{Pfaffian-Lemma}
		We have for all $N\ge 1$, that
		\[\det(M^2_N(\cal{W},\{\lam_i\}_i,\{\alpha_i\}_i))=\Phi_{2,N}(2\cal{W},\{\lam_i\}_i,\{2\alpha_i\}_i),\]\[\pf(M^1_{2N}(\cal{W},\{\lam_i\}_i,\{\alpha_i\}_i))=2^{-N}\Phi_{1,2N}(\cal{W},\{\lam_i\}_i,\{\alpha_i\}_i),\]\[\pf(M^4_{2N}(\cal{W},\{\lam_i\}_i,\{\alpha_i\}_i))=2^{-N}\Phi_{4,N}(2\cal{W},\{\lam_i\}_i,\{2\alpha_i\}_i),\]
		where here $\pf$ denotes the Pfaffian.
	\end{lem}
	\begin{proof}
		This follows immediately from Propositions 5.2.1, 6.1.8, and 6.3.4 of \cite{LogGas}.
	\end{proof}
	From this we may further write (\ref{int-21}) as
	\[\frac{\pf(M^1_{2N}(\cal{W},\{\lam_i\}_i,\{\alpha_i\}_i))\pf(M^4_{2N}(\cal{W},\{\lam_i\}_i,\{\alpha_i\}_i))}{\det(M^2_{2N}(\cal{W},\{\lam_i\}_i,\{\alpha_i\}_i))}.\]
	To relate the various $M^\beta_N$ further we introduce operators $J_N^{\pm 1}:=J^{\pm 1}(N,\cal{W}, \{\lam_i\}_i, \{\alpha_i\}_i)$ by
	\[J_N=w_N(x)^{-1/2}\frac{d}{dx}w_N(x)^{1/2}=\frac{d}{dx}-2Nx+\sum_{i=1}^{m}\frac{\alpha_i}{x-\lam_i}+\cal{W}'(x),\]
	\[J_N^{-1}f(x)=w_N(x)^{-1/2}\frac{1}{2}\int \Sign(x-y)f(y)w_N(y)^{1/2}dy.\]
	We note that as long as $f$ is continuously differentiable and of sub-exponential growth, then
	\[(J_NJ_N^{-1}f)(x)=w_N(x)^{-1/2}\frac{d}{dx}\frac{1}{2}\int\sign(x-y)f(y)w_N(y)^{1/2}dy=f(x),\]
	and similarly $(J_NJ_N^{-1}f)(x)=f(x)$. As this will always be the case in this work, we will use these relations without mention. The utility of these operators is clear from the following classical relations.
	\begin{lem}
		\label{Anti-Symmetry}
		Let $f$ and $g$ be polynomials. Then
		\[\int (J_N^{-1}f)(x)g(x)w_N(x)dx=\frac{1}{2}\int\int g(x)f(y)\Sign(x-y)\sq{w_N(x)w_N(y)}dxdy,\]
		\[\int (J_Nf)(x)g(x)w_N(x)dx=\frac{1}{2}\int (f'(x)g(x)-f(x)g'(x))w_N(x)dx.\]
	\end{lem}
	We note that in view of this lemma, we may express, for even $N$,
	\[[M_{N}^1(\cal{W},\{\lam_i\}_i,\{\alpha_i\}_i)]_{ij}=\int (J_N^{-1}x^{i-1})x^{j-1}w_N(x)dx,\]
	\[[M_{N}^4(\cal{W},\{\lam_i\}_i,\{\alpha_i\}_i)]_{ij}=\int (J_Nx^{i-1})x^{j-1}w_N(x)dx.\]
	
	To further simplify this, it will be useful to recall the sequence of orthogonal polynomials related to $w_N$. Namely, let $p_i(x):=p_i(x;N,\cal{W},\{\lam_i\}_i,\{\alpha_i\}_i)$ denote the unique sequence of polynomials, with $p_i$ of degree-$i$, such that
	\[\int p_i(x)p_j(x)w_N(x)dx=\delta_{ij},\]
	and such that if we denote the leading coefficient of $p_i$ as $\kappa_i:=\kappa_i(N,\cal{W},\{\lam_i\}_i,\{\alpha_i\}_i)$, we have that $\kappa_i>0$. It is important to note that each $p_i$ is $N$-dependant. As we will keep $N$ fixed throughout the remainder of this section, we will omit $N$, as well $(\cal{W},\{\lam_i\}_i,\{\alpha_i\}_i)$, from the notation for the $p_i$ below.
	
	We introduce the Christoffel-Darboux kernel
	\[K_{N}(x,y)=\sum_{i=0}^{N-1}p_{i}(x)p_i(y),\]
	as well as the corresponding operator
	\[(\Pi_{N} f)(x)=\int K_{N}(x,y)f(y)w_N(y)dy.\]
	Let us denote the space of polynomials of degree less than $m$ as $\cal{P}_m$. We will denote the space of polynomials as $\cal{P}$. Both of these inherit the inner-product $(f,g)_{w,N}=\int f(x)g(x)w_N(x)dx$. With respect to this inner-product $\Pi_N$ is the orthogonal projection of $\cal{P}$ onto $\cal{P}_N$. We will also use the notation $\Pi_N^{\perp}=I-\Pi_N$. We introduce an $N \times N$ matrix given for $1\le i,j\le N$ as
	\[[\Delta_N]_{ij}=\int (J_N^{-1}\Pi_N J_Np_{i-1})(x)p_{j-1}(x)w_N(x)dx.
	\label{delta_def}\]
	We observe that the matrix $\Delta_N$ is simply the representation of the operator $\Pi_N J_N^{-1}\Pi_N J_N \Pi_N$ with respect basis of $\cal{P}_N$ given by the orthogonal polynomials. We will often refer to $\Pi_N J_N^{-1}\Pi_N J_N \Pi_N$ as $\Delta_N$ when no confusion may arise. We have the following formula.
	\begin{lem}
		\label{OP identity}
		For $N$ even, we have that
		\[\bigg[\frac{\pf(M^1_{N}(\cal{W},\{\lam_i\}_i,\{\alpha_i\}_i))\pf(M^4_{N}(\cal{W},\{\lam_i\}_i,\{\alpha_i\}_i))}{\det(M^2_{N}(\cal{W},\{\lam_i\}_i,\{\alpha_i\}_i))}\bigg]^2=\det(\Delta_{N}(\cal{W},\{\lam_i\}_i,\{\alpha_i\}_i)).\]
	\end{lem}
	\begin{proof}
		For simplicity, we will omit $(\cal{W},\{\lam_i\}_i,\{\alpha_i\}_i)$ from the notation for $M^\beta_N$ for this proof. By changing to the basis of $\cal{P}_N$ given by orthogonal polynomials, we see that $\det(M_N^2)=\prod_{i=0}^{N-1}\kappa_i^2$. In addition, we see that if we define for $1\le i,j\le N$,
		\[[M^{'1}_N]_{ij}=\int (J_N^{-1}p_{i-1})(x)p_{j-1}(x)w_N(x)dx,\qquad [M^{'4}_N]_{ij}=\int (J_Np_{i-1})(x)p_{j-1}(x)w_N(x)dx,\]
		then by again changing the basis, we have that
		\[\det(M^1_{2N})\det(M^4_N)=(\prod_{i=0}^{N-1}\kappa_i^2)^2\det(M^{'1}_N)\det(M^{'4}_N). \label{Misc-Pfaffian-Equation}\]
		On the other hand, we see that
		\[[M^{'1}_N(M^{'4}_N)^T]_{ij}=\sum_{k=0}^{N-1}(\int (J_N^{-1}p_{i-1})(x)p_k(x)w_N(x)dx)(\int (J_Np_{j-1})(y)p_k(y)w_N(y)dy)=\]
		\[\int (J_N^{-1}p_{i-1})(x)(\int K_N(x,y)(J_Np_{j-1})(y)w_N(y)dy)w_N(x)dx=\]
		\[\int (J_N^{-1}p_{i-1})(x)(\Pi_N J_Np_{j-1})(x)w_N(x)dx=-\int p_{i-1}(x)(J_N^{-1}\Pi_N J_Np_{j-1})(x)w_N(x)dx.\]
		In particular
		\[\bigg[\frac{\pf(M^1_{N})\pf(M^4_{N})}{\det(M^2_{N})}\bigg]^2=\det(M^{'1}_N)\det(M^{'4}_N)=(-1)^N\det(\Delta_N).\]
		As $N$ is even, this completes the proof.
	\end{proof}
	These lemmas combined complete the verification of (\ref{Ratio-Form}). To proceed, we must reduce $\Delta_N$ to a more manageable form. In view of the identity,
	\[\Pi_N J_N^{-1}\Pi_N J_N\Pi_N=\Pi_N-\Pi_N J_N^{-1}[J_N,\Pi_N]\Pi_N\label{first-commutator-iden}\]
	we see that it suffices to understand the action of $J_N^{-1}[J_N,\Pi_N]$ on $\cal{P}_N$. To do so, we first compute the commutator of $[J_N,\Pi_N]$ on $\cal{P}_N$. For this, we recall the Christoffel-Darboux formula (see \cite{szego}):
	\[K_N(x,y)=\frac{\k_{N-1}}{\k_N}\left(\frac{p_N(x)p_{N-1}(y)-p_{N-1}(x)p_N(y)}{x-y}\right);\quad x\neq y, \label{chris-1}\]
	\[K_N(x,x)=\frac{\k_{N-1}}{\k_N}\left(p_N'(x)p_{N-1}(x)-p_{N-1}'(x)p_N(x)\right),\label{chris-2}\]
	as well as the classical three-term recurrence
	\[xp_{j}(x)=b_{j-1}p_{j-1}(x)+a_jp_j(x)+b_jp_{j+1}(x),\;\; b_j=\frac{\k_j}{\k_{j+1}},\;\; a_j=\beta_j-\beta_{j+1}, \label{ThreeTerm}\]
	where $\beta_j$ is defined so that $p_{j}(x)\k_j^{-1}=x^j+\beta_jx^{j-1}+\cdots$.
	Additionally, we denote by
	\[H_N(f)(y)=\pv\int\frac{f(x)}{x-y}w_N(x)dx\]
	the $w_N$-weighted Hilbert transform of $f$, where here $\pv$ denotes the Cauchy principal value integral taken at $y$. We define, for $1\le i\le m$,
	\[\ell_i(x)=e^{N\lam_i^2}N^{\alpha_i}\frac{\kappa_{N-1}}{\k_N}\frac{p_{N}(x)H_N(p_{N-1})(\lam_i)-H_N(p_{N})(\lam_i)p_{N-1}(x)}{x-\lam_i},
	\label{l-def}\]
	\[q_i(x)=e^{-N\lam_i^2}N^{-\alpha_i}K_N(x,\lam_i)\label{q-def}.\]
	The utility of these functions comes from the following decomposition.
	\begin{prop}
		\label{Commutator-Decom}
		For $f\in \cal{P}_N$, we have that
		\[[J_N, \Pi_N]f=-\sum_{k=1}^{m}e^{-N\lam_k^2}N^{-\alpha_k}\alpha_k\ell_k f(\lam_k)-2N\Pi_N^{\perp}(xf(x))+\Pi_N^{\perp}(\cal{W}'(x)f(x))=\]\[-\sum_{k=1}^{m}\alpha_k\ell_k \int q_k(x)f(x)w_N(x)dx-2N\Pi_N^{\perp}(xf(x))+\Pi_N^{\perp}(\cal{W}'(x)f(x)) .\label{Comm-Decom-Eqn}\]
	\end{prop}
	\begin{proof}[Proof of Proposition \ref{Commutator-Decom}]
		Observing that $\Pi_N f=f$ we see that
		\[[J_N,\Pi_N]f=\left[\frac{d}{dx},\Pi_N\right]f-2N\left[x,\Pi_N\right]f+\sum_{k}\alpha_k\left[\frac{1}{(x-\lam_k)},\Pi_N\right]f+[\cal{W}',\Pi_N]f=\]
		\[\Pi_N^{\perp}(f')-2N\Pi_N^{\perp}(xf)+\Pi_N^
		{\perp}(\cal{W}'f)+\sum_{k}\alpha_k\left[\frac{1}{(x-\lam_k)},\Pi_N\right]f.\]
		We note that $\Pi_N^{\perp}(f')=0$ so it suffices to evaluate $[(x-\lam_k)^{-1},\Pi_N]$. We compute that
		\[
		([\frac{1}{(x-\lam_k)},\Pi_N]f)(x)=\]\[\int(\frac{1}{(x-\lam_k)}-\frac{1}{(y-\lam_k)})\frac{k_{N-1}}{k_N}\frac{p_N(x)p_{N-1}(y)-p_N(y)p_{N-1}(x)}{(x-y)}f(y)w_N(y)dy=\]
		\[-\int\frac{k_{N-1}}{k_N}\frac{p_N(x)p_{N-1}(y)-p_N(y)p_{N-1}(x)}{(x-\lam_k)(y-\lam_k)}f(y)w_N(y)dy=\]
		\[-\int\frac{k_{N-1}}{k_N}\frac{p_N(x)p_{N-1}(y)-p_N(y)p_{N-1}(x)}{(x-\lam_k)}(\frac{f(y)-f(\lam_k)}{y-\lam_k}+\frac{f(\lam_k)}{y-\lam_k})w_N(y)dy=\]
		\[-\int\frac{k_{N-1}}{k_N}\frac{p_N(x)p_{N-1}(y)-p_N(y)p_{N-1}(x)}{(x-\lam_k)(y-\lam_k)}f(\lam_k)w_N(y)dy=-e^{-N\lam_k^2}N^{-\alpha_k}\ell_k(x)f(\lam_k),
		\]
		where in second to last equality we have used that fact that $(f(y)-f(\lam_k))/(y-\lam_k)$ is a polynomial of degree less than $N-1$, so that
		\[\int p_{N}(y)\frac{f(y)-f(\lam_k)}{y-\lam_k}w_N(y)dy=\int p_{N-1}(y)\frac{f(y)-f(\lam_k)}{y-\lam_k}w_N(y)dy=0.\]
		Together, these statements establish the first equality. To obtain the second equality, we observe that
		\[\int q_i(x)f(y)w_N(x)dx=e^{-N\lam_i^2}N^{-\alpha_i}f(\lam_i).\label{Evaulation-Lemma}\]
	\end{proof}
	We now introduce a modification of $\Delta_N$ that will be more computable. For $f\in \cal{P}_N$ let
	\[\Delta_N^0f=f+\Pi_N J_N^{-1}(\sum_{k=1}^{m}\alpha_k\ell_k\int q_k(y)f(y)w_N(y)dy+2N\Pi_N^{\perp}(xf(x))-\Pi_N^{\perp}(\cal{W}_0'(x)f(x))) \label{eqn:ignore-144}\]
	and let us denote the difference\[\Delta_N^1=\Delta_N-\Delta_N^0=-\Pi_N J_N^{-1} \Pi_N^{\perp}\cal{E}',\label{eqn:ignore-145}\]
	where $\cal{E}=\cal{W}-\cal{W}_0$, as before. It is clear that when $\cal{W}=0$, we have that $\Delta_N^0=\Delta_N$. In general, we shall show that this difference is negligible on the exponential scale.
	\begin{prop}
		Let $(\cal{W},\{\lam_i\}_i,\{\alpha_i\}_i)$ be as in Theorem \ref{main-theorem}. There exists a choice of $c>0$ such that
		\[\det(\Delta_N)=\det(\Delta_N^0)+O(e^{-Nc}).\]
		Moreover for any choice of compact subset of $\{\lam_i\}_{i=1}^{m}\in \{(\lam_1,\cdots \lam_m)\in (-1,1)^m:\lam_1>\cdots >\lam_m\}$, we may choose $c>0$ small enough that the error term is uniform over the chosen of compact subset.
		\label{Exponential-Error-Prop}
	\end{prop}
	This proposition will be proven in Section \ref{Exponential Error Section}. We now turn our attention to the variant $\Delta_N^0$. Let us write $d=\max(1,\deg(\cal{W}_0)-1)$, and let us define two subspaces of $\cal{P}_N$:\[\cal{P}_{N,2}=\{f\in \cal{P}_{N-d}:f(\lam_1)=\cdots =f(\lam_m)=0\},\]
	and $\cal{P}_{N,1}=\cal{P}_{N,2}^{\perp}$, the subspace of $\cal{P}_N$ which is orthogonal to $\cal{P}_{N,2}$ with respect to the inner-product product of $L^2(w_N)$. We first show a preliminary result about these subspaces.
	\begin{lem}
		\label{Added-lem}
		If $N>m+d+1$, then $\cal{B}_{N}=(q_{1},\cdots, q_{m},p_{N-d},\cdots, p_{N-1})$ is a basis of $\cal{P}_{N,1}$. Additionally for $f\in \cal{P}_{N,2}$, we have that $\Delta_N^0f=f$.
	\end{lem}
	\begin{proof}
		By the Euclidean algorithm for polynomials, we conclude that $\dim(\cal{P}_{N,2})=N-m-d$, so that $\dim(\cal{P}_{N,1})=m+d$. In view of (\ref{Evaulation-Lemma}) we see that $\cal{B}_{N}\subseteq \cal{P}_{N,1}$. Thus to show that $\cal{B}_{N}$ is a basis for $\cal{P}_{N,1}$, it suffices to show that it is linearly independent. For this, we observe that for $1\le i,j\le m$, and $1\le l\le d$
		\[\int q_i(x)\prod_{k=1,k\neq j}^{m}(x-\lam_k)w_N(x)dx=\delta_{ij}e^{-N\lam_i^2}N^{-\alpha_i}\prod_{k=1,k\neq j}^{m}(\lam_i-\lam_k),\]\[\int p_{N-l}(x)\prod_{k=1,k\neq j}^{m}(x-\lam_k)w_N(x)dx=0.\]
		From the first relation we observe that $(q_1,\cdots q_m)$ is linearly independent, and when combined with the second we see that \[\linSpan(q_1,\cdots q_m)\cap \linSpan(p_{N-d},\cdots, p_{N-1})=\{0\}.\] As the set $(p_{N-d},\cdots p_{N-1})$ is linearly independent by construction, these statements combined prove the linear independence of $\cal{B}_{N}$, which completes the proof of the first claim.
		
		To demonstrate the second claim, we observe that if $f\in \cal{P}_{N,2}$, then $(2Nx-\cal{W}_0'(x))f(x)\in \cal{P}_N$, so $\Pi_N^{\perp}(2Nx-\cal{W}_0'(x))f(x)=0$. We have as well that
		\[\sum_{k=1}^{m}\alpha_k\ell_k\int q_k(y)f(y)w_N(y)dy=\sum_{k=1}^{m}\alpha_k\ell_ke^{-N\lam_k^2}N^{-\alpha_k}f(\lam_k)=0,\]
		so indeed, we see that for $f\in \cal{P}_{N,2}$, only the identity term is nonzero, which completes the proof of the second claim.
	\end{proof}
	Let us now write the block-decomposition of the operator $\Delta_N^0$ on $\cal{P}_N$ with respect to $(\cal{P}_{N,1},\cal{P}_{N,2})$ as
	\[\Delta_N^0=\begin{bmatrix}[\Delta_N^0]_{11}& [\Delta_N^0]_{12}\\
		[\Delta_N^0]_{21}& [\Delta_N^0]_{22}\end{bmatrix},\]
	and similarly for other operators on $\cal{P}_N$. From the second claim of Lemma \ref{Added-lem}, we see that $[\Delta_N^0]_{12}=0$ and $[\Delta_N^0]_{22}=I$. In particular, we have that $\det(\Delta_N^0)=\det([\Delta_N^0]_{11})$.
	
	We now focus our attention on understanding $\det([\Delta_N^0]_{11})$. To proceed we observe that for $f\in \cal{P}_N$, we have that $(2Nx-\cal{W}'_0(x))f(x)\in \cal{P}_{N+d}$. Noting that for $g\in \cal{P}_{N+d}$,
	\[(\Pi_{N}^{\perp}g)(x)=\sum_{i=1}^{d}p_{N+i-1}(x)\int p_{N+i-1}(y)g(y)w_N(y)dy,\]
	we see that
	\[(\Pi_N^{\perp}(2Ny-\cal{W}'_0(y))f(y))(x)=\sum_{i=1}^{d}p_{N+i-1}(x)\int p_{N+i-1}(y)(2Ny-\cal{W}'_0(y))f(y)w_N(y)dy=\]
	\[\sum_{i=1}^{d}\sum_{k=0}^{N-1}p_{N+i-1}(x)\int (2Nz-\cal{W}'_0(z))p_{N+i-1}(z)p_{k}(z)w_N(z)dz\int p_k(y)f(y)w_N(y)dy=\]
	\[\sum_{i=1}^{d}\sum_{j=1}^{d}p_{N+i-1}(x)\int (2Nz-\cal{W}'_0(z))p_{N+i-1}(z)p_{N-j}(z)w_N(z)dz\int p_{N-j}(y)f(y)w_N(y)dy,\label{new-relation}\]
	where we have employed the relation $\Pi_Nf=f$ in the second equality, and in the third we have observed that $(2Nz-\cal{W}_0'(z))p_{k}(z)\in \cal{P}_{k+d+1}$ so that if $k<N-d$ and $i>0$
	\[\int (2Nx-\cal{W}_0'(x))p_{N+i-1}(x)p_{k}(x)w_N(x)dx=0.\]
	If we define $d$-by-$d$ matrices $M_{N,0},M_{N,1}$ by setting for $1\le i,j\le d$,
	\[[M_{N,0}]_{ij}=\int 2Nzp_{N+i-1}(z)p_{N-j}(z)w_N(z)dz,\;\;[M_{N,1}]_{ij}=-\int \cal{W}_0'(z)p_{N+i-1}(z)p_{N-j}(z)w_N(z)dz,\]
	then we may rewrite (\ref{new-relation}) as
	\[\Pi_N^{\perp}(2Ny-\cal{W}_0'(y)f(y))(x)=\sum_{i,j=1}^{d}[M_{N,0}+M_{N,1}]_{ij}p_{N+i-1}(x)\int p_{N-j}(y)f(y)w_N(y)dy.\label{new-relation-2}\]
	To understand the other terms in $\Delta_N^0$ we further define an $(d+m)$-by-$(d+m)$ matrix, $M_N$, by
	\[M_N=   \left(\begin{array}{@{}c|ccc@{}}
		M_{N,0}+M_{N,1} &  &  & \\\hline
		& \alpha_1 &  &  \\
		&  & \ddots & \\
		&  &  & \alpha_m
	\end{array}\right),\]
	and define an $(d+m)$-by-$(d+m)$ inner-product matrix by
	\[D_N=\left(\begin{array}{@{}c|c@{}}
		\int J_N^{-1}p_{N+j-1}(x)p_{N-i}(x)w_N(x)dx & \int J_N^{-1}\ell_l(x) p_{N-i}(x)w_N(x)dx \\\hline
		\int J_N^{-1}p_{N+j-1}(x)q_k(x) w_N(x)dx & \int J_N^{-1}\ell_{l}(x)q_k(x)w_N(x)dx  \\
	\end{array}\right)\]
	where $1\le i,j\le d$ and $1\le k,l\le m$. We now have the following identification.
	\begin{lem}
		\label{lem:ignore-207}
		For $N>d+m+1$, we have that $\det(\Delta^0_N)=\det(I+D_NM_N)$.
	\end{lem}
	\begin{proof}
		Let us denote $v_i=p_{N-i}$ for $1\le i\le d$ and $v_i=q_{i-d}$ for $d<i\le m+d$. By Lemma \ref{Added-lem} this is a basis for $\cal{P}_{N,1}$. We shall denote the matrix for the operator $[\Delta^0_N]_{11}$ with respect to this basis as $\Delta_{N,\cal{B}}^0$, so that by definition $\det([\Delta^0_N]_{11})=\det(\Delta_{N,\cal{B}}^0)$. We will also denote $u_i=p_{N+i-1}$ for $1\le i\le d$, and $u_i=\ell_{i-d}$ for $d<i\le m+d$. Now with (\ref{new-relation-2}), and recalling the notation $(f,g)_{w,N}=\int f(x)g(x)w_N(x)dx$, we see that for $f\in \cal{P}_N$
		\[\Delta_N^0f=f+\sum_{i,j=1}^{d+m}(\Pi_NJ_N^{-1}u_i)M_{ij}(v_j,f)_{w,N}.\label{eqn:ignore-155}\]
		Let us define an $(m+d)$-by-$(m+d)$ matrix $O$ by setting for $1\le i,j\le m+d$, $O_{ij}=(v_j,v_i)_{w,N}$. This matrix is invertible as $(v_i)_{i=1}^{d+m}$ is a linearly independent set. We also observe that
		\[(\Delta^0_Nv_j,v_i)_{w,N}=([\Delta^0_N]_{11}v_j,v_i)_{w,N}=\sum_{k=1}^{d+m}([\Delta^0_{N,\cal{B}}]_{kj}v_k,v_i)_{w,N}=[O\Delta_{N,\cal{B}}^0]_{ij}.\]
		On the other hand noting that $[D_{N}]_{ij}=(J^{-1}_Nu_j,v_i)_{w,N}$, we by (\ref{eqn:ignore-155}) that
		\[(\Delta^0_Nv_j,v_i)_{w,N}=O_{ij}+(D_NM_NO)_{ij},\]
		so that finally
		\[\Delta_{N,\cal{B}}^0=O^{-1}(I+D_NM_N)O.\]
		Taking the determinant of both sides completes the proof.
	\end{proof}
	We will later derive that (see Lemma \ref{Coefficient Lemma New}) for fixed $k\in \Z$ we have that $b_{N-k},a_{N-k}=O(1)$. Repeatedly applying (\ref{ThreeTerm}) we may thus conclude that for $1\le i,j\le m+d$, $[M_{N,1}]_{ij}=O(1)$ and $[M_{N,0}]_{ij}=2N b_{N-1}\delta_{ij}\delta_{j1}=O(N)\delta_{ij}\delta_{j1}$. Our understanding of $D_N$ comes from the following proposition.
	\begin{prop}
		\label{W-Integral-Proposition}
		For $N$ even, $1\le l,k\le m$, and $1\le i,j\le d$ we have that
		\[\int J_N^{-1}p_{N+j-1}(x)p_{N-i}(x)w_N(x)dx=O(N^{-1}),\;\;\;\int J_N^{-1}\ell_l(x) p_{N-i}(x)w_N(x)dx=O(N^{-1}\log(N)),\]\[
		\int J_N^{-1}p_{N+j-1}(x)q_k(x) w_N(x)dx=O(N^{-1/2}),\]\[
		\int J_N^{-1}\ell_{l}(x)q_k(x)w_N(x)dx=O(N^{-1/2}),\;\;\;\int J_N^{-1}p_{N}(x)p_{N-1}(x)w_N(x)dx=O(N^{-7/6}).\]
		Moreover each error term is uniform over compact subsets of $\{\lam_i\}_{i=1}^{m}\in \{(\lam_1,\cdots \lam_m)\in (-1,1)^m:\lam_1>\cdots >\lam_m\}$.
	\end{prop}
	The proof of this will be carried out in Section \ref{Integral Section}. We are now ready to give the proof of Theorem \ref{main-theorem}.
	
	\begin{proof}[Proof of Theorem \ref{main-theorem}]
		As noted above, it suffices to show that $\det(\Delta_N)=1+O(N^{-1/6})$. By Proposition \ref{Exponential-Error-Prop}, we see that it further suffices to show that $\det(\Delta^0_N)=\det(I+D_N M_N)=1+O(N^{-1/6})$. If we write $D_N$ as a block matrix with square blocks of size $1$ by $(d-1)$ by $m$, we see that Proposition \ref{W-Integral-Proposition}
		shows that
		\[D_N=\begin{bmatrix}
			O(N^{-7/6})& O(N^{-1})& O(N^{-1}\log(N))\\
			O(N^{-1})& O(N^{-1}) & O(N^{-1}\log(N))\\
			O(N^{-1/2}) & O(N^{-1/2}) & O(N^{-1/2})\\
		\end{bmatrix},\]
		where all errors are taken entrywise in each block. Similarly, the above observations for $M_N$ show that
		\[M_N=\begin{bmatrix}
			O(N)& O(1)& 0\\
			O(1)& O(1) & 0\\
			0 & 0 & O(1)\\
		\end{bmatrix}.\]
		In particular, we see that
		\[I+D_N M_N=\begin{bmatrix}
			1+O(N^{-1/6})& O(N^{-1})& O(N^{-1}\log(N))\\
			O(1)& I+O(N^{-1}) & O(N^{-1}\log(N))\\
			O(N^{1/2}) & O(N^{-1/2}) & I+O(N^{-1/2})\\
		\end{bmatrix}.\]
		If we denote the diagonal matrix $\Lambda_N=\diag(N^{-2/3},1,\cdots, 1)$ we see that
		\[\Lambda_N^{-1}(I+ D_N M_N)\Lambda_N=I+\Lambda_N^{-1} D_N M_N\Lambda_N=\begin{bmatrix}
			1+O(N^{-1/6})& O(N^{-1/3})& O(N^{-1/3}\log(N))\\
			O(N^{-2/3})& I+O(N^{-1}) & O(N^{-1}\log(N))\\
			O(N^{-1/6}) & O(N^{-1/2}) & I+O(N^{-1/2})\\
		\end{bmatrix}.\]
		Taking the determinant of this expression we see that \[\det(I+ D_N M_N)=\det(\Lambda_N^{-1}(I+ D_N M_N)\Lambda_N)=1+O(N^{-1/6}).\]
		Noting as well that all of the error estimates we have used are uniform in the sense of Theorem \ref{main-theorem}, we obtain the desired statement.
	\end{proof}
	
	\subsection{Case of Merging Singularities}
	
	As mentioned above, to prove Theorem \ref{GMCtheorem}, we will need to supplement Theorem \ref{main-theorem} with an analogous result in the case in which $m=2$, $\cal{W}=0$, and $\alpha_1=\alpha_2=\alpha$, but where $\lam_1,\lam_2$ are allowed to depend on $N$ and $\lam_1-\lam_2\to 0$. We will assume this set-up for the remainder of this subsection. We have the following analog of Theorem \ref{main-theorem}.
	
	\begin{prop}
		\label{Main-prop-Merging-2}
		Let us fix a choice $\alpha>0$, $\epsilon>0$, and $0<\gamma<1$. Then for any $\lam_1,\lam_2\in (-1+\epsilon,1-\epsilon)$ (possibly $N$-dependant) such that $(\lam_1-\lam_2)>N^{-\gamma}$, we have that
		\[\E[\prod_{i=1}^{2}|D_{2N,1}(\lam_i)|^{\alpha}] \E[\prod_{i=1}^{2}|D_{N,4}(\lam_i)|^{2\alpha}]=\E[\prod_{i=1}^{2}|D_{2N,2}(\lam_i)|^{2\alpha}](1+O(\max(N^{-1/6},N^{\gamma-1}\log(N)))).\label{lars}\]
		Moreover the error term is uniform over all available choices of $\lam_1,\lam_2\in(-1+\epsilon,1-\epsilon)$ with $(\lam_1-\lam_2)>N^{-\gamma}$.
	\end{prop}
	\begin{remark}
		We observe that as we have allowed the choice of $(\lam_1,\lam_2)$ to be $N$-dependant, the uniformity of the error claimed in Proposition \ref{Main-prop-Merging-2} follows immediately from the pointwise statement of (\ref{lars}).
	\end{remark}
	
	\begin{remark}
		We note that the case of $m=2$ and $(\lam_1-\lam_2)=0$ coincides with Theorem \ref{main-theorem} in the case of $m=1$ with $\lam_1=\lam_1$ and $\alpha\mapsto 2\alpha$. In particular, by Theorem \ref{main-theorem}, $|\det(\Delta_N)|=1+O(N^{-1/6})$. As such, we expect that the apparent divergence of the error-term in Proposition \ref{Main-prop-Merging-2} as $N(\lam_1-\lam_2)=O(1)$ is an artifact of our method of proof, which avoids using more refined asymptotics for the singularities in the merging case.
	\end{remark}
	
	All of the constructions of the above section work without modification. We also note that an analog of Proposition \ref{Exponential-Error-Prop} is not required as $\cal{W}=0$ in this case. We will show in Proposition \ref{Merging Error Bulk} below, that the conclusions of Lemma \ref{Coefficient Lemma New} still hold, so that $[M_{N,0}]_{11}=2N b_{N-1}=O(N)$ uniformly in the choice of $(\lam_1,\lam_2)$. We also have that $M_{N,1}=0$. Finally, we will need the following modification of Proposition \ref{W-Integral-Proposition}.
	\begin{prop}
		\label{merging-W-Integral-Proposition}
		With the assumptions of Proposition \ref{Main-prop-Merging-2} and $N$ even, we have for $1\le l,k\le 2$ that
		\[\int J_N^{-1}p_{N}(x)p_{N-1}(x)w_N(x)dx=O(N^{-1-\min(1-\gamma,1/6)}),\]\[\int J_N^{-1}\ell_l(x) p_{N-1}(x)w_N(x)dx=|\lam_1-\lam_2|^{\alpha}O(\max(N^{-1/2-(1-\gamma)},N^{-1}\log(N))),\]\[
		\int J_N^{-1}p_{N}(x)q_k(x) w_N(x)dx=|\lam_1-\lam_2|^{-\alpha}O(N^{-1/2}),\]\[
		\int J_N^{-1}\ell_{l}(x)q_k(x)w_N(x)dx=O(\max(N^{-1/2},N^{\gamma-1}\log(N))).\]
		Moreover all error terms are uniform over all available choices of $\lam_1,\lam_2\in(-1+\epsilon,1-\epsilon)$ with $(\lam_1-\lam_2)>N^{-\gamma}$.
	\end{prop}
	
	The proof of this result will be given in Section \ref{merging section}.
	
	\begin{proof}[Proof of Proposition \ref{Main-prop-Merging-2}]
		As above, it will suffice to show that \[\det(I+D_NM_N)=1+O(N^{-\min(1-\gamma,1/6)}\log(N)).\]
		If we let $\delta=\min(1-\gamma,1/6)$ and define the diagonal matrix $\Lam_N=\diag(N^{-1/2-\delta}|\lam_1-\lam_2|^{\alpha},1,1)$, then as above, Proposition \ref{merging-W-Integral-Proposition} shows that when we write the $3$-by-$3$ matrix $\Lam_N^{-1}D_NM_N\Lam_N$ in terms of square blocks of size $1$ and $2$ we have that
		\[I+\Lam_N^{-1}D_NM_N\Lam_N=\begin{bmatrix}
			1+O(N^{-\delta})& O(\max(N^{-(1-\gamma)+\delta},N^{-1/2+\delta}\log(N)))\\
			O(N^{-\delta})&I+O(\max(N^{-1/2},N^{\gamma-1}\log(N)))
		\end{bmatrix}.\label{eqn-ignore-339}\]
		We observe that $\max(N^{-(1-\gamma)+\delta},N^{-1/2+\delta}\log(N))\le 1$ and that $\max(N^{-1/2},N^{\gamma-1}\log(N))\le N^{-\delta}\log(N)$ so that (\ref{eqn-ignore-339}) implies the weaker bound
		\[I+\Lam_N^{-1}D_NM_N\Lam_N=
		\begin{bmatrix}
			1+O(N^{-\delta})& O(1)\\
			O(N^{-\delta})&I+O(N^{-\delta}\log(N))
		\end{bmatrix}.\]
		Taking a cofactor expansion in the first column we see that \[\det(I+\Lam_N^{-1}D_NM_N\Lam_N)=1+O(N^{-\delta}\log(N)),\] implies the desired claim.
	\end{proof}
	
	\section{Proof of Theorem \ref{GMCtheorem}\label{GMC section}}
	In this section, we will complete the proof of Theorem \ref{GMCtheorem} given in the introduction, assuming Theorem \ref{main-theorem} and Proposition \ref{Main-prop-Merging-2}. As explained in the introduction, it will suffice to show that the quantity
	\[\lim_{M\to\infty}\lim_{N\to\infty}\E(\nu_{N,\alpha}(\varphi)-\nu_{N,M,\alpha}(\varphi))^2\label{Jensen-Limit-2}\]
	coincides with the quantity
	\[\lim_{M\to\infty}\lim_{N\to\infty}\E(\mu_{2N,2,2\alpha}(\varphi)-\mu_{2N,M,2,2\alpha}(\varphi))^2\]
	which vanishes for $0<2\alpha<\sq{2}$ by Proposition 2.9 of \cite{GUEGMC}. We begin by noting that by classical asymptotics for smooth, compactly-supported linear statistics (see \cite{Johansson}), we have that
	\[\lim_{M\to \infty}\lim_{N\to \infty}\E(\nu_{N,M,\alpha}(\varphi)^2)=\lim_{M\to \infty}\lim_{N\to \infty}\E(\mu_{2N,M,2,2\alpha}(\varphi)^2).\]
	For the remaining terms, we have that
	\[\E(\nu_{N,\alpha}(\varphi)^2)=\int \int \varphi(x)\varphi(y)\frac{\E |D_{2N,1}(x)|^{\alpha}|D_{2N,1}(y)|^{\alpha}\E |D_{N,4}(x)|^{2\alpha}|D_{N,4}(y)|^{2\alpha}}{\E |D_{2N,1}(x)|^{\alpha}\E |D_{2N,1}(y)|^{\alpha}\E|D_{N,4}(x)|^{2\alpha}\E |D_{N,4}(y)|^{2\alpha}}dxdy\label{Integral-of-Second-Moment}.\]
	\[\E(\nu_{N,\alpha}(\varphi)\nu_{N,M,\alpha}(\varphi))=
	\int \int \varphi(x)\varphi(y)\frac{\E |D_{2N,1}(x)|^{\alpha}e^{\alpha X_{2N,M,1}(y)}\E |D_{N,4}(x)|^{2\alpha}e^{2\alpha X_{N,M,4}(y)}}{\E |D_{2N,1}(x)|^{\alpha}\E e^{\alpha X_{2N,M,1}(y)}\E |D_{N,4}(x)|^{2\alpha}\E e^{2\alpha X_{N,M,4}(y)}}dxdy.\label{integral-2-nu}\]
	Working first with the latter, we may apply Theorem \ref{main-theorem} to the integrand of (\ref{integral-2-nu}) to obtain that
	\[\E(\nu_{N,\alpha}(\varphi)\nu_{N,M,\alpha}(\varphi))=\int \int \varphi(x)\varphi(y)\frac{\E |D_{2N,2}(x)|^{2\alpha}e^{2\alpha X_{2N,M,2}(y)}}{\E |D_{2N,2}(x)|^{2\alpha}\E e^{2\alpha X_{2N,M,2}(y)}}(1+O(N^{-1/6}))dxdy=\]\[\int \int \varphi(x)\varphi(y)\frac{\E |D_{2N,2}(x)|^{2\alpha}e^{2\alpha X_{2N,M,2}(y)}}{\E |D_{2N,2}(x)|^{2\alpha}\E e^{2\alpha X_{2N,M,2}(y)}}dxdy(1+O(N^{-1/6}))=\]\[\E( \mu_{2N,2,2\alpha}(\varphi)\mu_{2N,M,2,2\alpha}(\varphi))(1+O(N^{-1/6})),\]
	where we have used the uniformity of the $O(N^{-1/6})$-term to obtain the second equality. Thus we are left to understand (\ref{Integral-of-Second-Moment}). This case is complicated by the non-uniformity of the error of Theorem \ref{main-theorem} around the diagonal. We will proceed similarly to the proof of Proposition 6.4 of \cite{GUEGMC}. For any $\epsilon>0$ and $\beta\in (2\alpha^2,1)$, we may decompose the integral of (\ref{Integral-of-Second-Moment}) into three integrals over $\{|x-y|>\epsilon\}$, $\{\epsilon>|x-y|>2N^{-\beta}\}$, and $\{2N^{-\beta}>|x-y|\}$. We denote these integrals by $A_{1,\epsilon}^N$, $A_{2,\epsilon}^N$ and $A_{3}^N$, respectively. Identically to the above case, we can show that for fixed $\epsilon>0$
	\[A_{1,\epsilon}^N=\int\int_{|x-y|>\epsilon} \varphi(x)\varphi(y)\frac{\E |D_{2N,2}(x)|^{2\alpha}|D_{2N,2}(y)|^{2\alpha}}{\E |D_{2N,2}(x)|^{2\alpha}\E |D_{2N,2}(y)|^{2\alpha}}dxdy(1+O(N^{-1/6})).\]
	In view of the proof of Proposition 6.4 of \cite{GUEGMC}, we have that
	\[\lim_{\epsilon\to 0}\lim_{N\to\infty} A_{1,\epsilon}^N=\lim_{N\to \infty}\E(\mu_{2N,2,2\alpha}(\varphi)^2),\]
	so that it suffices to show that $\lim_{N\to \infty} A_3^N=0$ and that $\lim_{\epsilon\to 0}\limsup_{N\to \infty} A_{2,\epsilon}^N=0$.
	To bound $A_{3}^N$, we observe by the Cauchy-Schwarz inequality we have that
	\[A_3^N\le \int \int_{2N^{-\beta}>|x-y|} \varphi(x)\varphi(y)\frac{\sq{\E |D_{2N,1}(x)|^{2\alpha}\E |D_{N,4}(x)|^{4\alpha}\E |D_{2N,1}(y)|^{2\alpha}\E |D_{N,4}(y)|^{4\alpha}}}{\E |D_{2N,1}(x)|^{\alpha}\E|D_{N,4}(x)|^{2\alpha}\E |D_{2N,1}(y)|^{\alpha}\E |D_{N,4}(y)|^{2\alpha}}dxdy.\]
	
	By repeatedly applying Theorem \ref{main-theorem} we may write the right-hand side as
	\[\int \int_{2N^{-\beta}>|x-y|} \varphi(x)\varphi(y)\frac{\sq{\E |D_{2N,2}(x)|^{4\alpha}\E |D_{2N,2}(y)|^{4\alpha}}}{\E |D_{2N,2}(x)|^{2\alpha}\E|D_{2N,2}(y)|^{2\alpha}}dxdy(1+O(N^{-1/6})).\]
	By applying the asymptotics of \cite{Krav} (see also the proof of Proposition 6.4 in \cite{GUEGMC}) we see that \[\frac{\sq{\E |D_{2N,2}(x)|^{4\alpha}\E |D_{2N,2}(y)|^{4\alpha}}}{\E |D_{2N,2}(x)|^{2\alpha}\E|D_{2N,2}(y)|^{2\alpha}}=O(N^{2\alpha^2-\beta}),\] uniformly on the support of $\varphi$. As $\beta>2\alpha^2$,
	we see that $\lim_{N\to \infty}A_3^N=0$.
	
	We finally show that $\lim_{\epsilon\to 0}\limsup_{N\to \infty} A_{2,\epsilon}^N=0$. By repeatedly applying Proposition \ref{Main-prop-Merging-2} as before, we have that
	\[A_{2,\epsilon}^N=\int \int_{2N^{-\beta}<|x-y|<\epsilon} \varphi(x)\varphi(y)\frac{\E |D_{2N,2}(x)|^{2\alpha}|D_{2N,2}(y)|^{2\alpha}}{\E |D_{2N,2}(x)|^{2\alpha}\E|D_{2N,2}(y)|^{2\alpha}}dxdy(1+O(N^{-\min(1-\beta,1/6)}\log(N))).\]
	\noindent
	It is shown in the last line of the proof of Proposition 6.4 of \cite{GUEGMC} that
	\[\lim_{\epsilon\to 0}\limsup_{N\to \infty}\int \int_{2N^{-\beta}<|x-y|<\epsilon} \varphi(x)\varphi(y)\frac{\E |D_{2N,2}(x)|^{2\alpha}|D_{2N,2}(y)|^{2\alpha}}{\E |D_{2N,2}(x)|^{2\alpha}\E|D_{2N,2}(y)|^{2\alpha}}dxdy=0.\]
	As $\beta<1$, together these imply that $\lim_{\epsilon\to 0}\limsup_{N\to \infty} A_{2,\epsilon}^N=0$, which completes the proof of Theorem \ref{GMCtheorem}.
	
	\section{Asymptotics of Orthogonal Polynomials and Related Quantities \label{Asymptotics-Section}}
	
	In this section, we state the leading order asymptotics for orthogonal polynomials with respect to the measure $w_N$, as well as those for some related quantities. Due to the complexity of the expressions of this asymptotics (especially in the regions around $\lam_i$), we will state our expressions as entries in a matrix product. These asymptotics were obtained in our setting by \cite{GUEGMC}, building upon the work of \cite{Krav}. In both cases, they are obtained via asymptotic analysis of an associated Riemann-Hilbert problem.
	
	We note that away from the points $\{\lam_i\}_{i=1}^{m}$, these asymptotics are closely related to the classical Plancherel-Rotach asymptotics for Hermite polynomials (see \cite{szego}). In particular, Propositions \ref{Asym p_N easy} and \ref{Asym p_N Airy} coincide with the corresponding classical expressions for Hermite polynomials except with a modification to the $O(1)$-term. Around the points $\{\lam_i\}_{i=1}^m$ though (where the asymptotics are described by Proposition \ref{Asym p_N Bessel}), they are described in terms of Bessel functions. Some related expressions appear in \cite{kui}, where they give a description of the asymptotics of the Christoffel-Darboux kernel around these points in terms of the Bessel kernel.
	
	The structure of this section will be as follows. We will begin this section by recalling some relevant background and notation to present Proposition \ref{Simple-Remainder Proposition}, which is essentially Theorem 4.37 of \cite{GUEGMC} restricted to $\R$. This expression gives (uniform) asymptotics for our desired orthogonal polynomials in terms of a variety of parametric. To make use of these asymptotics, we must recall and compute the behavior of these parametrices on $\R$. The result of this translation will be our main asymptotic results, which are given by Proposition \ref{Asym p_N easy}, \ref{Asym p_N Airy}, and \ref{Asym p_N Bessel}. After this, we will give asymptotics for $\ell_i$ and $q_i$ in Proposition \ref{l and q asym}.
	
	Finally, we will provide the modifications of these results needed in the merging case in Subsection \ref{subsection-merging}. These results will essentially follow from the results of \cite{Merging,RHmerging}. The main results are Propositions \ref{Merging Error Bulk} and \ref{Merging Error Bessel}. As we are only interested in the subcritical merging regime (i.e., $N(\lam_1-\lam_2)\to \infty$) our asymptotic expressions will essentially coincide with those in the non-merging case, with the only significant difference being that the bounds on the error terms are weaker, and the domains in which they are valid are smaller.
	
	We will fix for this section a choice of $(\cal{W},\{\lam_i\}_i,\{\alpha_i\}_i)$ as in Theorem \ref{main-theorem}. We begin by recalling an analytic matrix-valued function on $\C-\R$, $Y_N(z):=Y(z;N,\cal{W},\{\lam_i\}_i,\{\alpha_i\}_i)$, defined as
	\[Y_N(z)=\begin{bmatrix}\kappa_N^{-1}p_{N}(z)&\k_N^{-1}(2\pi i)^{-1}\int \frac{p_N(y)}{y-z}w_N(y)dy\\
		-2\pi i \k_{N-1}p_{N-1}(z)&-\k_{N-1}\int \frac{p_{N-1}(y)}{y-z}w_N(y)dy\end{bmatrix},\]
	where $w_N$ is defined in (\ref{w-def}). This function occurs as the solution to the Riemann-Hilbert problem defined by the measure $w_N$, in the sense of \cite{FIK} (see Proposition 3.4 of \cite{GUEGMC}). The asymptotics of $Y_N$ are then computed by using the method of nonlinear steepest descent, pioneered in \cite{DeiftZhou}. Our primary purpose is to specialize these asymptotics to $\R$. To ease the reader in translation of the results of \cite{GUEGMC}, we note that we are in their case of $t=s=1$ and their $(k,V,\cal{T},(x_i)_{i=1}^{k},(\beta_i)_{i=1}^k)$ coincides with our $(m,2x^2,2\cal{W},(\lam_{m-i})_{i=1}^{m},(2\alpha_i)_{i=1}^{m})$. It is well known in the quadratic case that $\ell_V=\ell$ where $\ell=-1-2\log(2)$, as can easily be verified from (2.3) and (2.4) in \cite{GUEGMC}, and that $d_1(\lam)=2/\pi$.
	
	When $x\in\R$ we will use the notation $f_+(x)=\lim_{z\to x}f(z)$ where the limit is taken along sequences in $\H$, within the domain of $f$, not tangential to $\R$. We are interested in understanding the asymptotics of $(Y_N)_+$. The asymptotics of $Y_N$ are stated in terms of a related matrix-valued function $S_N$ (defined in (4.5) and (4.11) of \cite{GUEGMC}). To understand the relation between $S_N$ and $Y_N$, we recall the $3$rd and $1$st Pauli matrices defined as
	\[\sigma=\sigma_3=\begin{bmatrix}1&0\\
		0&-1\end{bmatrix},\;\;\;\sigma_1=\begin{bmatrix}0&1\\1&0\end{bmatrix}.\]
	Then for $z\in \H$, such that $\Re(z)\in \R\setminus \{\lam_1,\cdots \lam_m\}$, and $\Im(z)$ is sufficiently small (depending only on $\Re(z)$), we have that
	\[Y_N(z)=e^{N\frac{\ell}{2}\sigma}S_N(z)e^{N(g_1(z)-\ell/2)\sigma}\;\;\text{for}\;\; |\Re(z)|\ge 1,\label{S-def-1}\]
	\[Y_N(z)=e^{N\frac{\ell}{2}\sigma}S_N(z)\begin{bmatrix} 1&0\\ f_1(z)^{-1}e^{-Nh_1(z)}&1\end{bmatrix}e^{N(g_1(z)-\ell/2)\sigma}\;\;\text{for}\;\; |\Re(z)|<1,\label{S-def-2}\]
	where $g_1(z)$, $h_1(z)$ and $f_1(z)$ are defined in (4.3), (4.9) and (4.13) of \cite{GUEGMC}. These follow from Definitions 4.1 and 4.9 of \cite{GUEGMC}. We will not need the exact form of these functions, though we will need their behavior of $\R$. For this, we define the functions \[s(x)=\begin{cases}2\int_x^1\sq{1-y^2}dy;\;|x|< 1\\
		2\int_{1}^{|x|}\sq{y^2-1}dy;\;|x|\ge 1\end{cases},\label{fun-s-def}\;\;\;\omega(x)=\prod_{i=1}^{m}|x-\lam_i|^{2\alpha_i}.\]
	We observe as well that
	\[s(x)=\begin{cases}|x|\sq{x^2-1}-\arcosh(|x|);\;\;|x|\ge 1\\
		-x\sq{1-x^2}+\arccos(x);\;\;|x|<1 \end{cases}.\]
	We may now state the behavior of the above functions on the real line.
	\begin{lem}
		\label{misc-function-lemma}
		We have that for $x\in \R$
		\[(g_1)_+(x)=x^2+\ell/2+(h_1)_+(x)/2,\;\;\; (h_1)_+(x)=\begin{cases}-2s(x);\;\;x\ge 1\\ 2\pi i-2s(x);\;\; x\le -1\\ i2s(x);\;\; |x|<1 \end{cases},\]
		and $(f_1)_+(x)=e^{2\cal{W}(x)}\omega(x)$.
	\end{lem}
	\begin{proof}
		The statements for $(h_1)_+$ and $(f_1)_+$ follow routinely from their definitions (i.e. (4.9) and (4.13) of \cite{GUEGMC}). For the results on $(g_1)_+$, combining the results of (4.6) and (4.7) of \cite{GUEGMC}, we see that for $x\in (-1,1)$,
		\[(g_1)_+(x)=x^2+\ell/2+h_+(x)/2=x^2+\ell/2+is(x).\]
		Finally, we see by direct computation that for $x\ge 1$,
		\[(g_1)_+(x)=\int_{-1}^1 \frac{2}{\pi}\sq{1-\lam^2}\log(x-\lam)d\lam=x^2+\ell/2-s(x),\]
		and $(g_1)_+(-x)=(g_1)_+(x)+i\pi$.
	\end{proof}
	By applying the results of Lemma \ref{misc-function-lemma} to (\ref{S-def-1}-\ref{S-def-2}) we see that for $x\in (-1,1)\setminus\{\lam_1,\cdots, \lam_m\}$
	\[(Y_N)_+(x)=e^{N\frac{\ell}{2}\sigma}(S_N)_+(x)\begin{bmatrix}1 &0\\
		e^{-2Nis(x)-2\cal{W}(x)}\omega(x)^{-1}&1\end{bmatrix}e^{Nis(x)\sigma}e^{Nx^2\sigma},\label{paramatrix-Y-S-(-1,1)}\]
	and for $|x|\ge 1$ that
	\[(Y_N)_+(x)=e^{N\frac{\ell}{2}\sigma}(S_N)_+(x)e^{-Ns(x)\sigma}e^{Nx^2\sigma}(-1)^{NI(x<0)\sigma}.\label{paramatrix-Y-S-outside}\]
	
	We are primarily interested in the first column of $Y_N$, so noting that $w_N(x)=\omega(x)e^{-2Nx^2+2\cal{W}(x)}$, we see that for $x\in (-1,1)\setminus\{\lam_1,\cdots \lam_m\}$
	\[e^{-N\frac{\ell}{2}\sigma}\begin{bmatrix} [(Y_N)_+]_{11}(x)\\ [(Y_N)_+]_{21}(x)\end{bmatrix}=(S_N)_+(x)\begin{bmatrix}e^{Nis(x)}\\ e^{-Nis(x)-2\cal{W}(x)}\omega(x)^{-1}
	\end{bmatrix}e^{Nx^2}=\]\[(S_N)_+(x)e^{Nis(x)\sigma}e^{\cal{W}(x)\sigma}\omega(x)^{\sigma/2}\begin{bmatrix}1\\ 1\end{bmatrix} w_N(x)^{-1/2},\label{first-col-bulk}\]
	and that for $|x|\ge 1$
	\[e^{-N\frac{\ell}{2}\sigma}\begin{bmatrix} [(Y_N)_+]_{11}(x)\\ [(Y_N)_+]_{21}(x)\end{bmatrix}=((S_N)_+(x)e_1)e^{N(x^2-s(x))}(-1)^{NI(x<0)}.\label{first-col-outside}\]
	
	Now Theorem 4.37 of \cite{GUEGMC} implies the following asymptotic result, which shows that up to a small error term $R_N$, $S_N$ may be described in terms of various explicit parametrices, whose behavior on $\R$ we will specify below. Here and elsewhere, we will use the notation $\lam_0=1$ and $\lam_{m+1}=-1$.
	\begin{prop}
		\label{Simple-Remainder Proposition} There exists $\delta_0>0$, such that for $0<\delta<\delta_0$, we may write
		\[(S_N)_+(x)=(I+R_{N}(x))(P_N)_+(x),\]
		for some matrix-valued function $R_{N}(x):=R(x;\delta,N,\cal{W},\{\lam_i\}_i,\{\alpha\}_i)$, and  \[(P_N)_+(x)=\begin{cases}(P_{\lam_i})_+(x);\;x\in (\lam_i-\delta,\lam_i+\delta)\\
			(P_{\pm 1})_+(x);\;x\in (\pm 1-\delta,\pm 1+\delta)\\
			(P_{\infty})_+(x);\;\text{otherwise}\end{cases},\]
		where here $P_{\infty}$,$P_{\lam_i}$, and $P_{\pm 1}$ are certain matrix-valued functions defined in Definition 4.12, 4.21, and 4.28 of \cite{GUEGMC}, respectively. Then, for any $l\ge 0$, we have that $R^{(l)}_N(x)=O(N^{-1})$ uniformly for $x\in\R\setminus \bigcup_{i=0}^{m+1}\{\lam_i-\delta,\lam_i+\delta\}$. Moreover for any choice of compact subset of $\{\lam_i\}_{i=1}^{m}\in \{(\lam_1,\cdots \lam_m)\in (-1,1)^m:\lam_1>\cdots >\lam_m\}$, we may choose $\delta>0$ small enough that $R^{(l)}_N(x)=O(N^{-1})$ uniformly
		in both $x\in \R\setminus \cup_{i=0}^{m+1}\{\lam_i-\delta,\lam_i+\delta\}$ and $\{\lam_i\}_{i=1}^{m}$ in our chosen compact set.
	\end{prop}
	\begin{remark}
		In the statement of Theorem 4.37 in \cite{GUEGMC}, the error bounds are only stated for $R_N$ and $R'_N$. On the other hand, the method of proof used to establish the error bounds on $R'_N$ extends trivially to higher derivatives.
	\end{remark}
	We note as well that all of these parametrices are $N$-dependant, though as we will not use them in our final asymptotic expressions, we omit this from the notation. We will compute the values of $P_{\infty}$, $P_{\pm 1}$, and $P_{\lam_i}$ on $\R$ below, though before beginning this computation we explain how to obtain results for the full range of polynomials $p_{N-k}$ for $k\in \Z$ fixed from Proposition \ref{Simple-Remainder Proposition} instead of only $(p_{N},p_{N-1})$. As the polynomials $p_{N-k}(x):=p_{N-k}(x;N,\cal{W},\{\lam_i\}_i,\{\alpha_i\}_i)$ themselves depend on the $N$-dependant measure $w_N$, we observe that the $(N-k)$-case of Proposition \ref{Simple-Remainder Proposition} does not precisely describe an asymptotic for $p_{N-k}$. On the other hand, we observe that if we define $\eta_k=\eta_{N,k}=\sq{(N-k)/N}$, then for $N>k$, and $\cal{A}=\sum_{i=1}^{m}\alpha_i$, the polynomial
	\[\eta_k^{\cal{A}+1/2}p_{N-k}(\eta_kx;N,\cal{W},\{\lam_i\}_i,\{\alpha_i\}_i)\label{to-iden}\]
	coincides with the $(N-k)$-th orthogonal polynomial with respect to the rescaled measure
	\[w_{N,k}(x)=w_{N,k}(x;\cal{W},\{\lam_i\}_i,\{\alpha_i\}_i)=w_{N}(\eta x)\eta^{-2\cal{A}}=e^{2\cal{W}(\eta_k x)}e^{-2(N-k)x^2}\prod_{i=1}^{m}|x-\eta_k^{-1}\lam_i|^{2\alpha_i}.\]
	If we denote $\cal{W}_{N,k}(x)=\cal{W}(\eta_kx)$, then we see that
	\[w_{N,k}(x;\cal{W},\{\lam_i\}_i,\{\alpha_i\}_i)=w_{N-k}(x;\cal{W}_{N,k},\{\eta_k^{-1}\lam_i\}_i,\{\alpha_i\}_i).\]
	In particular, we may apply Proposition \ref{Simple-Remainder Proposition} to $w_{N,k}$ to obtain asymptotics for $p_{N-k}$ defined with respect to $w_N$, though with a number of scaling factors of $\eta_k$ appearing. For clarity of the notation, we will reserve the notations $P_{\lam_i}$,$P_{\infty}$,$P_{\pm 1}$ for the parametrices resulting from the choice of parameters $(N,\cal{W},\{\lam_i\}_i,\{\alpha_i\}_i)$ and will take care to explicitly specify when objects are considered which are defined with respect to the measure $w_{N,k}$.
	
	Lastly, to account for the factors of $\eta_k$ appearing in the following asymptotics, we will have repeated use for the following result, which follows from an application of Taylor's Theorem.
	\begin{lem}
		\label{sk expansion lem}
		For fixed $k\in \Z$, we define $s_{N,k}(x)=s(\eta_k^{-1}x)$. Then if one defined functions $r_{N}$ such that
		\[(N-k)s_{N,k}(x)=Ns(x)-k\arcos(x)+r_{N}(x) \;\;\text{for}\;\; |x|<1,\]
		\[(N-k)s_{N,k}(x)=Ns(x)+k\arcosh(|x|)+r_{N}(x) \;\; \text{for}\;\; |x|> 1,\]
		then for any $l\ge 0$, we have that $r_{N}^{(l)}(x)=O(N^{-1})$ uniformly over compact subsets of $x\in \R\setminus\{\pm 1\}$.
	\end{lem}
	
	We are now ready to begin stating our asymptotic expressions, beginning with those for some leading-order coefficients of $p_{N-k}$ specified in (\ref{ThreeTerm}).
	
	\begin{lem}
		\label{Coefficient Lemma New}
		For fixed $k\in \Z$, we have that
		\[\kappa_{N-k}^2=\pi^{-1}e^{-(N-k)\ell+k}D_{\infty}^{-2}(1+O(N^{-1})),\;\;\;\;\; \beta_{N+1-k}=O(1),\label{constant-asymptotics-new}\]
		where here $D_\infty=2^{-\cal{A}}\exp(\frac{1}{\pi}\int_{-1}^1\frac{\cal{W}(u)}{\sq{1-u^2}}du)$. Moreover the error terms are uniform over compact subsets of $\{\lam_i\}_{i=1}^{m}\in \{(\lam_1,\cdots \lam_m)\in (-1,1)^m:\lam_1>\cdots >\lam_m\}$.
	\end{lem}
	
	\begin{remark}
		\label{Remark-coefficents-new}
		We observe from Lemma \ref{Coefficient Lemma New} that
		\[D_{\infty}^{-\sigma}e^{-N\frac{\ell}{2}\sigma}\begin{bmatrix}\kappa_{N}^{-1}p_{N}(x)\\-2\pi i\kappa_{N-1}p_{N-1}(x)\end{bmatrix}=(I+O(N^{-1}))\begin{bmatrix}\pi^{1/2}p_{N}(x)\\
			-i\pi^{1/2}p_{N-1}(x)\end{bmatrix},\label{coefficent-new-eqn}\]
		where the error term in (\ref{coefficent-new-eqn}) is independent of $x$ and uniform in $\{\lam_i\}_{i=1}^{m}$ in the same sense as Lemma \ref{Coefficient Lemma New}.
	\end{remark}
	
	For notational clarity, we will delay the proof of Lemma \ref{Coefficient Lemma New} to after the statement of Proposition \ref{Asym p_N easy} below. Now we define functions
	\[D(x)=\begin{cases}\exp(i[-\cal{A}\arcos(x)+\frac{\sq{1-x^2}}{\pi}\pv \int_{-1}^{1}\frac{\cal{W}(u)}{\sq{1-u^2}}\frac{1}{x-u}du]);\;\;\;|x|<1\\
		\exp([-\cal{A}\arcosh(x)+\frac{\sq{x^2-1}}{\pi}\int_{-1}^{1}\frac{\cal{W}(u)}{\sq{1-u^2}}\frac{1}{x-u}du]);\;\;\;x> 1\\
		\exp([-\cal{A}\arcosh(-x)-\frac{\sq{x^2-1}}{\pi}\int_{-1}^{1}\frac{\cal{W}(u)}{\sq{1-u^2}}\frac{1}{x-u}du]);\;\;\;x<-1\end{cases}.\]
	We introduce a fixed matrix-valued function $A(x)$, defined for $|x|<1$ by
	\[A(x)=\frac{1}{2(1-x^2)^{1/4}}\begin{bmatrix}e^{-i\pi/4}\sq{1+x}+e^{i\pi/4}\sq{1-x}& i[e^{-i\pi/4}\sq{1+x}-e^{i\pi/4}\sq{1-x}]\\
		-i[e^{-i\pi/4}\sq{1+x}-e^{i\pi/4}\sq{1-x}]&e^{-i\pi/4}\sq{1+x}+e^{i\pi/4}\sq{1-x}\end{bmatrix} \label{B in Bulk},\]
	and on $|x|>1$ as
	\[A(x)=\frac{1}{2(x^2-1)^{1/4}}\begin{bmatrix}\sq{|1+x|}+\sq{|1-x|}&i[\sq{|1+x|}-\sq{|1-x|}]\\
		-i[\sq{|1+x|}-\sq{|1-x|}]&\sq{|1+x|}+\sq{|1-x|}\end{bmatrix} \label{B in non-Bulk}.\]
	We now define \[T_{\infty}(x)=A(x)D(x)^{-\sigma}.\label{Tinf-def}\] Lastly we define an additional function \[T_{N,k,\infty}(x)=T_{\infty}(\eta_k^{-1}x;\cal{W}_{N,k},\{\eta_k^{-1}\lam_i\}_i,\{\alpha_i\}_i).\]
	
	We may now state the asymptotics away from the points $\{\pm 1,\lam_1,\cdots \lam_m\}$.
	\begin{prop}
		\label{Asym p_N easy}
		For $\delta>0$ small enough, and any choice of $k\in \Z$, we may define $R_N(x)$ such that for any choice of $0\le i\le m$, and $x\in (\lam_{i+1}+\delta,\lam_{i}-\delta)$ we have that
		\[\begin{bmatrix}\pi^{1/2}p_{N-k}(x)\\
			-i\pi^{1/2}p_{N-k-1}(x)\end{bmatrix}w_N(x)^{1/2}=(I+R_N(x))T_{\infty}(x)e^{[-ik\arcos(x)-i\pi\sum_{k=1}^{i}\alpha_k]\sigma}\begin{bmatrix}e^{iNs(x)}\\ e^{-iNs(x)} \end{bmatrix} \label{Bulk Asymptotic},\]
		and for $|x|>1+\delta$ we have that
		\[\begin{bmatrix}\pi^{1/2}p_{N-k}(x)\\
			-i\pi^{1/2}p_{N-k-1}(x)\end{bmatrix}w_N(x)^{1/2}=[(I+R_N(x))T_{N,k,\infty}(x)e_1]e^{-Ns_{N,k}(x)}(-1)^{(N-k)I(x<0)}\label{Outside Asymptotic}.\]
		Then for any $l\ge 0$ and any choice of compact subset of $\{\lam_i\}_{i=1}^{m}\in \{(\lam_1,\cdots \lam_m)\in (-1,1)^m:\lam_1>\cdots >\lam_m\}$, we may choose $\delta_0>0$ small enough that for $0<\delta<\delta_0$ we have that $R^{(l)}_N(x)=O(N^{-1})$ uniformly in both $x\in \R\setminus\bigcup_{i=0}^{m+1}(\lam_i-\delta,\lam_{i}+\delta)$ and $\{\lam_i\}_{i=1}^{m}$ in our chosen compact set.
	\end{prop}
	
	In the proof of both of these results, we will need to understand the parametric $P_{\infty}$ given in Definition 4.12 of \cite{GUEGMC}. While we will not recall the definitions for each of their functions, we will recall some of their real limits for the ease of the reader. Namely the functions $r$, $a$, and $\cal{D}_1$ defined in (4.19), (4.20) and (4.21) satisfy for $x\in \R$
	\[r_+(x)=\begin{cases}\sign(x)\sq{x^2-1};\;\; |x|\ge 1\\
		i\sq{1-x^2};\;\; |x|<1
	\end{cases},\;\; a_+(x)=\frac{|x-1|}{|x+1|}\begin{cases}1\;\; |x|\ge 1\\
		e^{i\pi/4};\;\; |x|<1
	\end{cases},\]
	\[(\cal{D}_1)_+(x)=D(x)\omega(x)^{1/2}\begin{cases}1;\;\; |x|\ge 1\\
		e^{\cal{W}(x)}e^{-\pi i \sum_{l=1}^{m}\alpha_l I(x<\lam_l)};\;\; |x|<1
	\end{cases},\]
	where in the last case for $\cal{D}_1$ we have used the Sokhotski–Plemelj formula: for smooth $f$ and $a<0<b$ we have that
	\[\lim_{\epsilon\to 0^+}\int_a^b\frac{f(x)}{x+i\epsilon}dx=-i\pi f(0)+\pv \int_a^b\frac{f(x)}{x}dx.\]
	From this we can see that for $x\in (-1,1)\setminus\{\lam_1,\cdots, \lam_m\}$,
	\[(P_{\infty})_+(x)=D_{\infty}^{\sigma}T_{\infty}(x)e^{-i\pi\sum_{k=1}^{m}I(x<\lam_k)\alpha_k\sigma}e^{-\cal{W}(x)\sigma}\omega(x)^{-\sigma/2}, \label{Paramatrix-(-1,1)}\]
	and for $x\in \R\setminus(-1,1)$ we have that
	\[(P_{\infty})_+(x)=D_{\infty}^{\sigma}T_{\infty}(x)\omega(x)^{-\sigma/2}. \label{Paramatrix infinity}\]
	
	\begin{proof}[Proof of Lemma \ref{Coefficient Lemma New}]
		The proof will be similar to the proof of the $\cal{W}=0$ case given in Section 4 of \cite{Krav} (one may also see the proof of this case given in Section 5 of \cite{GUEAsym}). We will first provide a proof in the special case of $k=1$. We observe that we may write
		\[\kappa_{N-1}^2=\frac{1}{-2\pi i}\lim_{x\to \infty}\frac{[(Y_N)_+]_{21}(x)}{x^{N-1}},\;\;\;\beta_{N}=\lim_{x\to \infty}\frac{[(Y_N)_+]_{11}(x)-x^N}{x^{N-1}}.\]
		Now employing Theorem \ref{Simple-Remainder Proposition} in view of (\ref{first-col-outside}) and (\ref{Paramatrix infinity}), we see that for any small $\delta>0$, $x>1+\delta$, and $j=1,2$
		\[e^{N\ell/2(-1)^{j}}[(Y_N)_+]_{j1}(x)=[(I+R_N(x))D_{\infty}^{\sigma}T_{\infty}(x)\omega(x)^{-\sigma/2}]_{j1}e^{N(x^2-s(x))},\label{ignore-throw}\]
		where here $R_N(x)=O(N^{-1})$ uniformly in the sense of Theorem \ref{Simple-Remainder Proposition}. Routine application of Taylor's Theorem shows that
		\[s(x)=x\sq{x^2-1}-\arcosh(x)=x^2+\ell/2-\log(x)+O(x^{-2}), \;\;\;\omega(x)^{-1/2}=x^{-\cal{A}}(1+\sum_{i=1}^{m}\frac{\alpha_i\lam_i}{x}+O(x^{-2})),\]
		\[A(x)=I+\frac{1}{2x}\begin{bmatrix}0&i\\-i&0\end{bmatrix}+O(x^{-2}),\;\;\; D(x)=x^{-\cal{A}}D_{\infty}(1+\frac{1}{\pi x}\int_{-1}^1 \frac{\cal{W}(u)u}{\sq{1-u^2}}du+O(x^{-2})).\]
		From these results, we see that for $j=1,2$,
		\[[D_{\infty}^{\sigma}T_{\infty}(x)\omega(x)^{-\sigma/2}]_{j1}e^{N(x^2-s(x))}=D_{\infty}^{-(-1)^j}[A(x)]_{j1}D(x)^{-1}\omega(x)^{-1/2}e^{N(x^2-s(x))}=\]
		\[D_{\infty}^{-(-1)^j-1}\bigg[I(1+\frac{1}{x}\bigg[\sum_{i=1}^{m}\alpha_i\lam_i-\frac{1}{\pi}\int_{-1}^1 \frac{\cal{W}(u)u}{\sq{1-u^2}}du\bigg])+\frac{1}{2x}\begin{bmatrix}0&i\\-i&0\end{bmatrix}+O(x^{-2})\bigg]_{j1}x^Ne^{-N\frac{\ell}{2}}=\]
		\[e^{-N\frac{\ell}{2}}D_{\infty}^{-(-1)^j-1}\begin{bmatrix}1+x^{-1}(\sum_{i=1}^{m}\alpha_i\lam_i-\frac{1}{\pi}\int \frac{\cal{W}(u)u}{\sq{1-u^2}}du)\\-x^{-1}i/2\end{bmatrix}_jx^N+O(x^{N-2}).\label{throw-eqn2}\]
		We
		now recall from (4.58) of \cite{GUEGMC} that $R_N(x)=O(|x|^{-1})$. In fact, we additionally have that $R_N(x)=O(N^{-1}|1+x|^{-1})$, uniformly in $x\in \R$, and compact subsets of $\{\lam_i\}_{i=1}^{m}$, as in Proposition \ref{Simple-Remainder Proposition}. This is not explicitly stated in \cite{GUEGMC}, but follows from a standard contour deformation argument using the uniform decay of the jump matrix ($\Delta=J_R-I$ in \cite{GUEGMC}) for large $|z|$, as in the $\cal{W}=0$ case discussed in Section 4 of \cite{Krav} (one also may consult the detailed argument given in the proof of Proposition 7.5 of \cite{clae} in a related case).
		
		Employing this asymptotic, (\ref{throw-eqn2}), and (\ref{ignore-throw}), we see that
		\[\kappa_{N-1}^2=\frac{1}{-2\pi i}\lim_{x\to \infty}\frac{[(Y_N)_+]_{21}(x)}{x^{N-1}}=\frac{1}{4\pi}e^{-N\ell }D_{\infty}^{-2}(1+O(N^{-1}))=\pi^{-1}e^{-(N-1)\ell+1 }D_{\infty}^{-2}(1+O(N^{-1})).\]
		Similarly, we derive that
		\[\beta_N=\lim_{x\to \infty}\frac{[(Y_N)_+]_{11}(x)-x^N}{x^{N-1}}=\sum_{i=1}^{m}\alpha_i\lam_i-\frac{1}{\pi}\int_{-1}^1\frac{\cal{W}(u)u}{\sq{1-u^2}}du+O(N^{-1})=O(1).\]
		
		This completes the proof of the $k=1$ case. To complete the general case, we see that
		\[\kappa_{N-k-1}(N,\cal{W},\{\lam_i\}_i,\{\alpha_i\}_i)=\eta_k^{-\cal{A}-1/2-(N-k-1)}\kappa_{N-k-1}(N-k,\cal{W}_{N,k},\{\eta_k^{-1}\lam_i\}_i,\{\alpha_i\}_i),\]
		\[\beta_{N-k}(N,\cal{W},\{\lam_i\}_i,\{\alpha_i\}_i)=\eta_k\beta_{N-k}(N-k,\cal{W}_{N,k},\{\eta_k^{-1}\lam_i\}_i,\{\alpha_i\}_i).\]
		Observing that $\eta_k^{-\cal{A}-1/2-(N-k-1)}=e^{k/2}(1+O(N^{-1}))$ and that $D_{\infty}(\cal{W}_{N,k},\{\eta_k^{-1}\lam_i\}_i,\{\alpha_i\}_i)=D_{\infty}(1+O(N^{-1}))$, we obtain the general case by applying the $k=1$ case to $w_{N,k}$.
	\end{proof}
	Before continuing on to the proof of Proposition \ref{Asym p_N easy}, we record the following useful remark.
	\begin{remark}
		\label{det-remark-T}
		We observe as $T_{\infty}(x)$ is smooth on $\R\setminus\{\pm 1\}$ and independent of $\{\lam_i\}_{i=1}^{m}$ and $N$, we may show that for fixed $k\in \Z$, if we define $\bar{R}_N(x)$ so that $T_{N,k,\infty}(x)=T_{\infty}(x)+\bar{R}_N(x)$,
		then for $l\in \N$, $\bar{R}_N^{(l)}(x)=O(N^{-1})$ uniformly over both compact subsets of $x\in \R\setminus\{\pm 1\}$ (and of course still independently of $\{\lam_i\}_{i=1}^{m}$). On the other hand, taking $R_{N}(x)=\bar{R}_{N}(x)T_{\infty}^{-1}(x)$, we may write
		\[T_{N,k,\infty}(x)=(I+R_N(x))T_{\infty}(x).\]
		Now as one may check that $\det(T_{\infty}(x))=1$, we see that the entries of $T_{\infty}^{-1}(x)$ are also smooth on $\R\setminus\{\pm 1\}$, so we see that $R_N$ satisfies the same estimates as $\bar{R}_N$.
	\end{remark}
	\begin{proof}[Proof of Proposition \ref{Asym p_N easy}]
		We will take $\delta=\delta_0/2$, where $\delta_0$ is the constant given in Proposition \ref{Simple-Remainder Proposition}. We will begin with the proof in the case that $k=0$. By Proposition \ref{Simple-Remainder Proposition}, (\ref{Paramatrix-(-1,1)}), and (\ref{first-col-bulk}) we see that for $x\in (\lam_{i+1}+\delta,\lam_{i}-\delta)$
		\[D_{\infty}^{-\sigma}e^{-N\frac{\ell}{2}\sigma}\begin{bmatrix}\kappa_N^{-1}p_{N}(x)\\-2\pi i\kappa_{N-1}p_{N-1}(x)\end{bmatrix}w_N(x)^{1/2}=D_{\infty}^{-\sigma}(I+\bar{R}_N(x))(P_\infty)_+(x)e^{Nis(x)\sigma}e^{\cal{W}(x)\sigma}\omega(x)^{\sigma/2}\begin{bmatrix}1\\ 1\end{bmatrix}=\]
		\[D_{\infty}^{-\sigma}(I+\bar{R}_N(x))D_{\infty}^{\sigma}T_{\infty}(x)e^{-i\pi\sum_{k=1}^{i}\alpha_k\sigma}\begin{bmatrix}e^{Nis(x)}\\ e^{-Nis(x)}\end{bmatrix}=(I+R_N(x))T_{\infty}(x)e^{-i\pi\sum_{k=1}^{i}\alpha_k\sigma}\begin{bmatrix}e^{Nis(x)}\\ e^{-Nis(x)}\end{bmatrix},\]
		where $\bar{R}_N(x)$ is the error term from Proposition \ref{Simple-Remainder Proposition} and $R_N(x)=D_{\infty}^{-\sigma}\bar{R}_N(x)D_{\infty}^{\sigma}$. By Remark \ref{Remark-coefficents-new}, we see that modifying $R_N(x)$ by the constant matrix $I+O(N^{-1})$ appearing in (\ref{coefficent-new-eqn}), we obtain (\ref{Bulk Asymptotic}) in the $k=0$ case. Now using the fact that the coefficients $D_{\infty}$ are $N$-independent and independent of $\{\lam_i\}_i$, we see that the function $R_N$ satisfies the same estimates as $\bar{R}_N$, as stated in Proposition \ref{Simple-Remainder Proposition}. In particular, $R_N^{(l)}(x)=O(N^{-1})$ uniformly in the sense needed in Proposition \ref{Asym p_N easy}. Similarly, using (\ref{Paramatrix infinity}) and (\ref{first-col-outside}) and the same choice of $R_N$, we obtain (\ref{Outside Asymptotic}) in the case of $k=0$. Together these complete the $k=0$ case.
		
		To establish the case of $k\neq 0$, with notation as above, we see that by Remark \ref{det-remark-T}, Remark \ref{Remark-coefficents-new}, and the $k=0$ case applied to the measure $w_{N,k}$, we see for $x\in (\lam_{i+1}+\delta,\lam_{i}-\delta)$,
		
		Now recalling $r_{N}(x)$ from Lemma \ref{sk expansion lem} we see that
		\[T_{\infty}(x)e^{-i\pi\sum_{k=1}^{i}\alpha_k\sigma}\begin{bmatrix}e^{i(N-k)s_{N,k}(x)}\\ e^{-i(N-k)s_{N,k}(x)}\end{bmatrix}=(I+\hat{R}_N(x))T_{\infty}(x)e^{[-ik\arcos(x)-i\pi\sum_{k=1}^{i}\alpha_k]\sigma}\begin{bmatrix}e^{iNs(x)}\\ e^{-iNs(x)}\end{bmatrix},\]
		where $\hat{R}_N(x)=(I+R_N(x))T_{\infty}(x)e^{r_{N}(x)\sigma}T_{\infty}(x)^{-1}-I$. As $r_N(x)=O(N^{-1})$, $T_{\infty}(x)$ is smooth on $(-1,1)$, and $\det(T_{\infty}(x))=1$, we see from the estimates on $R_N$ that $\hat{R}_N^{(l)}(x)=O(N^{-1})$ for $l\ge 0$. Together, these results complete the proof of (\ref{Bulk Asymptotic}). The general case of (\ref{Outside Asymptotic}) follows more simply from Remark \ref{Remark-coefficents-new} and a similar appeal to the $k=0$ case.
	\end{proof}
	
	\begin{remark}
		\label{exp-decay-s}
		for any $\delta>0$, we may find $c>0$, such that $s(x)>2cx^2$ for $x\in \R\setminus [-1-\delta,1+\delta]$. Noting that $T_{\infty}$ grows only polynomially in $x$, we see that the asymptotic (\ref{Paramatrix infinity}) implies that for any choice of $k\in \Z$, $l\ge 0$, and $|x|>1+\delta$, we have that
		\[p_{N-k}^{(l)}(x)w_N(x)^{1/2}=O(e^{-Ncx^2}),\]
		and additionally, for any compact subset of $\{\lam_i\}_{i=1}^{m}\in \{(\lam_1,\cdots \lam_m)\in (-1,1)^m:\lam_1>\cdots >\lam_m\}$, we may choose $c>0$ small enough that this error bound is uniform in $|x|>1+\delta$ and our chosen compact subset of $\{\lam_i\}_{i=1}^{m}$.
	\end{remark}
	
	We will now turn to describe the asymptotics of the orthogonal polynomials around the points $\pm 1$. We define functions
	\[f_{1}(x)=\sign(x-1)\left(\frac{3}{2}|s(x)|\right)^{2/3},\;\;\;f_{-1}(x)=f_1(-x).\]
	It is routine to verify that $f_{\pm 1}$ are defined and smooth in a neighborhood around $\pm 1$ (see also Definition 4.25 of \cite{GUEGMC}). Moreover, one may verify that $f_{\pm 1}'(\pm 1)=\pm 2$ and that $f_{\pm 1}(\pm 1)=0$. We also define $f_{N,k,\pm1}(z)=f_{\pm 1}(\eta_k^{-1} z)$ as before. We will also need additional matrix valued functions
	\[T_{1}(x)=\sq{\pi} e^{-i\frac{\pi}{4}}T_{\infty}(x)e^{\cal{W}(x)I(x\ge 1)\sigma}e^{i\frac{\pi}{4} \sigma}\begin{bmatrix}1&-1\\
		1&1\end{bmatrix}e^{i\sigma\frac{\pi}{4}I(x<1)}|f_{1}(x)|^{\sigma/4},\label{eqn:T1-def}\]\[ T_{-1}(x)=\sq{\pi} e^{-i\frac{\pi}{4}}T_{\infty}(x)e^{\cal{W}(x)I(x\le -1)\sigma}e^{-i\pi \cal{A}\sigma}e^{i\frac{\pi}{4}\sigma}\begin{bmatrix}1&1\\
		-1&1\end{bmatrix}e^{i\sigma\frac{\pi}{4}I(x>-1)}|f_{-1}(x)|^{\sigma/4}.\]
	We are now ready to state the asymptotics around $\pm 1$. In the following proposition, $\Ai$ denotes the Airy function.
	\begin{prop}
		\label{Asym p_N Airy}
		For $\delta>0$ small enough, and any choice of $k\in \Z$, we may find $R_N(x)$ such that for $x\in (1-\delta,1+\delta)$ we have that
		\[\begin{bmatrix}\pi^{1/2}p_{N-k}(x)\\
			-i\pi^{1/2}p_{N-k-1}(x)\end{bmatrix}w_N(x)^{1/2}=(I+R_N(x))T_{1}(x)\begin{bmatrix}
			(N-k)^{1/6}\Ai((N-k)^{2/3}f_{N,k,1}(x))\\
			(N-k)^{-1/6}\Ai'((N-k)^{2/3}f_{N,k,1}(x))
		\end{bmatrix}\label{pN asym on 1},\]
		and for $x\in (-1-\delta,-1+\delta)$ we have that \[\begin{bmatrix}\pi^{1/2}p_{N-k}(x)\\
			-i\pi^{1/2}p_{N-k-1}(x)\end{bmatrix}w_N(x)^{1/2}=(I+R_N(x))T_{-1}(x)\begin{bmatrix}
			(N-k)^{1/6}\Ai((N-k)^{2/3}f_{N,k,-1}(x))\\
			-(N-k)^{-1/6}\Ai'((N-k)^{2/3}f_{N,k,-1}(x))
		\end{bmatrix}(-1)^{N-k}\label{pN asym on -1}.\]
		Additionally, for any $l\ge 0$, and any choice of compact subset of $\{\lam_i\}_{i=1}^{m}\in \{(\lam_1,\cdots \lam_m)\in (-1,1)^m:\lam_1>\cdots >\lam_m\}$, we may choose $\delta_0>0$ such that for $0<\delta<\delta_0$ we have that $R^{(l)}_N(x)=O(N^{-1})$ uniformly in both $x\in (-1-\delta,-1+\delta)\cup (1-\delta,1+\delta)$ and $\{\lam_i\}_{i=1}^{m}$ in our chosen compact set.
	\end{prop}
	
	To prove this, we will have to recall the form of $P_{\pm 1}$. We begin with the case of $P_{1}$, which is given in Definition 4.28 of \cite{GUEGMC}. We observe that for the function $\xi_1(z) $ defined in (4.42) of \cite{GUEGMC}, we have that $(\xi_{1})_+(x)=N^{2/3}f_1(x)$, and additionally we have that the function defined in (4.28) of \cite{GUEGMC} has $(\phi_1)_+(x)=(h_1)_+(x)/2$. Employing these observations and Lemma \ref{misc-function-lemma} we derive from Definition 4.28 of \cite{GUEGMC} that for $x\in (1-\delta,1+\delta)$
	\[(P_{1})_+(x)=F_+(x)Q_+(Nf_1(x))e^{(I(x\ge 1)-iI(x<1))s(x)\sigma}\omega(x)^{-\sigma/2}e^{-\cal{W}(x)\sigma},\label{P+1-exp}\]
	\[F_+(x)=(P_{\infty})_+(x)e^{\cal{W}(x)\sigma}\omega(x)^{\sigma/2}e^{i\frac{\pi}{4}\sigma}\sq{\pi}\begin{bmatrix}1&-1\\ 1&1 \end{bmatrix}e^{i\sigma\frac{\pi}{4}I(x<1)}|f_1(x)|^{\sigma/4}N^{\sigma/6}e^{-i\pi/12},\]
	\[Q_+(\zeta)=\begin{bmatrix}\Ai(\zeta)& \Ai(\omega^2\zeta)\\
		\Ai'(\zeta)&\omega^2 \Ai'(\omega^2 \zeta)\end{bmatrix}e^{-\pi i \sigma/6};\;\;\zeta> 0,\]
	\[Q_+(\zeta)=\begin{bmatrix}\Ai(\zeta)& \Ai(\omega^2\zeta)\\
		\Ai'(\zeta)&\omega^2 \Ai'(\omega^2 \zeta)\end{bmatrix}e^{-\pi i \sigma/6}\begin{bmatrix}1&0\\-1&1\end{bmatrix};\;\;\zeta<0,\]
	where here $\omega=e^{2\pi i/3}$. Now employing (\ref{Paramatrix-(-1,1)}) and (\ref{Paramatrix infinity}) we see that
	\[F_+(x)=D_{\infty}^{\sigma}T_{\infty}(x)e^{\cal{W}(x)I(x\ge
		1)}e^{i\frac{\pi}{4}\sigma}\sq{\pi}\begin{bmatrix}1&-1\\ 1&1 \end{bmatrix}e^{i\sigma\frac{\pi}{4}I(x<1)}|f_1(x)|^{\sigma/4}N^{\sigma/6}e^{-i\pi/12}=D_{\infty}^{\sigma}e^{i\pi/6}T_1(x)N^{\sigma/6}.\label{F+1-new}\]
	
	We now proceed to the $-1$ case. While not defined in \cite{GUEGMC}, the definition is a similar modification of the equation obtained in the $\cal{W}=0$ case considered in (80) of \cite{Krav} (or similarly, from the $+1$ case for the reversed measure $w_N(-x,\cal{W}(-x),\{-\lam_i\}_{i=1}^{m},\{\alpha_i\}_{i=1}^{m}$). In particular, we obtain by a similar argument that for $x\in (-1-\delta,-1+\delta)$
	\[(P_{-1})_+(x)=\hat{F}_+(x)\sigma Q_-(Nf_{-1}(x))\sigma e^{(I(x\le -1)-iI(x>-1))s(x)\sigma}\omega(x)^{-\sigma/2}e^{-\cal{W}(x)\sigma}(-1)^{N\sigma},\label{P-1-exp}\]
	\[\bar{F}_+(x)=D_{\infty}^{\sigma}e^{i\frac{\pi}{6}}T_{-1}(x)N^{\sigma/6},\label{F-1-new}\]
	\[\sigma Q_-(\zeta)\sigma  =\begin{bmatrix}\Ai(\zeta)& \omega^2 \Ai(\omega\zeta)\\
		-\Ai'(\zeta)&-\Ai'(\omega^2 \zeta)\end{bmatrix}e^{-\pi i \sigma/6};\;\zeta\ge 0,\]
	\[\sigma Q_-(\zeta)\sigma =\begin{bmatrix}\Ai(\zeta)& \omega^2 \Ai(\omega\zeta)\\
		-\Ai'(\zeta)&-\Ai'(\omega^2 \zeta)\end{bmatrix}e^{-\pi i \sigma/6}\begin{bmatrix}1&0\\-1&1\end{bmatrix};\;\zeta<0.\]
	
	\begin{remark}
		\label{det-remark-airy}
		We observe that $F(z)$ ($\bar{F}(z)$) is analytic in a neighborhood around $1$ ($-1$) respectively (see Lemma 4.30 of \cite{GUEGMC}). In particular, from (\ref{F+1-new}) and (\ref{F-1-new}) we see that $T_{\pm 1}(z)$ is as well. We note that $T_{\pm 1}$ is independant of $\{\lam_i\}_i$ and $N$. As we also have $|\det(T_{\pm 1}(x))|=4\pi$, the same aurgument as in Remark \ref{det-remark-T} shows that for $k\in \Z$, there is $R_N(x)$ such that for $x\in (\pm 1-\delta,\pm 1+\delta)$
		\[T_{\pm 1}(\eta_k^{-1}x;\cal{W}_{N,k},\{\eta_k^{-1}\lam_i\}_i,\{\alpha_i\}_i)=(I+R_N(x))T_{\pm 1}(x),\]
		where for any $\ell\ge 0$, $R_N^{(\ell)}(x)=O(N^{-1})$ uniformly in both $x\in (\pm 1-\delta,\pm 1+\delta)$ (and again independently of the choice of $\{\lam_i\}_{i=1}^{m}$).
	\end{remark}
	
	\begin{proof}[Proof of Proposition \ref{Asym p_N Airy}]
		The proof is essentially the same as that of Proposition \ref{Asym p_N easy}. In particular, employing Proposition \ref{Simple-Remainder Proposition}, (\ref{Paramatrix-(-1,1)}), (\ref{first-col-bulk}), (\ref{first-col-outside}) and Remark \ref{Remark-coefficents-new} as in the proof of Proposition \ref{Asym p_N easy}, we obtain the $k=0$ case routinely from the expansions for $P_{\pm 1}$ given by (\ref{P+1-exp}), (\ref{P-1-exp}), (\ref{F+1-new}), and (\ref{F-1-new}). The proof of the arbitrary $k$ follows similarly by employing Remark \ref{det-remark-airy}.
	\end{proof}
	
	Finally, we will state the asymptotics in the region around the points $\{\lam_1,\cdots,\lam_m\}$. Define for $x\in (-1,1)$, the function $f_{x}(z)=s(x)-s(z)=2\int_x^z\sq{1-y^2}dy$. This function is smooth and invertible in a neighborhood of $x$, and we see that $f_x'(x)=2\sq{1-x^2}$. As before, we define $f_{N,k,x}(z)=f_{\eta_k^{-1}x}(\eta_k^{-1}z)$. We define a matrix-valued function for $1\le j\le m$ by
	\[T_{N,k,\lam_j}(x)=e^{-i\frac{\pi}{4}}T_{\infty}(x)e^{i\frac{\rho_{N,j,k}}{2}\sigma}e^{-i\frac{\pi}{4}\sigma}\frac{1}{\sq{2}}\begin{bmatrix}1&i\\i&1\end{bmatrix},\label{T_lam-def}\]
	where $\rho_{N,j,k}$ denotes
	\[\rho_{N,j,k}=2(N-k)s_{N,k}(\lam_j)-\pi \alpha_j-2\pi \sum_{i=1}^{j-1}\alpha_i.\]
	For simplicity, we will denote $T_{N,0,\lam_j}(x)=T_{N,\lam_j}(x)$ and $\rho_{N,j,0}=\rho_{N,j}$.
	\begin{remark}
		\label{T_lam-size-remark}
		Unlike $T_{\infty}(x)$ and $T_{\pm 1}(x)$, the function $T_{N,k,\lam_j}(x)$ depends on $N$. On the other hand, this is only through the factor $\exp(i\rho_{N,j,k}\sigma/2)$, whose norm is $N$-independent. Thus by Remark \ref{det-remark-T} we see that $T_{N,k,\lam_j}(x)$ is smooth on $(-1,1)$, and that for any $\delta>0$, and any $l\ge 0$, $T_{N,k,\lam_j}^{(l)}(x)$ is uniformly bounded jointly in $N$, $x\in (\lam_j-\delta,\lam_j+\delta)$ and any chosen compact subset of $\{\lam_i\}_{i=1}^{m}\in (-1,1)^m$. Observing that $|\det(T_{N,k,\lam_j}(x))|=1$, we see additionally that the same is true for $T_{N,k,\lam_j}^{-1}(x)$. Lastly, combining these observations with Lemma \ref{sk expansion lem}, we see that for any $k\in \Z$, there is $R_N(x)$ such that
		\[T_{N,k,\lam_j}(x)=(I+R_N(x))T_{N,\eta_k^{-1}\lam_j}(\eta_k^{-1}x;\cal{W}_{N,k},\{\eta_k^{-1}\lam_i\}_i,\{\alpha_i\}_i),\]
		where for any $l\ge 0$, $R_N^{(l)}(x)=O(N^{-1})$ uniformly both $x\in (\lam_j-\delta,\lam_j+\delta)$ and compact subsets of $\{\lam_i\}_{i=1}^{m}\in (-1,1)^m$.
	\end{remark}
	Lastly, we define the following vector of special functions
	\[\cal{J}_{\alpha}(x)=\begin{bmatrix}\sign(x)\sq{\pi |x|}J_{\alpha+1/2}(|x|)\\ \sq{\pi |x|}J_{\alpha-1/2}(|x|) \end{bmatrix},\label{calJ-def}\]
	where here $J_{\nu}$ denotes the Bessel function of parameter $\nu$.
	
	\begin{prop}
		\label{Asym p_N Bessel}
		For $\delta>0$ small enough, and any choice of $k\in \Z$, we may find $R_N(x)$ such that for any choice of $1\le i\le m$ and $x\in (\lam_i-\delta,\lam_i+\delta)$ we have that
		\[\begin{bmatrix}\pi^{1/2}p_{N-k}(x)\\
			-i\pi^{1/2}p_{N-k-1}(x)\end{bmatrix}w_N(x)^{1/2}=(I+R_N(x))T_{N,k,\lam_i}(x)\cal{J}_{\alpha}((N-k)f_{N,k,\lam_i}(x))\label{Bessel Asymptotic}.\]
		Additionally, for any $l\ge 0$, and any choice of compact subset of $\{\lam_i\}_{i=1}^{m}\in \{(\lam_1,\cdots \lam_m)\in (-1,1)^m:\lam_1>\cdots >\lam_m\}$, we may choose $\delta_0>0$ such that for $0<\delta<\delta_0$ we have that $R^{(l)}_N(x)=O(N^{-1})$ uniformly in both $x\in \bigcup_{k=1}^{m}(\lam_k-\delta,\lam_k+\delta)$ and $\{\lam_i\}_{i=1}^{m}$ in our chosen compact set.
	\end{prop}
	
	As before, this result will follow from analyzing the parametrix $P_{\lam_j}$ defined in Definition 4.21 of \cite{GUEGMC}. We observe that the function defined in (4.27) of \cite{GUEGMC} satisfies \[(W_j)_+(x)=e^{\cal{W}(x)}\omega(x)^{1/2}e^{-i\pi \alpha_j\sign(x-\lam_j)}.\]
	Employing this we see that for $x\in (\lam_j-\delta,\lam_j+\delta)$
	\[(P_{\lam_j})_+(x)=(E_j)_+(x)(\Psi_{\alpha_j})_+(Nf_{\lam_j}(x))e^{-\cal{W}(x)\sigma}\omega(x)^{-\sigma/2}e^{i\pi \alpha_j\sign(x-\lam_j)\sigma}e^{-Nis(x)\sigma},\label{Para-def-stop}\]
	\[(E_j)_+(x)=(P_\infty)_+(x)e^{\cal{W}(x)\sigma}\omega(x)^{\sigma/2}e^{-i\pi \alpha_j\sign(x-\lam_j)
		\sigma/2}e^{iNs(\lam_j)\sigma}e^{-i\frac{\pi}{4}\sigma}\frac{1}{\sq{2}}\begin{bmatrix}
		1&i\\
		i&1
	\end{bmatrix},\]
	and where $\Psi_{\alpha}$ is a fixed matrix-valued function (defined in (4.26-4.33) of \cite{strong-RH}) and satisfying
	\[(\Psi_{\alpha})_+(\zeta)=\frac{\sq{\pi}}{2}\sq{\zeta}\begin{bmatrix}H_{\alpha+1/2}^{(2)}(\zeta)&-iH_{\alpha+1/2}^{(1)}(\zeta)\\
		H_{\alpha-1/2}^{(2)}(\zeta)&-iH_{\alpha-1/2}^{(1)}(\zeta)\end{bmatrix}e^{-(\alpha+1/4)\pi i \sigma};\;\;\zeta>0,\]
	\[(\Psi_{\alpha})_+(\zeta)=\frac{\sq{\pi}}{2}\sq{-\zeta}\begin{bmatrix}iH_{\alpha+1/2}^{(1)}(-\zeta)&-H_{\alpha+1/2}^{(2)}(-\zeta)\\
		-iH_{\alpha-1/2}^{(1)}(-\zeta)&H_{\alpha-1/2}^{(2)}(-\zeta)\end{bmatrix}e^{(\alpha+1/4)\pi i \sigma};\;\;\zeta<0,\]
	where here $H^{(1)}_{\beta}$ and $H^{(2)}_{\beta}$ denote the Hankel functions of the first and second kind at parameter $\beta$, respectively. Employing (\ref{Paramatrix-(-1,1)}) we see that we may write
	\[(E_j)_+(x)=D_{\infty}^{\sigma}e^{i\frac{\pi}{4}}T_{N,\lam_j}(x). \label{E_j-form}\]
	
	\begin{proof}[Proof of Proposition \ref{Asym p_N Bessel}]
		As we have that $H_{\beta}^{(1)}(x)+H_{\beta}^{(2)}(x)=2J_{\beta}(x)$ (see Section 9.1 of \cite{AS}), we see that for $\zeta\in \R$
		\[(\Psi_{\alpha})_+(\zeta)e^{i\pi\alpha\sign(\zeta)\sigma}\begin{bmatrix}1\\ 1 \end{bmatrix}=e^{-i\pi/4}\cal{J}_{\alpha}(\zeta).\]
		Combining this with (\ref{Para-def-stop}), (\ref{E_j-form}), and observing that $\sign(x-\lam_j)=\sign(Nf_{\lam_j}(x))$, we see that
		\[(P_{\lam_j})_+(x)e^{Nis(x)\sigma}e^{\cal{W}(x)\sigma}\omega(x)^{\sigma/2}\begin{bmatrix}1\\ 1\end{bmatrix}=D_{\infty}^{\sigma}e^{i\frac{\pi}{4}}T_{N,\lam_j}(x)e^{-i\frac{\pi}{4}}\cal{J}_{\alpha_j}(Nf_{\lam_j}(x))=D_{\infty}^{\sigma}T_{N,\lam_j}(x)\cal{J}_{\alpha_j}(Nf_{\lam_j}(x)).\]
		Now employing Proposition \ref{Simple-Remainder Proposition}, (\ref{first-col-bulk}), and Remark \ref{Remark-coefficents-new} we obtain the $k=0$ case of (\ref{Bessel Asymptotic}). The proof for arbitrary $k$ follows from Remark \ref{T_lam-size-remark}.
	\end{proof}
	
	Together these results will allow us to uniformly describe the behavior of $p_{N-k}$ on $\R$. What remains is to understand the behavior of $\ell_i$ and $q_i$ introduced in (\ref{l-def}) and (\ref{q-def}) for $1\le i\le m$. To do so, we introduce the following matrix
	\[V_{N,\lam_i}=e^{N\lam_i^2\sigma}N^{\alpha_i\sigma}\begin{bmatrix}-\kappa_{N-1}H(p_{N-1})(\lam_i)&-\kappa_N^{-1}(2\pi i)^{-1}H(p_N)(\lam_i)\\
		2\pi i\kappa_{N-1}p_{N-1}(\lam_i)&\kappa^{-1}_Np_{N}(\lam_i)\end{bmatrix}.\]
	The utility of this matrix comes from the following relation
	\[(x-\lam_i)\begin{bmatrix}-\ell_i(x)\\
		2\pi i q_i(x)\end{bmatrix}=V_{N,\lam_i}\begin{bmatrix}\kappa_{N}^{-1}p_{N}(x)\\
		-2\pi i\kappa_{N-1}p_{N-1}(x)\end{bmatrix}.\label{l and q relation}\]
	In particular, in view of the above understanding of $p_{N}$ and $p_{N-1}$, to understand $\ell_i$ and $q_i$, we only need to provide asymptotics for $V_{N,\lam_i}$.
	\begin{prop}
		\label{l and q asym}
		For $1\le j\le m$, we have that
		\[V_{N,\lam_j}=c_{j}^{-\sigma}e^{i\frac{\pi}{4}}e^{-\frac{\pi i}{4}\sigma}\sigma_1T_{N,\lam_j}(\lam_j)^{-1}D_{\infty}^{-\sigma}(I+O(N^{-1}))e^{-N\frac{\ell}{2}\sigma},\]
		where here we denote
		\[c_{j}=\frac{\sq{2\pi}(1-\lam_j^2)^{\alpha_j/2}}{\Gamma(\alpha_j+1/2)}e^{-\cal{W}(\lam_j)}\prod_{k\neq
			j}|\lam_k-\lam_j|^{-\alpha_k}.\]
		Moreover the error term is uniform over compact subsets of $\{\lam_i\}_{i=1}^{m}\in \{(\lam_1,\cdots \lam_m)\in (-1,1)^m:\lam_1>\cdots >\lam_m\}$.
	\end{prop}
	
	\begin{remark}
		\label{l and q remark}
		We observe that by combining Remark \ref{Remark-coefficents-new}, Proposition \ref{l
			and q asym}, and (\ref{l and q relation}) we see that for $1\le i\le m$ there are constants $C_{k,N}$ for $1\le k\le 4$, such that
		\[\ell_{i}(x)=C_{1,N}\frac{p_N(x)}{x-\lam_i}+C_{2,N}\frac{p_{N-1}(x)}{x-\lam_i},\]
		\[q_{i}(x)=C_{3,N}\frac{p_N(x)}{x-\lam_i}+C_{4,N}\frac{p_{N-1}(x)}{x-\lam_i},\]
		where for all $i$, we have that $C_{i,N}=O(1)$ uniformly in compact subsets of $\{\lam_i\}_{i=1}^{m}\in \{(\lam_1,\cdots \lam_m)\in (-1,1)^m:\lam_1>\cdots >\lam_m\}$.
	\end{remark}
	
	\begin{proof}
		By the Christoffel-Darboux formula (\ref{chris-1}), we see that $\det(Y_N(z))=1$. From this, we observe that for $1\le j\le m$
		\[V_{N,\lam_j}=e^{N\lam_j^2\sigma}N^{\alpha_j\sigma}(Y_N)_+(\lam_j)^{-1}.\label{fur-rel}\]
		Additionally, we observe the matrix relation
		\[e^{-\cal{W}(x)\sigma}\omega(x)^{-\sigma/2}e^{-Nis(x)\sigma}\begin{bmatrix}1 &0\\
			e^{-2Nis(x)-2\cal{W}(x)}\omega(x)^{-1}&1\end{bmatrix}e^{Nis(x)\sigma}e^{Nx^2\sigma}=\begin{bmatrix}1 &0\\
			1&1\end{bmatrix}w_N(x)^{-\sigma/2}.\]
		Using this relation and Proposition \ref{Simple-Remainder Proposition}, (\ref{paramatrix-Y-S-(-1,1)}), (\ref{Para-def-stop}) and  (\ref{E_j-form}), we see that for $x\in (\lam_j,\lam_j+\delta)$
		\[(Y_N)_+(x)=e^{N\frac{\ell}{2}\sigma}(I+R_N(x))D_{\infty}^{\sigma}e^{\frac{i\pi}{4}}T_{N,\lam_j}(x)(\Psi_{\alpha_j})_+(Nf_{\lam_j}(x))e^{i\pi \alpha_j\sigma}\begin{bmatrix}1 &0\\
			1&1\end{bmatrix}w_N(x)^{-\sigma/2},\label{tl-qf}\]
		where $R_N$ is the remainder term from Proposition \ref{Simple-Remainder Proposition}. We recall from Section 4.3 of \cite{strong-RH} (more precisely, the analytic continuation of their (4.21)) that for $\zeta>0$:
		\[(\Psi_{\alpha})_+(\zeta)\begin{bmatrix}1&0\\ e^{-2\pi i\alpha} &1\end{bmatrix}=\sq{\zeta}\begin{bmatrix}\sq{\pi}I_{\alpha+1/2}(-i\zeta)&-\frac{1}{\sq{\pi}}K_{\alpha+1/2}(-i\zeta)\\ -i\sq{\pi}I_{\alpha-1/2}(-i\zeta)& -\frac{i}{\sq{\pi}}K_{\alpha-1/2}(-i\zeta) \end{bmatrix}e^{-i\frac{\alpha}{2}\pi \sigma}, \label{new-temp-112}\]
		where here $I_{\alpha}$ and $K_{\alpha}$ are modified Bessel functions of the first and second kind, respectively, where both are taken with respect to the principal branch given by $\arg(z)\in (-\pi,\pi)$. The asymptotics of the right-hand side of (\ref{new-temp-112}) as $\zeta\to 0$ are computed in (112) of \cite{Krav} as
		\[e^{i\frac{\pi}{4}}\bigg(\begin{bmatrix}0&-1\\
			-i&0\end{bmatrix}+C_\alpha(\zeta)\bigg)\bigg(\frac{\sq{2\pi}\zeta^{\alpha}}{\Gamma(\alpha+1/2)2^{\alpha}}\bigg)^{\sigma}e^{-i\pi\alpha\sigma},\]
		where $C_{\alpha}(\zeta)=O(\zeta\log(|\zeta|))$ entrywise as $\zeta\to 0$. We note as well that
		\[e^{-i\pi \alpha\sigma}\begin{bmatrix}1&0\\ e^{-2\pi i\alpha} &1\end{bmatrix}^{-1}e^{i\pi \alpha\sigma}\begin{bmatrix}1 &0\\
			1&1\end{bmatrix}=I.\]
		Recalling that both $R_N(x)$ and $T_{N,\lam_j}(x)$ are continuous at $\lam_j$, we see by (\ref{tl-qf}) that
		\[(Y_N)_+(\lam_j)=e^{N\frac{\ell}{2}\sigma}(I+R_N(\lam_j))D_{\infty}^{\sigma}e^{i\frac{\pi}{4}}T_{N,\lam_j}(\lam_j)e^{i\frac{\pi}{4}}\begin{bmatrix}0&-1\\
			-i&0\end{bmatrix}\bigg(\frac{\sq{2\pi}(Nf'_{\lam_j}(\lam_j))^{\alpha_{j}}}{2^{\alpha_j}\Gamma(\alpha_j+1/2)}\bigg)^{\sigma}\times \]\[(\prod_{k\neq j}|\lam_j-\lam_k|^{-\alpha_k\sigma})e^{-\cal{W}(\lam_j)\sigma}e^{N\lam_j^2\sigma}=e^{N\frac{\ell}{2}\sigma}(I+R_N(\lam_j))D_{\infty}^{\sigma}T_{N,\lam_j}(\lam_j)\begin{bmatrix}0&-i\\1&0\end{bmatrix}c_{j}^{\sigma}N^{\alpha_j\sigma}e^{N\lam_j^2\sigma}.\]
		where in the last step we have used that $f_{x}'(x)=2\sq{1-x^2}$. Combining this with (\ref{fur-rel}) and observing that \[\begin{bmatrix}0&-i\\1&0\end{bmatrix}^{-1}=e^{i\frac{\pi}{4}}e^{-\frac{\pi i}{4}\sigma}\sigma_1\] completes the proof.
	\end{proof}
	
	\subsection{Modifications for Merging Singularities \label{subsection-merging}}
	
	We will conclude this section by considering the modifications that need to be made when we are in the case of merging singularities required in Proposition \ref{Main-prop-Merging-2}. In particular, for the rest of this section assume that we are in the case of $m=2$, $\cal{W}=0$, $\alpha_1=\alpha_2=\alpha>0$, and are allowing $\lam_1,\lam_2$ to depend on $N$. Analysis of the Riemann-Hilbert problem with such merging singularities is undertaken in \cite{Merging}, where they modify the above parametrices to obtain asymptotics in this case. We will primarily be reliant on the results on merging singularities in \cite{Merging,RHmerging}, as well as the discussion in \cite{GUEGMC}. As we only require the supercritical case for merging singularities (i.e., we consider $|\lam_1-\lam_2|^{-1}=O(N^{\gamma})$ where $\gamma<1$), all the above results will still hold only with worse bounds on the error terms and different domains.
	
	Away from the points $\{\lam_1,\lam_2\}$ though, the above asymptotics hold essentially as stated.
	\begin{prop}
		\label{Merging Error Bulk}
		Let $\epsilon>0$. Then there is $\delta_0>0$, such that for $0<\delta<\delta_0$, and any choice of $\lam_1,\lam_2\in (-1+\epsilon,1-\epsilon)$ (possibly $N$-dependant) with $\delta>\lam_1-\lam_2>0$, there is an error term $R_N$, such that for $l\ge 0$, we have that $R^{(l)}_N(x)=O(N^{-1})$ uniformly in both $x\in \R\setminus\{\pm 1-\delta,\pm 1+\delta\}\cup [\lam_2-\delta,\lam_1+\delta]$ and the choice $(\lam_1,\lam_2)$, and such that the following asymptotics hold: For $x\in (-1+\delta,1-\delta)\setminus (\lam_2-\delta,\lam_1+\delta)$, we have (\ref{Bulk Asymptotic}), for $|x|\ge 1+\delta$, we have (\ref{Outside Asymptotic}), and for $x\in (\pm 1-\delta,\pm 1+\delta)$ we have (\ref{pN asym on 1}) or (\ref{pN asym on -1}), for $\pm 1=1$ and $\pm 1=-1$, respectively. Lastly, (\ref{constant-asymptotics-new}) and thus (\ref{coefficent-new-eqn}) hold as well, with error term uniform in the choice of $(\lam_1,\lam_2)$ as well.
	\end{prop}

	For the remaining regions around $\lam_1$ and $\lam_2$, the asymptotics of Proposition \ref{Asym p_N Bessel} and Proposition \ref{Bulk Asymptotic} occur in a smaller region and have a worse error bound.
	
	\begin{prop}
		\label{Merging Error Bessel}
		Let $\epsilon>0$ and $0<\gamma<1$. Then there is $\delta_0>0$, such that for $0<\delta<\delta_0$, and any choice of $\lam_1,\lam_2\in (-1+\epsilon,1-\epsilon)$ (possibly $N$-dependant) with $\delta>\lam_1-\lam_2>N^{-\gamma}$, the following holds. Denoting $r_N=2(\lam_1-\lam_2)$, there is an error term $R_N$, such that for $l\ge 0$, we have that $R^{(l)}_N(x)=O(r_N^{-1})$ uniformly in $x\in(\lam_2-\delta,\lam_1+\delta)$ and the choice of $(\lam_1,\lam_2)$, and such that the following asymptotics hold: For $x\in (\lam_2-\delta,\lam_2-\delta r_N^{-1})\cup (\lam_2+\delta r_N^{-1},\lam_1-\delta r_N^{-1})\cup (\lam_1+\delta r_N^{-1},\lam_1+\delta)$, we have (\ref{Bulk Asymptotic}), and for $1\le i\le 2$, and $x\in (\lam_i-\delta r_N^{-1},\lam_i+\delta r_N^{-1})$, we have that (\ref{Bessel Asymptotic}) holds. Additionally, the asymptotics of Proposition \ref{l and q asym} holds with error term now of order $O(N^{\gamma-1})$, but now uniform in the choice of $(\lam_1,\lam_2)$.
	\end{prop}
	
	We will now begin to establish the prerequisites for the proofs for these statements. For the remainder of the section we will denote $u=(\lam_1-\lam_2)^2/4$ and $v=(\lam_1+\lam_2)/2$. As before, this will essentially follow from recalling and simplifying various parametrices. More specifically, in \cite{Merging} they give asymptotics of $Y_N$ in the merging case in terms of a related matrix-valued function $''S''$, which we denote by $S^M_N$ to avoid confusion with the above $S_N$. To aid the reader in the translation of their notation, their $(n,t,V,a,b)$ coincides with our $(N,u,-2(x-v)^2,1-v,-1-v)$. The definition of $S^M_N$ in \cite{Merging} is given in (4.10) and (5.1), only employing Fig. 8 rather than Fig. 7, as explained in their Section 6.1. The construction is almost identical to the one in \cite{GUEGMC} (which we used above) with the only non-notational change being that for $x\in (\lam_1,\lam_2)$ the function $S_N^M$ is defined by (\ref{paramatrix-Y-S-outside}) rather than (\ref{paramatrix-Y-S-(-1,1)}) (see the difference between Figure 1 of \cite{GUEGMC} and Figure 8 of \cite{Merging}). In particular, the function $(S^M_N)_+$ may be given in terms of the above $(S_N)_+$ as follows: For $x\notin [\lam_1,\lam_2]$, $(S^M_N)_+(x)=(S_N)_+(x)$, and for $x\in (\lam_1,\lam_2)$, we have that \[(S_N^M)_+(x)=(S_N)_+(x)\begin{bmatrix} 1 & 0 \\ 1 & 1 \end{bmatrix}.\label{S-S^M}\] 
	They then, as above, solve this by employing a sequence of paramatrices. We are only concerned with the case that $t>0$, which is dealt with in Section 6 of \cite{Merging}. They construct paramatrices $N$, $P^{(a)}$, $P^{(b)}$, and $P$, where $P^{(a)}$ is defined in a neighborhood of $1$, $P^{(b)}$ is defined in a neighborhood of $-1$, $P$ is defined in a neighborhood of $[\lam_2,\lam_1]$, and $N$ is defined in the remainder. These first three though are defined as $P^{(\infty)}$, $P^{(1)}$, and $P^{(-1)}$ in \cite{GUEGMC}, which coincide with our $P_{\infty}$, $P_1$, and $P_{-1}$. In particular, they prove the following result.
	\begin{prop}
		\label{prop merging}
		Let us fix $\epsilon>0$. Then there is $\delta_0>0$ such that for any $0<\delta<\delta_0$, and any choice of $\lam_1,\lam_2\in (-1+\epsilon,1-\epsilon)$ (possibly $N$-dependant) with $\delta>(\lam_1-\lam_2)>0$. Then there is $R_N$ such that
		\[(S^M_N)_+(x)=(I+R_N(x))P_{\delta}(x),\]
		where here \[(P_{\delta})_+(x)=\begin{cases}(P_{u,v})_+(x);\;x\in (\lam_2-\delta,\lam_1+\delta)\\
			(P_{\pm 1})_+(x);\;x\in (\pm 1-\delta,\pm 1+\delta)\\
			(P_{\infty})_+(x);\;\text{otherwise}\end{cases},\]
		and where for any $l\ge 0$, we have that $R_{N}^{(l)}(x)=O(N^{-1})$ uniformly in $x\in \R\setminus\{\pm 1-\delta,\pm 1+\delta,\lam_2-\delta,\lam_1+\delta\}$ and the choice of $(\lam_1,\lam_2)$.
	\end{prop}
	This result for $R_N$ essentially follows from Section 6.5 of \cite{Merging}, though they only state the error bounds for $R_N$. On the other hand, they actually establish that $R_N$ satisfies a small norm Riemann-Hilbert problem, with $N$-independent contours, whose jump matrices are $O(N^{-1})$ uniformly in a $N$-independent neighborhood of the contours. In particular, applying a contour deformation argument and the Cauchy integral formula as in the proof of Theorem 4.37 in \cite{GUEGMC} (see also the discussion in Section 4 of \cite{Krav}), we obtain the desired error bounds on $R_N^{(l)}$ for all $l\ge 0$. In addition, while the results of \cite{Merging} establish uniformly only in $u$, the extension to uniformity in  $v$ is shown in Appendix G of \cite{GUEGMC}. 
	
	As the paramatrices $P_{\infty}$ and $P_{\pm 1}$ here coincide with those considered in the previous subsection exactly, replacing Proposition \ref{Simple-Remainder Proposition} with Proposition \ref{prop merging} as appropriate in the proofs in the previous subsection gives a proof of Proposition \ref{Merging Error Bulk} (noting that $(S_N)_+=(S_N^M)_+$ in all regions considered in this proposition).
	
	On the other hand, the parametrix $P_{u,v}$ does not directly coincide with any of the parametrices above, and indeed when one lets $(\lam_1-\lam_2)=O(N^{-1})$, it appears to be quite different from a simple combination of the ones considered above. Instead, we must use the asymptotics for $P_{u,v}$ in powers of $O(N^{-1}u^{-1/2})=O(N^{-1}r_N)$ which are provided in \cite{RHmerging}, which we will show yields the following result.
	
	\begin{lem}
		\label{merging-paramatrix-lemma}
		With the set-up of Proposition \ref{prop merging}, there is $\delta_1>0$, such that if $(\lam_1-\lam_2)>\delta_1N^{-1}$, then there is $R_N(x)$ such that
		\[(P_{u,v})_+(x)=(I+R_N(x))(P_{\infty })_+(x);\;\; x\notin (\lam_2-\sq{u} \delta,\lam_1+\sq{u} \delta),\]
		\[(P_{u,v})_+(x)=(I+R_N(x))(P_{\infty })_+(x)\begin{bmatrix}1&0\\
			1&1\end{bmatrix};\;\; x\in (\lam_2+\sq{u} \delta,\lam_1-\sq{u} \delta),\]
		\[(P_{u,v})_+(x)=(I+R_N(x))(P_{\lam_1})_+(x);\;\;\; x\in (\lam_1,\lam_1+\sq{u} \delta),\]
		\[(P_{u,v})_+(x)=(I+R_N(x))(P_{\lam_1})_+(x)\begin{bmatrix}1&0\\
			1&1\end{bmatrix};\;\;x\in (\lam_1-\sq{u} \delta,\lam_1),\]
		\[(P_{u,v})_+(x)=(I+R_N(x))(P_{\lam_{2}})_+(x);\;\;\; x\in (\lam_2-\sq{u}\delta,\lam_2),\]
		\[(P_{u,v})_+(x)=(I+R_N(x))(P_{\lam_2})_+(x)\begin{bmatrix}1&0\\
			1&1\end{bmatrix};\;\;x\in (\lam_2,\lam_2+\sq{u} \delta),\]
		where for any $l\ge 0$, we have that $R_{N}^{(l)}(x)=O(u^{-1/2}N^{-1})$ uniformly in both $x$ in these regions, and in the choice of $(\lam_1,\lam_2)$.
	\end{lem}
	
	Recalling (\ref{S-S^M}) for $x\in [\lam_1,\lam_2]$, we see that the combined asymptotics of this lemma and Proposition \ref{merging-paramatrix-lemma} essentially coincide with those in Proposition \ref{Simple-Remainder Proposition} around $\{\lam_1,\lam_2\}$, except that now that choice of $\delta$ is now $N$-dependant and the error terms are now of order $O(N^{-1}u^{-1/2})$ instead of $O(N^{-1})$. Thus repeating the proofs given above, replacing $O(N^{-1})$ by $O(N^{-1}u^{-1/2})=O(N^{-1}r_N)$ when necessary is sufficient to prove Proposition \ref{Merging Error Bessel}.
	
	The remainder of this subsection will consist of the proof of Lemma \ref{merging-paramatrix-lemma}. This essentially follows from Section 5 of \cite{RHmerging}, though we will have to trace a variety of notations from \cite{Merging,RHmerging}. A similar translation is given in Appendix G of \cite{GUEGMC}.
	
	In preparation, let us define for $x\in \R$
	\[\lam(x)=N(\frac{s(\lam_1)+s(\lam_2)}{2}-s(x)),\quad s_{N}=2Ni(s(\lam_2)-s(\lam_1)).\]
	We recall a form for the parametrix $P_{u,v}$, which is defined in Section 6.4 of \cite{Merging}. For $x\in(v-\delta_0,v+\delta_0)$
	\[(P_{u,v})_+(x)=E_+(x)\Psi^{(2)}_+(\lam(x);s_{N})\omega(x)^{\sigma/2}e^{Nis(x)\sigma}(\sigma\sigma_1),\]
	\[E_+(x)=D_{\infty}^{\sigma}T_{\infty}(x)(\sigma\sigma_1)^{-1}e^{i\pi \alpha\sigma}e^{-iN\frac{s(\lam_1)+s(\lam_2)}{2}\sigma},\label{psi-merge}\]
	where $\Psi^{(2)}(\lam,s)$ is a fixed matrix-valued function defined in Section 3 of \cite{Merging}, which depends only on $\alpha$ and $s$. We will need the asymptotics of $\Psi^{(2)}_+$ as $s\to \infty$. To get these, we note that in view of (3.11) and (3.5) of \cite{Merging}, we may write
	\[\Psi^{(2)}(x,s)=\Psi_{CK}(\frac{-4x}{|s|}i;s)\begin{cases} e^{-\sign(x)\pi i\alpha\sigma};\;\;|x|>|s|/4\\
		I;\;\; |x|<|s|/4
	\end{cases},\label{recant}\]
	where $\Psi_{CK}$ is the a solution to the model Riemann-Hilbert problem introduced in Section 3 of \cite{RHmerging} with parameters $\alpha_1=\alpha_2=\alpha$ and $\beta_1=\beta_2=0$. In Section 5 of \cite{RHmerging} uniform asymptotics for $\Psi_{CK}$ in the case of $|s|\to \infty$ are derived. To demonstrate these, the Riemann-Hilbert problem for $\Psi_{CK}$ is solved using elementary transformations and yet another auxiliary Riemann-Hilbert problem $M$. We note for the convenience of the reader that in the case relevant to us (i.e., $\beta=0$), $M$ is expressed explicitly in terms of Hankel and modified Bessel functions in (2.47) of \cite{Merging}. In particular, we see that
	\[M_+(x)=(\sigma\sigma_1)e^{-i\frac{\pi}{4}\sigma}\frac{1}{\sq{2}}\begin{bmatrix}1&i\\
		i&1\end{bmatrix}(\Psi_{\alpha})_+(x)(\sigma\sigma_1)^{-1}e^{-i\frac{\pi}{2}\alpha i\sigma}.\label{final-merging}\]
	The following asymptotics of $\Psi^{(2)}(x;s)$ follow from those in Section 5 of \cite{RHmerging} and (\ref{recant}).
	\begin{lem}
		\label{s-asymptotics-lemma}
		There exists $s_0>0$ and $\delta_0>0$ such that for any $s\ge s_0$ and $\delta_0>\delta>0$, there is $\bar{R}_s$, such that
		\[\Psi^{(2)}_+(x;s)=(I+\bar{R}_s(x))e^{ix\sigma}e^{-i\alpha\pi\sigma\sign(x)};\;\; |x|>\frac{|s|}{4}+|s|\delta\]
		\[\Psi^{(2)}(x;s)=(I+\bar{R}_s(x))e^{ix\sigma}\begin{bmatrix}1&-1\\
			0&1\end{bmatrix};\;\; |x|<\frac{|s|}{4}-|s|\delta\]
		\[\Psi^{(2)}(x;s)=(I+\bar{R}_s(x))e^{-i\pi\alpha/2\sigma}e^{-i|s|/4\sigma}M(x
		)e^{\frac{3}{2}\alpha i\pi \sigma};\;\; \frac{|s|}{4}\le x\le \frac{|s|}{4}+|s|\delta\]
		\[\Psi^{(2)}(x;s)=(I+\bar{R}_s(x))e^{-i\pi \alpha/2\sigma}e^{-i|s|/4\sigma}M(x)e^{-\frac{1}{2}\alpha i\pi \sigma}\begin{bmatrix}1&-1\\
			0&1\end{bmatrix};\;\;\frac{|s|}{4}-|s|\delta\le x\le \frac{|s|}{4}\]
		\[\Psi^{(2)}(x;s)=(I+\bar{R}_s(x))e^{-\frac{3i\pi}{2} \alpha\sigma}e^{i|s|/4\sigma}M(x)e^{-\frac{1}{2}\alpha\pi \sigma};\;\; \frac{|s|}{4}\le -x\le \frac{|s|}{4}+|s|\delta\]
		\[\Psi^{(2)}(x;s)=(I+\bar{R}_s(x))e^{-\frac{3i\pi}{2} \alpha\sigma}e^{i|s|/4\sigma}M(x)e^{\frac{3}{2}\alpha i\pi \sigma}\begin{bmatrix}1&-1\\
			0&1\end{bmatrix};\;\;\frac{|s|}{4}-|s|\delta\le -x\le \frac{|s|}{4},\]
		and such that for $l\ge 0$, we have that $\bar{R}_s^{(l)}(x)=O(s^{-k}(s+|x|)^{-1})$ uniformly in $s\in (s_0,\infty)$ and $x$ in the regions above.
	\end{lem}
	\begin{remark}
		The factor of $s^{-k}$ occuring in the error term in Lemma \ref{s-asymptotics-lemma} is due to the $s^{-1}$-scaling in (\ref{recant}). Indeed if one denotes by $\hat{R}_s$ the error term for the corresponding error term in the expansion of $\Psi_{CK}(x;s)$, one has that for any $l\ge 0$, $\hat{R}_s^{(l)}(x)=O(s^{-1}(1+|x|)^{-1})$ (see (5.25) of \cite{RHmerging}). 
	\end{remark}
	\begin{proof}[Proof of Lemma \ref{merging-paramatrix-lemma}]
		We recall that $\sigma_1\sigma=-\sigma\sigma_1$, so that $a^{\sigma}\sigma_1=\sigma a^{-\sigma}$. We also note that
		\[(\sigma\sigma_1)^{-1}\begin{bmatrix}
			1&-1\\ 0 & 1
		\end{bmatrix}(\sigma\sigma_1)=\begin{bmatrix}
			1&0\\ 1 & 1
		\end{bmatrix}.\]
		With these results and (\ref{final-merging}), all of these expressions follow by routine computation. For example, we observe for $x\in (\lam_1,\lam_1+\sq{u}\delta)$ that (potentially changing $\delta$) $\lam(x)\in (|s|/4,|s|/4+\delta |s|)$ so that
		\[(P_{u,v})_+(x)=D_{\infty}^{\sigma}T_{\infty}(x)(\sigma\sigma_1)^{-1}e^{i\pi \alpha\sigma}e^{-iN\frac{s(\lam_1)+s(\lam_2)}{2}\sigma}(I+\bar{R}_s(\lam(x)))e^{iN\frac{s(\lam_1)+s(\lam_2)}{2}\sigma}e^{-i\alpha\pi\sigma}\omega(x)^{\sigma/2}(\sigma \sigma_1).\]
		Noting that $\bar{R}_{s_N}(\lam(x))=O(u^{-1/2}N^{-1})$ we see that conjugating $\bar{R}_{s_N}$ as above
		\[(P_{u,v})_+(x)=(I+O(N^{-1}u^{-1/2}))D_{\infty}^{\sigma}T_\infty(x).\]
		All claims follow from similar computations.
	\end{proof}
	
	\section{Proof of Proposition \ref{Exponential-Error-Prop} \label{Exponential Error Section}}
	
	In this section, we will prove Proposition \ref{Exponential-Error-Prop}. As $\cal{E}$ vanishes on a neighborhood of $[-1,1]$, the integrals which compose the entries of $\Delta^1_N$ are supported on regions where the integrand is exponentially small (see Remark \ref{exp-decay-s}). Combining this with some rough estimates for the elements of $\Delta^0_N$ (see Proposition \ref{Norm-Error-Proposition}), we will be able to obtain Proposition \ref{Exponential-Error-Prop} by employing bounds for the determinant under perturbation.
	
	Now, as before we fix a choice of $(\cal{W},\{\lam_i\}_i,\{\alpha_i\}_i)$ satisfying the hypotheses of Theorem \ref{main-theorem}. For a linear operator $A$ on an inner-product space, let $\|A\|_F=\sq{\Tr(A^*A)}$ denote the Frobenius norm. We recall the operators $\Delta_N^1$ and $\Delta_N^0$ (defined in (\ref{eqn:ignore-144}) and (\ref{eqn:ignore-145})) of Section \ref{preliminary-section}, which act on $\cal{P}_N$ with the inner-product $(f,g)_{w,N}=\int f(x)g(x)w_N(x)dx$. The main technical result of this section is the following proposition.
	\begin{prop}
		\label{Norm-Error-Proposition}
		There are $C,c>0$, such that
		\[\|\Delta_N^1\|_F=O(e^{-Nc}),\;\;\;\|\Delta_N^0\|_F=O(N^C).\]
		Moreover for any choice of compact subset of $\{\lam_i\}_{i=1}^{m}\in \{(\lam_1,\cdots \lam_m)\in (-1,1)^m:\lam_1>\cdots >\lam_m\}$, we may choose $C,c$ such that these estimates are uniform on the chosen compact set.
	\end{prop}
	We will now give the proof of Proposition \ref{Exponential-Error-Prop} assuming Proposition \ref{Norm-Error-Proposition}.
	\begin{proof}[Proof of Proposition \ref{Exponential-Error-Prop}]
		We recall from Lemma \ref{Added-lem} that $[\Delta_N^0]_{22}=I$, so that $[\Delta_N]_{22}=I+[\Delta^1_N]_{22}$. We observe that by Proposition \ref{Norm-Error-Proposition}, we have that $\|[\Delta_N^1]_{22}\|_F=O(e^{-Nc})$. Thus by the Neumann series for the inverse, we see that $[\Delta_N]_{22}$ is invertible and that 
		$\|[\Delta_N]_{22}^{-1}-I\|_F=O(e^{-Nc})$. By the Schur complement formula, we have that
		\[\det(\Delta_N)=\det([\Delta_N]_{22})\det([\Delta_N]_{11}-[\Delta_N]_{12}[\Delta_N]_{22}^{-1}[\Delta_N]_{21}).\label{shur-4}\]
		Now recall that $[\Delta_N^0]_{12}=0$ so that $[\Delta_N]_{12}=[\Delta_N^1]_{12}$. Thus we have that
		\[\|[\Delta^1_N]_{11}-[\Delta_N]_{12}[\Delta_N]_{22}^{-1}[\Delta_N]_{21}\|_F=\]\[\|[\Delta^1_N]_{11}-[\Delta_N^1]_{12}[\Delta_N]_{22}^{-1}[\Delta_N]_{21}\|_F\le\|[\Delta^1_N]_{11}\|_F+\|[\Delta^1_N]_{12}\|_F\|[\Delta_N]_{22}^{-1}\|_F\|[\Delta_N]_{21}\|_F.\]
		The work above shows that $\|[\Delta_N]_{22}^{-1}\|_F=\|I\|_F+O(e^{-Nc})=O(N)$, and by Proposition \ref{Norm-Error-Proposition} we have that $\|[\Delta^1_N]_{11}\|_F,\|[\Delta^1_N]_{12}\|_F=O(e^{-Nc})$. Combining this with the above inequality, we see that
		\[\|[\Delta^1_N]_{11}-[\Delta_N]_{12}[\Delta_N]_{22}^{-1}[\Delta_N]_{21}\|_F=O(e^{-Nc/2}).\label{nu-bound-sec-4}\]
		We recall the bound on $\ell$-by-$\ell$ matrices $A$ and $B$ (see Theorem 2.12 of \cite{Det-Bounds})
		\[|\det(A)-\det(B)|\le \ell \|A-B\|_F(\max(\|A\|_F,\|B\|_F))^{\ell-1}.\]
		Recalling that $\dim(\cal{P}_{N,1})=m+d$, we see that by applying this inequality, the bound of (\ref{nu-bound-sec-4}), and Proposition \ref{Norm-Error-Proposition}
		\[|\det([\Delta_N]_{11}-[\Delta_N]_{12}[\Delta_N]_{22}^{-1}[\Delta_N]_{21})-\det([\Delta_N^0]_{11})|=O(e^{-Nc/4}).\label{nu-add}\]
		We will further need the bound on $\ell$-by-$\ell$ matrices $A$ (see Corollary 2.14 of \cite{Det-Bounds})
		\[|\det(I+A)-1|\le (\ell\|A\|_F+1)^\ell-1.\]
		Employing the formula $[\Delta_N]_{22}=I+[\Delta_N^1]_{22}$, as well as Proposition \ref{Norm-Error-Proposition}, this bound yields that $|\det([\Delta_N]_{22})-1|= O(e^{-Nc/2})$.
		Using this, (\ref{shur-4}), and (\ref{nu-add}), we have that 
		\[\det(\Delta_N)=\det([\Delta_N]_{22})(\det([\Delta_N^0]_{11})+O(e^{-Nc/4}))=\det([\Delta_N^0]_{11})(1+O(e^{-Nc/2}))+O(e^{-Nc/4}).\]
		Finally applying Proposition \ref{Norm-Error-Proposition} to see that $\det([\Delta_N^0]_{11})=O(e^{Nc/4})$, and recalling that $\det([\Delta^0_N]_{11})=\det(\Delta^0_N)$ completes the proof.
	\end{proof}
	The remainder of the section will be spent on the proof of Proposition \ref{Norm-Error-Proposition}. We first note that by Proposition \ref{Commutator-Decom} we have that
	\[\|\Delta^1_N\|_F=\|\Pi_N J_N^{-1}\Pi_N^{\perp}\cal{E}' \Pi_N \|_F\le \|\Pi_N J_N^{-1}\cal{E}' \Pi_N \|_F+\|\Pi_N J_N^{-1}\Pi_N\cal{E}' \Pi_N \|_F,\label{wall}\]
	\[\|\Delta^0_N\|_F=\|\Pi_N J_N^{-1}\Pi_N J_N \Pi_N\|_F\le \|I\|_F+2N\|\Pi_N J_N^{-1}\Pi_N x\Pi_N \|_F+\sum_{i=1}^{m}|\alpha_i|\|\Pi_N J_N^{-1}\ell_iq_i^T\Pi_N\|_F,\label{money}\]
	where here, the transpose $f^T$ is taken with respect to the inner-product on $\cal{P}_N$ inherited from $L^2(w_N)$. Thus it suffices to compute the order of all terms on the right. The terms $\|\Pi_N J_N^{-1}\cal{E}' \Pi_N \|_F$, $\|\Pi_N J_N^{-1}\Pi_N\cal{E}' \Pi_N \|_F$ and $\|\Pi_N J_N^{-1}\Pi_N x\Pi_N \|_F$ will all follow from the same method, which we now illustrate. Let $A$ be an operator on $\cal{P}_N$ with Schwartz kernel $Af(x)=\int G(x,y)f(y)w_N(y)dy$. Then evaluating the Frobenius norm with respect to the basis given by orthogonal polynomials, we see that
	\[\|A\|^2_F=\sum_{i=0}^{N-1}\int \int \int G(x,y)G(x,z)p_i(y)p_i(z)w_N(x)w_N(y)w_N(z)dxdydz=\]\[\int \int \int G(x,y)G(x,z)K_N(y,z)w_N(x)w_N(y)w_N(z)dxdydz.\label{frob-decom}\]
	Let us denote $\cal{K}_N(x,y)=w_N(x)^{1/2}w_N(y)^{1/2}K_N(x,y)$ where $K_N$ is the Christoffel-Darboux kernel. In the sense above, the operator $\Pi_N J_N^{-1}\cal{E}'\Pi_N$ has kernel $G(x,y)=w_N^{-1/2}(x)(\frac{1}{2}\int \cal{K}_N(x,z)\sign(z-y)dz)\cal{E}'(y)w_N^{-1/2}(y)$ on $\cal{P}_N$. We may employ (\ref{frob-decom}) to obtain
	\[\|\Pi_N J_N^{-1}\cal{E}' \Pi_N \|^2_F=\]
	\[\label{Schwartz-kernel}\frac{1}{4}\int \cdots \int \cal{K}_N(x_1,x_2)\sign(x_2-x_3)\cal{E}'(x_3)\cal{K}_N(x_1,x_4)\sign(x_4-x_5)\cal{E}'(x_5)\cal{K}_N(x_3,x_5)\prod_{i=1}^{5}dx_i.\]
	By taking the absolute value of the integrand, we may upper-bound this quantity by
	\[\frac{1}{4}(\int \int \int |\cal{K}_N(x_1,x_2)\cal{K}_N(x_1,x_4)|dx_1dx_2dx_4)(\int \int |\cal{E}'(x_3)\cal{E}'(x_5)\cal{K}_N(x_3,x_5)|dx_3dx_5).\label{Schwartz-kernel-2}\]
	By the Cauchy-Schwarz inequality we have that $|\cal{K}_N(x,y)|\le \cal{K}_N(x,x)^{1/2}\cal{K}_N(y,y)^{1/2}$. Repeatedly applying this bound to (\ref{Schwartz-kernel-2}), and observing that $\int \cal{K}_N(x,x)dx=N$, we see that $\|\Pi_N J_N^{-1}\cal{E}' \Pi_N \|^2_F\le \frac{N}{4}I_1^2I_2^2$,
	where
	\[I_1=\int \cal{K}_N(w,w)^{1/2}dw,\;\;I_2=\int \cal{K}_N(w,w)^{1/2}|\cal{E}'(w)|dw.\]
	To compute these integrals let us denote $\varphi_{N,1}(x)=p_N(x)w_N(x)^{1/2}$ and $\varphi_{N,2}(x)=p_{N-1}(x)w_N(x)^{1/2}$. 
	We note that by Remark \ref{exp-decay-s} for any $\delta>0$ there is $c>0$ small enough so that for $|x|>1+\delta$
	\[\frac{d}{dx}\varphi_{N,i}(x),\varphi_{N,i}(x)=O(\exp(-cNx^2)).\]  
	By the Christoffel-Darboux formula (\ref{chris-2}) and Lemma \ref{Coefficient Lemma New}, we see that $\cal{K}_N(x,x)=O(\exp(-cNx^2))$ uniformly on $|x|>1+\delta$. Now applying the Cauchy-Schwarz inequality, we see that
	\[I_1=\int_{|w|>1+\delta}\cal{K}_N(w,w)^{1/2}dw+\int_{|w|<1+\delta}\cal{K}_N(w,w)^{1/2}dw\le\]
	\[\int_{|w|>1+\delta}\cal{K}_N(w,w)^{1/2}dw+(2+2\delta)^{1/2}(\int \cal{K}_N(w,w)dw)^{1/2}\label{I2-method}.\]
	The integral on the right-hand side of (\ref{I2-method}) is seen to be of order $O(e^{-Nc/4})$ using the above asymptotic, and the second is exactly $(2+2\delta)^{1/2}N^{1/2}$, so that $I_1=O(N^{1/2})$. When applying
	the same trick to $I_2$, the first term is $O(e^{-Nc/4})$, and the latter vanishes, so that $I_2=O(e^{-Nc/4})$, so that $\|\Pi_N J_N^{-1}\cal{E}'\Pi_N\|_F^2=O(e^{-Nc/8})$. The same method works to show that $\|\Pi_N J_N^{-1}\Pi_N\cal{E}' \Pi_N \|_F=O(e^{-Nc/4})$ and that $\|\Pi_N J_N^{-1}\Pi_N x\Pi_N \|_F=O(N^{5})$.
	
	We now focus our attention to $\|\Pi_N J_N^{-1}\ell_iq_i^T\Pi_N\|_F$. In this case we see that, recalling (\ref{Evaulation-Lemma}) and applying again the inequality $|\cal{K}_N(x,y)|\le \cal{K}_N(x,x)^{1/2}\cal{K}_N(y,y)^{1/2}$,
	\[\|\Pi_N J_N^{-1}\ell_iq_i^T\Pi_N\|^2_F=\]
	\[e^{-2N\lam_i^2}N^{-2\alpha_i}K_N(\lam_i,\lam_i)\int \int K_N(x_1,x_2)J_N^{-1}\ell_i(x_1)J_N^{-1}\ell_i(x_2)w_N(x_1)w_N(x_2)dx_1dx_2\le \]
	\[e^{-2N\lam_i^2}N^{-2\alpha_i}K_N(\lam_i,\lam_i)(\int \cal{K}_N(x,x)^{1/2}|\epsilon(\ell_iw_N^{1/2})(x)|dx)^2,\]
	where here
	\[\epsilon(f)(x)=\frac{1}{2}\int \sign(x-y)f(y)dy.\]
	
	To understand these quantities we will first employ Proposition \ref{Bessel Asymptotic} to write, for $x\in (\lam_i-\delta,\lam_i+\delta)$ and $k=1,2$
	\[\varphi_{N,k}(x)=G_{N,k}(x)\cal{J}_{\alpha_i}(Nf_{\lam_i}(x)),\;\; G_{N,k}(x)=i^{I(k=2)}\pi^{-1/2}e_k^T(I+R_N(x))T_{N,\lam_i}(x).\label{throw-castle}\]
	We also recall some classical asymptotics of Bessel functions (see Chapter 9 of \cite{AS}): for $\nu>-1$ and $x\in (0,\infty)$, we have that \[x^{-\nu}J_{\nu}(x),\;\frac{d}{dx}(x^{-\nu}J_{\nu}(x))=O((1+x)^{-\nu-1/2}),\]
	where both the errors are uniform in $x\in (0,\infty)$. From this, we see that for $\nu>-1/2$ and for $x\in \R$,
	\[x^{-\nu}\cal{J}_{\nu}(x),\; \frac{d}{dx}(x^{-\nu}\cal{J}_\nu(x))=O((1+|x|)^{-\nu+1/2}),\label{moss-tow}\]
	where again, the error term is uniform in $x\in \R$. We will denote $\hat{\varphi}_{N,k}(x)=\varphi_{N,k}(x)|x-\lam_i|^{-\alpha_i}$, which removes the singular part of $\varphi_{N,k}$ around $\lam_i$ coming from $w_N(x)^{1/2}$. Now employing (\ref{moss-tow}) and Remark \ref{T_lam-size-remark} to the representation in (\ref{throw-castle}), we see that for $x\in (\lam_i-\delta,\lam_i+\delta)$ and $k=1,2$, we have that,
	\[\hat{\varphi}_{N,k}(x),\;\frac{d}{dx}\hat{\varphi}_{N,k}(x)=O(N^{\alpha_i+2}),\label{can-asym}\]
	again uniformly in $x$ and compact subsets of $\{\lam_i\}_{i=1}^m\in \{(\lam_1,\cdots \lam_m)\in (-1,1)^m:\lam_1>\cdots >\lam_m\}$.
	Applying this to the Christoffel-Darboux formula (\ref{chris-2}), we obtain that
	\[e^{-2N\lam_i^2}N^{-2\alpha_i}K_N(\lam_i,\lam_i)=O(N^{8}).\]
	We now focus on the term $\epsilon (\ell_iw_N^{1/2})(x)$. We note that
	\[\left|\int \sign(x-y) \frac{\varphi_{N,k}(y)dy}{y-\lam_i}\right|\le \frac{1}{\delta}\int_{|y-\lam_i|>\delta} |\varphi_{N,k}(y)|dy+\left|\int_{|y-\lam_i|<\delta} \sign(x-y) \frac{\varphi_{N,k}(y)}{y-\lam_i}dy\right|.\label{pink}\]
	Using that $\int \varphi_{N,k}^2(y)dy=1$ and proceeding as in (\ref{I2-method}), we may see that the first integral on the right-hand side of (\ref{pink}) is $O(1)$, leading us to bound the second. We have that
	\[\int_{|y-\lam_i|<\delta}\sign(x-y)\frac{\varphi_{N,k}(y)}{y-\lam_i}dy=\int_{|y-\lam_i|<\delta}\sign(x-y)\frac{(\hat{\varphi}_{N,k}(y)-\hat{\varphi}_{N,k}(\lam_i))|y-\lam_i|^{\alpha_i}}{y-\lam_i}dy+\]\[\hat{\varphi}_{N,k}(\lam_i)\int_{|y-\lam_i|<\delta}\sign(x-y)\frac{|y-\lam_i|^{\alpha_i}}{y-\lam_i}dy.\]
	The latter term is $O(N^{\alpha_i+2})$ by (\ref{can-asym}). Additionally, applying again (\ref{can-asym}) and the mean value theorem to $\varphi_{N,k,i}$, we see that
	\[\int_{|y-\lam_i|<\delta}\sign(x-y)\frac{(\hat{\varphi}_{N,k}(y)-\hat{\varphi}_{N,k}(\lam_i))|y-\lam_i|^{\alpha_i}}{y-\lam_i}dy=O(N^{\alpha_i+2}).\]
	These bounds are additionally uniform in $y$, so that in particular we may conclude that for $k=1,2$ that $\epsilon(\varphi_{N,k})(x)=O(N^{\alpha_i+2})$, uniformly in $x\in \R$. By Remark \ref{l and q remark}, we see that $\epsilon(\ell_iw_N^{1/2})(x)=O(N^{\alpha_i+2})$.
	Again proceeding as in (\ref{I2-method}), we obtain that
	\[\int \cal{K}_N(x,x)^{1/2}\epsilon(\ell_iw_N^{1/2})(x)dx=O(N^{\alpha_i+4}).\]
	In conclusion, we obtain that
	\[e^{-2N\lam_i^2}K_N(\lam_i,\lam_i)(\int \cal{K}_N(x,x)^{1/2}\epsilon(\ell_iw_N^{1/2})(x)dx)^2=O(N^{16(1+\alpha_i)}).\]
	In total these results show that the right-hand side of (\ref{wall}) is of order $O(e^{-Nc/4})$, and that the right-hand side of (\ref{money}) is of order $O(N^{16(1+\sum_{i=1}^{m}\alpha_i)})$, which together completes the proof of Proposition \ref{Norm-Error-Proposition}.
	
	\section{Proof of Proposition \ref{W-Integral-Proposition} \label{Integral Section}}
	
	In this section, we will prove Proposition \ref{W-Integral-Proposition}. The evaluation of these integrals is a key technical step in this paper. Away from the points $\{\pm 1\}$ and $\{\lam_i\}_{i=1}^m$ this computation will follow by standard methods for oscillatory integrals (see for example Lemma \ref{DoubleWatsonsLemma}). The remainder of the work is then spent understanding the contributions around the points $\{\pm 1\}$ and $\{\lam_i\}_{i=1}^m$. This will follow from applications of the classical asymptotics for the Bessel and Airy functions and careful analysis to compare these integrals with those in the remainder of the bulk. There is some similarity with methods used in \cite{Universality,UniversalityEdge,UniversalityLaguerre}, particularly for the asymptotics away from the points $\{\lam_i\}_{i=1}^m$. 
	
	To begin, we first note that we may write
	\[\int J_N^{-1}f(x)g(x)w_N(x)dx=\frac{1}{2}\int  \int\sign(x-y)g(x)f(y)w_N(x)^{1/2}w_N(y)^{1/2}dxdy.\label{main-j-iden}\]
	We observe that if functions $f$ and $g$ are supported on disjoint intervals, then we have that
	\[\int \int\sign(x-y)g(x)f(y)w_N(x)^{1/2}w_N(y)^{1/2}dxdy=\pm \int g(x)w_N(x)^{1/2}dx\int f(y)w_N(y)^{1/2}dy.\label{disjoint-support}\]
	
	This will be an important observation. With this identity and a partition of unity, we will reduce the integrals into a sequence of simpler integrals supported over regions where a single asymptotic expression holds. We again fix $(\cal{W},\{\lam_i\}_i,\{\alpha_i\}_i)$ satisfying the conditions of Proposition \ref{main-theorem}, and in all results below, the error will be uniform in a neighborhood of $\{\lam_i\}\in \{(\lam_1,\cdots, \lam_m)\in (-1,1)^m:\lam_1>\cdots >\lam_m\}$. We will employ, as before, the notation $\lam_0=1$ and $\lam_{m+1}=-1$.
	
	\begin{lem}
		\label{Bulk Single Integral of pN}
		Assume that $\phi$ is a smooth function on $\R$ of sub-exponential growth. Assume that there is some $\delta>0$ such that $\phi$ vanishes on $(\lam_i-\delta,\lam_i+\delta)$ for all $0\le i\le m+1$. Then for each $k$ and any $c\in \N$, we have that
		\[\int \phi(x)p_{N-k}(x)w_N(x)^{1/2}dx=O(N^{-c}).\]
	\end{lem}
	
	\begin{proof}
		For notational clarity, we will assume that $k=0$. The proof will proceed by applying the results of Section \ref{Asymptotics-Section} on each region. We first note that
		\[\int \phi(x)p_{N}(x)w_N(x)^{1/2}dx=\int_{|x|>1+\delta}\phi(x)p_{N}(x)w_N(x)^{1/2}dx+\int_{|x|<1}\phi(x)p_{N}(x)w_N(x)^{1/2}dx.\]
		We note that by Remark \ref{exp-decay-s} we have that
		\[\int_{|x|>1+\delta}\phi(x)p_{N}(x)w_N(x)^{1/2}dx=O(e^{-NC/2}).\]
		We now focus on the computation of $\int_{|x|<1}\phi(x)p_{N}(x)w_N(x)^{1/2}dx$. If $\phi$ is supported on $(\lam_{i+1}+\delta,\lam_{i}-\delta)$, for some $0\le i\le m$, we may apply Proposition \ref{Asym p_N easy} to rewrite
		\[\int_{|x|<1}\phi(x)p_{N}(x)w_N(x)^{1/2}dx=
		\frac{1}{\pi^{1/2}}\bigg[\int\phi(x)[(I+R_N(x))T_{\infty}(x)]_{11}e^{-i\pi \sum_{k=1}^{i}\alpha_k}e^{iNs(x)}dx+\]\[\int\phi(x)[(I+R_N(x))T_{\infty}(x)]_{12}e^{i\pi \sum_{k=1}^{i}\alpha_k}e^{-iNs(x)}dx\bigg], \label{pN easy term formulae}\]
		where $R_N$ is the error term of Proposition \ref{Asym p_N easy}. We now recall that if $f$ is a smooth function of compact support, then for any $c\in \N$, we have that
		\[|\int e^{-iNx}f(x)dx|\le \frac{1}{N^c}\int |f^{(c)}(x)|dx.\label{integration-by-parts}\]
		This inequality, and a change of variables, will be sufficient to complete the proof. Indeed, let us consider the first term of (\ref{pN easy term formulae}) and denote $f_{N,l}(x)=\phi(x)[(I+R_N(x))T_{\infty}(x)]_{1l}e^{(-1)^li\pi \sum_{k=1}^{i}\alpha_k}$. As we have that for any $l\ge 0$,  $R^{(l)}_N(x)=O(N^{-1})$ uniformly on the support of $\phi$, and recalling that by Remark \ref{det-remark-T}, $T_{\infty}(x)$ is $N$-independent and smooth, we see that for any $c\in \N$, we have that $f_{N,l}^{(c)}(x)=O(1)$ uniformly for $x\in \R$. By (\ref{integration-by-parts}) we see that
		\[|\int f_{N,l}(x) e^{\pm iNs(x)}dx|=|\int (s^{-1})'(y)f_{N,l}(s^{-1}(y))e^{\pm iNy}dy|\le \frac{1}{N^c}\int| \frac{d^c}{dy^c}\left((s^{-1})'(y)f_{N,l}(s^{-1}(y))\right)|dy.\]
		As $s$ is smooth and non-degenerate on the support of $f_N$, we see that the integrand of the right-hand side is uniformly $O(1)$, and supported inside of the compact set $[s(1),s(-1)]=[0,\pi]$, so that terms of (\ref{pN easy term formulae}) are of order $O(N^{-c})$. Applying this argument to each region completes the proof.
	\end{proof}
	
	\begin{lem}
		\label{Bulk double Integral of pN}
		Let $\phi$ and $\varphi$ be smooth functions of subexponential growth on $\R$. Assume that there is some $\delta>0$ such that both $\phi$ and $\varphi$ vanishes on $(\lam_i-\delta,\lam_i+\delta)$ for all $0\le i\le m+1$.
		Then for each $n$ and $m$, we have that
		\[\frac{1}{2}\int \int \phi(x)\varphi(y)p_{N-n}(x)p_{N-m}(y)\Sign(x-y)w_N(x)^{1/2}w_N(y)^{1/2}dxdy=\]\[\frac{1}{2\pi N}\int_{-1}^{1} \frac{\phi(x)\varphi(x)}{(1-x^2)}\sin((m-n)\arcos(x))dx+O(N^{-2}).\]
	\end{lem}
	
	\begin{proof}
		Applying a partition of unity and linearity, we may assume we are in the case where $\phi$ is supported on $(\lam_{i+1}+\delta,\lam_{i}-\delta)$ for some $i$, or on $(1+\delta,\infty)\cup (-\infty,-1-\delta)$, and similarly for $\psi$. In the case where $\psi$ and $\phi$ have disjoint support, we see by (\ref{disjoint-support}) that
		\[|\int \int \phi(x)\varphi(y)p_{N-n}(x)p_{N-m}(y)\Sign(x-y)w_N(x)^{1/2}w_N(y)^{1/2}dxdy|=
		\]\[|\int \phi(x)p_{N-n}(x)w_N(x)^{1/2}dx\int \varphi(y)p_{N-m}(y)w_N(y)^{1/2}dy|.\]
		By Lemma \ref{Bulk Single Integral of pN}, both of these integrals on the right-hand side are $O(N^{-2})$. Thus we may assume that we are in the case where either both $\psi$ and $\phi$ are supported on $(\lam_{i+1}+\delta,\lam_{i}-\delta)$ for some $i$, or on $(1+\delta,\infty)\cup (-\infty,-1-\delta)$.
		In the latter case, by the Cauchy-Schwarz inequality, we have that
		\[|\int\int_{\min(|x|,|y|)>1+\delta}\phi(x)\varphi(y)p_{N-n}(x)p_{N-m}(y)\Sign(x-y)w_N(x)^{1/2}w_N(y)^{1/2}dxdy|\le\]\[(\int_{|x|>1+\delta}p_{N-n}(x)^2\phi(x)^2w_N(x)dx\int_{|y|>1+\delta}p_{N-m}(y)^2\varphi(y)^2w_N(y)dy)^{1/2}\label{druid}\]
		Then applying Remark \ref{exp-decay-s} as in the proof of Lemma \ref{Bulk Single Integral of pN}, we see that the right-hand side of (\ref{druid}) is $O(e^{-NC/4})$. 
		
		Thus we may assume that $\phi$ and $\varphi$ are both supported on $(\lam_{i+1}+\delta,\lam_{i}-\delta)$ for some $i$. To compute this integral, we will use the following lemma (see Chapter 8.4 of \cite{AsymptoticExpansionsOfIntegrals}).
		\begin{lem}
			\label{DoubleWatsonsLemma}
			Let $h$ be a smooth, compactly supported function on $\R$. Then we have that
			\[\frac{1}{2}\int \int h(x,y)e^{\pm iN(x+y)}\sign(x-y)dxdy=O(N^{-2})\]
			\[\frac{1}{2}\int \int h(x,y)e^{\pm iN(x-y)}\sign(x-y)dxdy=\pm \frac{i}{N}\int h(x,x)dx+O(N^{-2}).\]
			More specifically, the error may be bounded by
			$\frac{1}{N^2}(\int\int \|\nabla h(x,y)\|dxdy+4\int|\frac{d}{dx}h(x,x)|dx)$.
		\end{lem}
		
		Using the asymptotics of Proposition \ref{Bulk Asymptotic}, we may write
		\[\frac{1}{2}\int \int \phi(x)\varphi(y)p_{N-n}(x)p_{N-m}(y)\Sign(x-y)w_N(x)^{1/2}w_N(y)^{1/2}dxdy=\]
		\[\frac{1}{2\pi}\sum_{j,k=1,2}\int \int f_{N,j,k}(x,y)e^{-(-1)^jNis(x)-(-1)^kNis(y)}\sign(x-y)dxdy,\]
		where here 
		\[f_{N,j,k}(x,y)=\phi(x)\varphi(y)[(I+R_N(x))T_{\infty}(x)]_{1j}[(I+\bar{R}_N(y))T_{\infty}(y)]_{1k}\times \]\[\exp((-1)^ji(n\arcos(x)+\pi \sum_{l=1}^{i}\alpha_l)+(-1)^ki(m\arcos(y)+\pi \sum_{l=1}^{i}\alpha_l)),\]
		and $R_N$ and $\bar{R}_N$ denote the error terms in Proposition \ref{Asym p_N easy} for the case of $n$ and $m$, respectively. Recalling Remark \ref{det-remark-T} as before, as well as the bound $R^{(c)}_N(x),\bar{R}^{(c)}_N(y)=O(N^{-1})$ for any $c\in \N$, we see that for $j,k\in \{1,2\}$, we have that $f_{N,j,k}(x,y)=O(1)$ and $\|\nabla f_{N,j,k}(x,y)\|=O(1)$, both uniformly in $x,y\in \R$.
		From this, we see that \[\|\nabla ((s^{-1})'(x)(s^{-1})'(y)f_{N,j,k}(s^{-1}(x),s^{-1}(y)))\|=O(1)\label{double-bulk-error}\] uniformly in $x,y\in\R$. Thus applying Lemma \ref{DoubleWatsonsLemma} and a change of variables with respect to $s$, and recalling that $s'(x)<0$ for $x\in (-1,1)$, so that $\sign(s(x)-s(y))=-\sign(x-y)$, we see that
		\[\frac{1}{2\pi}\sum_{j,k=1,2}\int \int f_{N,j,k}(x,y)e^{-(-1)^jNis(x)-(-1)^kNis(y)}\sign(x-y)dxdy=\]
		\[-\frac{1}{2\pi}\sum_{j,k=1,2}\int \int |(s^{-1})'(x)(s^{-1})'(y)|f_{N,j,k}(s^{-1}(x),s^{-1}(y))e^{-(-1)^jNix-(-1)^kNiy}\sign(x-y)dxdy=\]
		\[-\frac{i}{\pi N}\int \frac{f_{N,1,2}(x,x)}{|s'(x)|}dx+\frac{i}{\pi N}\int \frac{f_{N,2,1}(x,x)}{|s'(x)|}dx+O(N^{-2}).\label{bulk-error-ignore}\] We now observe that as $R_N(x),\bar{R}_N(x)=O(N^{-1})$, we have that
		\[f_{N,j,k}(x,y)=\phi(x)\varphi(y)[T_{\infty}(x)]_{1j}[T_{\infty}(y)]_{1k}e^{(-1)^ji(n\arcos(x)+\pi \sum_{l=1}^{i}\alpha_l)+(-1)^ki(m\arcos(y)+\pi \sum_{l=1}^{i}\alpha_l)}+O(N^{-1}).\]
		We may express the right-hand side of (\ref{bulk-error-ignore}) as 
		\[\frac{i}{\pi N}\int \int \frac{\phi(x)\varphi(x)}{|s'(x)|}[T_{\infty}(x)]_{11}[T_{\infty}(x)]_{12}(e^{i(n-m)\arcos(x)}-e^{-i(n-m)\arcos(x)})dx+O(N^{-2}).\]
		Combining this with the identities $[T_{\infty}(x)]_{11}[T_{\infty}(x)]_{12}=\frac{1}{2(1-x^2)^{1/2}}$ and $s'(x)=-2\sq{1-x^2}$, we obtain that
		\[\frac{i}{4\pi N}\int_{-1}^{1} \frac{\phi(x)\varphi(x)}{(1-x^2)}(e^{i(n-m)\arcos(x)}-e^{-i(n-m)\arcos(x)})dx+O(N^{-2})=\]\[\frac{1}{2\pi N}\int_{-1}^{1} \frac{\phi(x)\varphi(x)}{(1-x^2)}\sin((m-n)\arcos(x))dx+O(N^{-2}).\]
		Considered altogether, these results yield the desired claim.
	\end{proof}
	
	We will now begin to compute integrals of functions supported around $\lam_i$ for $1\le i\le m$. It will be convenient to prove a result to match the behavior of the asymptotics of Proposition \ref{Asym p_N easy} and \ref{Asym p_N Bessel} over their overlapping regions.
	
	We recall the asymptotics for the Bessel function (see Chapter 9 of \cite{AS}):
	\[J_{\nu}(x)=\sq{\frac{2}{\pi x}}\bigg[\cos(x-\frac{\nu\pi}{2}-\frac{\pi}{4})+O(x^{-1})\bigg]\label{BesselAsymptotic}.\]
	By employing this expansion, we obtain that with $\cal{J}_{\alpha}$ defined as in (\ref{calJ-def}), we have that,
	\[\cal{J}_{\alpha}(x)=(I+O(x^{-1}))\cal{I}_{\alpha}(x),\quad \cal{I}_{\alpha}(x)=\begin{bmatrix}\sign(x)\sq{2} \sin(|x|-\frac{\alpha\pi}{2})\\ \sq{2}\cos(|x|-\frac{\alpha\pi}{2}) \end{bmatrix},\label{eqn:5.6}\]
	uniformly in $x\in \R$. We will make use of the following matching computation.
	\begin{lem}
		\label{Bessel Matching Lemma}
		For any $q\in \R$ and $x\in (-1,1)$, we have that
		\[T_{N,k,\lam_j}(x)\cal{I}_{\alpha_j+q}((N-k)(s_{N,k}(\lam_j)-s_{N,k}(x)))=\]\[T_{\infty}(x)e^{(-i\pi\sum_{k=1}^{j-1}\alpha_k-i\pi \alpha_jI(x<\lam_j))\sigma}e^{i\sign(x-\lam_j)\frac{\pi}{2} q\sigma}\begin{bmatrix}e^{i(N-k)s_{N,k}(x)}\\ e^{-i(N-k)s_{N,k}(x)} \end{bmatrix}.\label{bessel-matching-lemma-eqn}\]
	\end{lem}
	\begin{proof}
		We observe that as $s_{N,k}$ is monotonically decreasing, if we denote $\pm=\sign(x-\lam_j)=\sign(s_{N,k}(\lam_j)-s_{N,k}(x))$ then
		\[\cal{I}_{\alpha_j+q}((N-k)(s_{N,k}(\lam_j)-s_{N,k}(x)))=\begin{bmatrix}\sq{2}\sin((N-k)(s_{N,k}(\lam_j)-s_{N,k}(x))\mp(\alpha_j+q)\frac{\pi}{2})\\ \sq{2}\cos((N-k)(s_{N,k}(\lam_j)-s_{N,k}(x))\mp (\alpha_j+q)\frac{\pi}{2})\end{bmatrix}.\]
		We also note that for any $a,b\in \R$, we have that
		\[e^{-i\frac{\pi}{4}}e^{-i\frac{\pi}{4}\sigma}\frac{1}{\sq{2}}\begin{bmatrix}1&i\\i&1\end{bmatrix}\begin{bmatrix}\sq{2}\sin(b-a)\\ \sq{2}\cos(b-a)\end{bmatrix}=e^{-i\frac{\pi}{4}}e^{-i\frac{\pi}{4}\sigma}\begin{bmatrix}ie^{i(a-b)}\\e^{-i(a-b)}\end{bmatrix}=e^{ia\sigma}\begin{bmatrix}e^{-ib}\\e^{ib}\end{bmatrix},\]
		so that we see that
		\[e^{-i\frac{\pi}{4}}e^{-i\frac{\pi}{4}\sigma}\frac{1}{\sq{2}}\begin{bmatrix}1&i\\i&1\end{bmatrix}\cal{I}_{\alpha_j+q}((N-k)(s_{N,k}(\lam_j)-s_{N,k}(x)))=e^{\pm \frac{i\pi}{2}(\alpha_j+q)\sigma}e^{-i(N-k)s_{N,k}(\lam_j)\sigma}\begin{bmatrix}e^{i(N-k)s_{N,k}(x)}\\e^{-i(N-k)s_{N,k}(x)}\end{bmatrix}.\label{secret-st}\]
		Multiplying both sides of (\ref{secret-st}) by $T_{\infty}(x)e^{i\frac{\rho_{N,k,j}}{2}\sigma}$, and recalling the definitions of $T_{N,k,\lam_j}$ given in (\ref{T_lam-def}), we see that the left-hand and right-hand sides simplify to the left-hand and right-hand sides of (\ref{bessel-matching-lemma-eqn}). 
	\end{proof}
	We will also need the following elementary result, which will follow by integration by parts.
	\begin{lem}
		\label{integration-by-parts-lemma}
		Let $f,g:[0,\infty)\to [0,\infty)$ be monotone increasing smooth bijections, with everywhere positive derivatives on $(0,\infty)$. Let $h,l$ be smooth functions of compact support and let $F,G$ be continuous functions, all defined on $[0,\infty)$. Then we have that
		\[\frac{1}{2}\int_0^{\infty} \int_0^{\infty} h(x)l(y)F(f(x))G(g(y))\sign(x-y)dxdy=I+II+III,\label{integration-by-parts-formula}\]
		where here
		\[I=\int_0^{\infty}F(f(x))(\int_0^{g(x)}G(z)dz)\frac{h(x)l(x)}{g'(x)}dx,\]
		\[II=\int_0^{\infty}(\int_0^{f(x)}F(z)dz)(\int_0^{g(x)}G(z)dz)\frac{h(x)}{f'(x)}\frac{d}{dx}(\frac{l(x)}{g'(x)})dx,\]
		\[III=\frac{1}{2}\int_0^{\infty} \int_0^{\infty}(\int_0^{f(x)}F(z)dz)(\int_0^{g(y)}G(z)dz)\frac{d}{dx}(\frac{h(x)}{f'(x)})\frac{d}{dy}(\frac{l(y)}{g'(y)})\sign(x-y)dxdy.\]
	\end{lem}
	\begin{proof}
		We first introduce the functions $\bar{h}(x)=h(f^{-1}(x))(f^{-1})'(x)$ and $\bar{l}(x)=l(g^{-1}(x))(g^{-1})'(x)$, so that we may write the left-hand side of (\ref{integration-by-parts-formula}) as
		\[\frac{1}{2}\int_0^{\infty} \int_0^{\infty} \bar{h}(u)\bar{l}(v)F(u)G(v)\sign(f^{-1}(u)-g^{-1}(v))dudv.\label{xan-ig}\]
		We observe that by monotonicity of $g$ we have that $\sign(f^{-1}(u)-g^{-1}(v))=\sign(g(f^{-1}(u))-v)$, and furthermore by integration by parts we see that for $u>0$
		\[\frac{1}{2}\int_0^{\infty}\bar{l}(v)G(v)\sign(g(f^{-1}(u))-v)dv=\frac{1}{2}\int_0^{g(f^{-1}(u))}\bar{l}(v)G(v)dv-\frac{1}{2}\int_{g(f^{-1}(u))}^{\infty}\bar{l}(v)G(v)dv=\]
		\[\bar{l}(g(f^{-1}(u)))(\int_0^{g(f^{-1}(u))}G(z)dz)-\frac{1}{2}\int_0^{\infty}\bar{l}'(v)(\int_0^{v}G(z)dz)\sign(g(f^{-1}(u))-v)dv.\]
		Applying this, we see that (\ref{xan-ig}) may be rewritten as the sum $I'+II'$ where
		\[I'=\int_0^{\infty}\bar{h}(u)F(u)\bar{l}(g(f^{-1}(u)))(\int_0^{g(f^{-1}(u))}G(z)dz)du,\]
		\[II'=-\frac{1}{2}\int_0^{\infty}\int_0^{\infty}\bar{h}(u)F(u)\bar{l}'(v)(\int_0^{v}G(z)dz)\sign(g(f^{-1}(u))-v)dudv.\]
		Again applying a change of variables with respect to $f$, and observing that $\bar{l}(g(y))=l(y)/g'(y)$, we see that $I=I'$. Similarly, noting that \[\bar{l}'(g(y))=\frac{l'(y)}{(g'(y))^2}-\frac{l(y)g''(y)}{(g'(y))^3}=\frac{1}{g'(y)}\frac{d}{dy}\left(\frac{l(y)}{g'(y)}\right)\]
		we see that
		\[II'=-\frac{1}{2}\int_0^{\infty}\int_0^{\infty}F(f(x))(\int_0^{g(y)}G(z)dz)h(x)\frac{d}{dy}(\frac{l(y)}{g'(y)})\sign(x-y)dxdy.\]
		By applying essentially the same argument, integrating by parts with respect to $x$ instead of $y$, we may show that $II'=II+III$. More specifically, we may observe that if we swap the dummy-variables $(x,y)$ in $II'$, then $II'$ is just the left-hand side of (\ref{integration-by-parts-formula}) with a modified set of functions, so repeating the above argument, one may indeed confirm that $II'=II+III$.
	\end{proof}
	We are now able to state the appropriate generalizations of the above lemmas to the case of functions supported around $\lam_i$. 
	\begin{lem}
		\label{both Bessel Lemma}
		For each $1\le i\le m$, there exists $\delta>0$, such that if $\phi$ and $\varphi$ are smooth
		functions supported on $(\lam_i-\delta,\lam_i+\delta)$, then we have for any $n$ and $m$ that
		\[\int \phi(x)p_{N-n}(x)w_N(x)^{1/2}dx=O(N^{-1}),\]
		\[\frac{1}{2}\int\int  \phi(x)\varphi(y)p_{N-n}(x)p_{N-m}(y)\sign(x-y)w_N(x)^{1/2}w_N(y)^{1/2}dxdy=\]\[\frac{1}{2\pi N}\int_{-1}^{1} \frac{\phi(x)\varphi(x)}{(1-x^2)}\sin((m-n)\arcos(x))dx+O(N^{-2}).\]
	\end{lem}
	
	\begin{proof}
		For convenience, let us denote $\lam=\lam_i$ and $\alpha=\alpha_i$ for this proof. We choose $\delta>0$ so the asymptotics of Proposition \ref{Asym p_N Bessel} hold on $(\lam-2\delta,\lam+2\delta)$. We recall the following lemma due to \cite{BesselInt}.
		\begin{lem}
			\label{SingleBesselLemma}
			Let $h$ be a smooth, compactly supported function, and assume that $\mu+\nu>-1$. Then we have that
			\[\int_{0}^{\infty}J_{\nu}(Nt)t^{\mu}h(t)dt-D(\nu,\mu)\frac{h(0)}{N^{\mu+1}}=O(N^{-\mu-2}),\;\;\;D(\nu,\mu)=2^{\mu}\frac{\Gamma((\mu+\nu+1)/2)}{\Gamma((\nu-\mu+1)/2)}.\label{SingleBesselEqn}\]
			More specifically, the error is bounded by:
			$\frac{C_{\mu,\nu}}{N^{\mu+2}}(\sum_{i=0}^{\ceil{\mu+3/2}}\int_0^{\infty} |h^{(i)}(x)|dx)$
			for some absolute constant $C_{\mu,\nu}>0$.
		\end{lem}
		\noindent
		From Proposition \ref{Asym p_N Bessel}, we may write
		\[\phi(x)p_{N-n}(x)w_N(x)^{1/2}=h_{N,1}(x)\sign(x-\lam)\sq{\pi |(N-n)f_{N,n,\lam}(x)|}J_{\alpha+1/2}((N-n)f_{N,n,\lam}(x))+\]\[h_{N,2}(x)\sq{\pi |(N-n)f_{N,n,\lam}(x)|}J_{\alpha-1/2}((N-n)f_{N,n,\lam}(x)),\]
		where here $h_{N,k}(x)=\phi(x)\frac{1}{\pi^{1/2}}[(I+R_N(x))T_{N,n,\lam}(x)]_{1k}$. We recall from Remark \ref{T_lam-size-remark} that for $c\in \N$ and $i,j\in \{1,2\}$, we have that $[T_{N,n,\lam}^{(c)}(x)]_{ij}=O(1)$ uniformly in $x\in (\lam-\delta,\lam+\delta)$. As we also have that $R^{(c)}(x)=O(N^{-1})$ uniformly, we see that $h_{N,k}^{(c)}(x)=O(1)$ for $c\in \N$. From this, we see from Lemma \ref{SingleBesselLemma} and a change of coordinates that there is $C$, only dependant on $\alpha$, such that
		\[|\int_{\lam}^{\lam\pm \delta}\phi(x)p_{N-n}(x)w_N(x)^{1/2}dx|\le N^{-1}\pi^{1/2}(D(\alpha+1/2,1/2)|h_{N,1}(\lam)|+\]\[D(\alpha-1/2,1/2)|h_{N,2}(\lam)|)|(f_{N,n,\lam}^{-1})'(0)|+\frac{C}{N^2}\sum_{k=1,2}\sum_{l=0}^{2}\int_0^{\infty}|\frac{d^l}{dx^{l}}\left((f_{N,n,\lam}^{-1})'(x)h_{N,k}(f_{N,n,\lam}^{-1}(x))\right)|dx.\label{floyd}\]
		We see as above that the integrands on the right-hand side of (\ref{floyd}) are $O(1)$, so we see that 
		\[\int_{\lam}^{\lam\pm \delta}\phi(x)p_{N-n}(x)w_N(x)^{1/2}dx=O(N^{-1}),\label{Half-Single-Integral-Bessel}\]
		which establishes the first claim.
		
		The case of double integrals will require more care. We note that by (\ref{disjoint-support}) and (\ref{Half-Single-Integral-Bessel}), we have that
		\[\frac{1}{2}\int_{\lam}^{\lam \pm \delta} \int_{\lam}^{\lam\mp \delta} \phi(x)\varphi(y)p_{N-n}(x)p_{N-m}(y)\sign(x-y)w_N(x)^{1/2}w_N(y)^{1/2}dxdy=O(N^{-2}),\]
		as we may split this double integral into a product of the single variable integrals dealt with above. Thus by breaking the double-integral into four regions, we see that
		\[\frac{1}{2}\int \int \phi(x)\varphi(y)p_{N-n}(x)p_{N-m}(y)\sign(x-y)w_N(x)^{1/2}w_N(y)^{1/2}dxdy=\]\[\frac{1}{2}\int_{\lam}^{\infty}\int_{\lam}^{\infty} \phi(x)\varphi(y)p_{N-n}(x)p_{N-m}(y)\sign(x-y)w_N(x)^{1/2}w_N(y)^{1/2}dxdy+\]\[\frac{1}{2}\int^{\lam}_{-\infty}\int^{\lam}_{-\infty}\phi(x)\varphi(y)p_{N-n}(x)p_{N-m}(y)\sign(x-y)w_N(x)^{1/2}w_N(y)^{1/2}dxdy+O(N^{-2}).\label{m4}\]
		
		We compute the first integral, with the other case being identical. For notational ease, we will denote $g_{N,k}(x)=(N-k)f_{N,k,\lam}(x)$. In view of Proposition \ref{Asym p_N Bessel}, we further denote $h_{N,i}(x)=\phi(x)[(I+R_N(x))T_{N,n,\lam}(x)]_{1i}$ and $l_{N,i}(x)=\varphi(x)[(I+\bar{R}_N(x))T_{N,m,\lam}(x)]_{1i}$, where $R_N$ and $\bar{R}_N$ denote the error terms in Proposition \ref{Asym p_N Bessel} for case of $k=n$ and $k=m$, respectively. Thus by employing Proposition \ref{Asym p_N Bessel}, we may rewrite the first integral on the right-hand side of (\ref{m4}) as
		\[\frac{1}{2\pi}\sum_{k,l=1,2}\int_{\lam}^{\infty}\int_{\lam}^{\infty}h_{N,k}(x)[\cal{J}_{\alpha}(g_{N,n}(x))]_kl_{N,l}(y)[\cal{J}_{\alpha}(g_{N,m}(y))]_l\sign(x-y)dxdy.\label{Oct1}\]
		We now observe that $f_{N,k,\lam}'(x)=2\eta^{-1}_k\sq{1-\eta_k^{-2}x^2}$, and that $f_{N,k,\lam}(\lam)=0$. Thus by Lemma \ref{integration-by-parts-lemma} we see that for $k$ fixed and $N$ sufficiently large,  we may rewrite (\ref{Oct1}) as $I+II+III$ where
		\[I=\frac{1}{\pi}\sum_{k,l=1,2}\int_\lam^{\infty} [\cal{J}_{\alpha}(g_{N,n}(x))]_k(\int_\lam^{g_{N,m}(x)}[\cal{J}_{\alpha}(z)]_ldz)\frac{h_{N,k}(x)l_{N,l}(x)}{g_{N,m}'(x)}dx,\]
		\[II=\frac{1}{\pi}\sum_{k,l=1,2}\int_\lam^{\infty} (\int _\lam ^{g_{N,n}(x)}[\cal{J}_{\alpha}(z)]_kdz)(\int_\lam^{g_{N,m}(x)}[\cal{J}_{\alpha}(z)]_ldz)\frac{h_{N,k}(x)}{g_{N,n}'(x)}\frac{d}{dx}(\frac{l_{N,l}(x)}{g_{N,m}'(x)})dx,\]
		\[III=\frac{1}{2\pi}\sum_{k,l=1,2}\int_\lam^{\infty}\int_\lam^{\infty} (\int _\lam^{g_{N,n}(x)}[\cal{J}_{\alpha}(z)]_kdz)(\int_\lam^{g_{N,m}(y)}[\cal{J}_{\alpha}(z)]_ldz)\times\]\[(\frac{d}{dx}\frac{h_{N,k}(x)}{g_{N,n}'(x)})(\frac{d}{dy}\frac{l_{N,l}(y)}{g_{N,m}'(y)})\sign(x-y)dxdy.\]
		
		We now recall another important asymptotic for understanding these terms. Assuming that $\ell+\nu>-1$, we may define
		\[J_{\nu,\ell}(x)=\int_0^{x}y^{\ell}J_{\nu}(y)dy,\quad  J_{\nu,\ell}(x)=D(\nu,\ell)-I_{\nu,\ell}(x).\]
		We have the following asymptotics (see Chapter 2 of \cite{BesselInt})
		\[I_{\nu,\ell}(z)=z^{\ell-1/2}\sq{\frac{2}{\pi}}\cos(z-\frac{\nu \pi}{2}+\frac{\pi}{4})+O(z^{\ell-3/2}),\quad z>1;\quad I_{\nu,\ell}(z)=O(1),\quad z\le 1 \label{IntegralAsymptotic}.\]
		From these asymptotics, we see that for $k,l=1,2$ \[\int_\lam^{g_{N,n}(x)}[\cal{J}_{\alpha}(z)]_kdz=O(1),\;\; \int_\lam^{g_{N,m}(x)}[\cal{J}_{\alpha}(z)]_ldz=O(1).\label{ig-ig2}\]
		As $h_{N,k}(x),l_{N,l}(x),h_{N,k}'(x),l_{N,l}'(x)=O(1)$ uniformly, and as $g_{N,n}(x)=(N-k)f_{N,n,\lam}(x)$, we thus see that $II,III=O(N^{-2})$.
		
		We may further rewrite $I$ as the sum of two terms
		\[I'=-\sum_{k,l=1,2}\int_\lam^{\infty} \sq{g_{N,n}(x)}J_{\alpha-(-1)^k1/2}(g_{N,n}(x))I_{\alpha-(-1)^l1/2,1/2}(g_{N,m}(x))\frac{h_{N,k}(x)l_{N,l}(x)}{g_{N,m}'(x)}dx\]
		\[II'=\sum_{k,l=1,2}D(\alpha-(-1)^l1/2,1/2)\int_\lam^{\infty} J_{\alpha-(-1)^k1/2}(g_{N,n}(x))\frac{h_{N,k}(x)l_{N,l}(x)}{g_{N,m}'(x)}dx.\]
		We have that $II'=O(N^{-2})$ by Lemma \ref{SingleBesselLemma}. Finally, employing the asymptotics for $J_\mu$ and $I_{\mu,\ell}$ given by (\ref{IntegralAsymptotic}) and (\ref{eqn:5.6}) one derives that
		\[I'=-\frac{2}{\pi}\sum_{k,l=1,2}\int_{\lam}^{\infty} \cos(g_{N,n}(x)-\frac{\alpha\pi}{2}+\frac{((-1)^k-1)\pi }{4})\times\]\[\cos(g_{N,m}(x)-\frac{\alpha\pi}{2}+\frac{((-1)^l+1)\pi }{4})\frac{h_{N,k}(x)l_{N,l}(x)}{g_{N,m}'(x)}dx+O(N^{-2})=\]\[-\sum_{k,l=1,2}\frac{1}{\pi}\int_{\lam}^{\infty} [\cal{I}_\alpha(g_{N,n}(x))]_k[\cal{I}_{\alpha-1}(g_{N,m}(x))]_l\frac{h_{N,k}(x)l_{N,l}(x)}{g_{N,m}'(x)}dx+O(N^{-2}).\label{bessel-pt-1}\]
		From the error bounds on $R_N$ and $\bar{R}_N$, we see that  $h_{N,i}(x)=\phi(x)[T_{N,n,\lam}(x)]_{1i}+O(N^{-1})$ and $l_{N,i}(x)=\varphi(x)[T_{N,m,\lam}(x)]_{1i}+O(N^{-1})$, uniformly. Applying these, we can further write $I'$ as 
		\[-\frac{1}{\pi(N-m)}\sum_{k,l=1,2}\int_{\lam}^{\infty} [\cal{I}_\alpha(g_{N,n}(x))]_k[\cal{I}_{\alpha-1}(g_{N,m}(x))]_l\frac{\phi(x)
			\varphi(x)[T_{N,n,\lam}(x)]_{1k}[T_{N,m,\lam}(x)]_{1l}}{f_{N,m,\lam}'(x)}dx+O(N^{-2}).\label{bessel-pt-2}\]
		Recalling that $g_{N,k}(x)=(N-k)f_{N,k,\lam}(x)=(N-k)(s_{N,k}(\lam)-s_{N,k}(x))$, we see that by Lemma \ref{Bessel Matching Lemma}
		\[I'=-\frac{1}{\pi(N-m)}\int_{\lam}^{\infty} [T_{\infty}(x)e^{-i\pi\sum_{p=1}^{j-1}\alpha_p\sigma}\begin{bmatrix}e^{i(N-n)s_{N,n}(x)}\\ e^{-i(N-n)s_{N,n}(x)}\end{bmatrix}]_1\times\]\[[T_{\infty}(x)e^{-i\pi\sum_{p=1}^{j-1}\alpha_p\sigma}e^{-i\frac{\pi}{2}\sigma}\begin{bmatrix}e^{i(N-m)s_{N,m}(x)}\\ e^{-i(N-m)s_{N,m}(x)}\end{bmatrix}]_1\frac{\phi(x)\varphi(x)}{|s_{N,m}'(x)|}dx+O(N^{-2}).\label{bessel-pt-3}\]
		Employing Lemma \ref{sk expansion lem}, we may now expand the integral on the right-hand side of (\ref{bessel-pt-3}) as 
		\[-\frac{i}{\pi N}\sum_{k,l=1,2}\int_{\lam}^\infty e^{-(-1)^kNis(x)-(-1)^lNis(x)}\frac{f_{k,l}(x)}{|s'(x)|}dx+O(N^{-2}),\label{bessel-5}\]
		where here
		\[f_{k,l}(x)=\phi(x)e^{(-1)^k in\arcos(x)}[T_{\infty}(x)]_{1k}e^{(-1)^ki\pi\sum_{p=1}^{j-1}\alpha_p}\varphi(x)(-1)^{l}e^{(-1)^l im\arcos(x)}[T_{\infty}(x)]_{1l}e^{(-1)^l i\pi\sum_{p=1}^{j-1}\alpha_p}.\]
		We see that by integration by parts (\ref{integration-by-parts}) that the terms with $l=k$ in (\ref{bessel-5}) are $O(N^{-2})$. Thus (\ref{bessel-5}) may further be written as
		\[-\frac{i}{\pi N}\int_{\lam}^{\infty} \bigg(\frac{f_{1,2}(x)}{|s'(x)|}+\frac{f_{2,1}(x)}{|s'(x)|}\bigg)dx+O(N^{-2})=\]\[-\frac{i}{\pi N}\int_{\lam}^{\infty} \phi(x)\varphi(x)(e^{-i(n-m)\arcos(x)}-e^{i(n-m)\arcos(x)})\frac{[T_{\infty}(x)]_{11}[T_{\infty}(x)]_{12}}{|s'(x)|}dx+O(N^{-2}).\]
		The identities at the end of Lemma \ref{Bulk double Integral of pN} now show that this coincides with
		\[I'=\frac{1}{2\pi N}\int_{\lam}^{\infty} \frac{\phi(x)\varphi(x)}{(1-x^2)}\sin((m-n)\arccos(x))dx+O(N^{-2})\]
		This completes the evaluation of the first integral on the right-hand side of (\ref{m4}). The second integral is similarly given by
		\[\frac{1}{2\pi N}\int_{-\infty}^{\lam} \frac{\phi(x)\varphi(x)}{(1-x^2)}\sin((m-n)\arccos(x))dx+O(N^{-2}).\]
		Together these computations complete the proof of the second claim.
	\end{proof}
	
	To compute the integrals supported around $\{\pm 1\}$ we will need a lemma similar to Lemma \ref{Bessel Matching Lemma}. We recall that for $x>0$ we have that (Chapter 10 of \cite{AS})
	\[\Ai(-x)=\frac{1}{\sq{\pi} x^{1/4}}(\sin(\frac{2}{3}x^{3/2}+\pi/4)+O(x^{-3/2})),\;\;\Ai'(-x)=-\frac{x^{1/4}}{\sq{\pi}}(\cos(\frac{2}{3}x^{3/2}+\pi/4)+O(x^{-3/2})).\label{airy-abe}\]
	Similarly to Lemma \ref{Bessel Matching Lemma}, we have the following matching formula.
	\begin{lem}
		\label{Airy Matching}
		We have for $|x|\le 1$ that
		\[\frac{1}{\sq{\pi}}T_{1}(x)\begin{bmatrix}|f_{1}(x)|^{-1/4}\sin(a+\pi/4)\\
			-|f_{1}(x)|^{1/4}\cos
			(a+\pi/4)
		\end{bmatrix}=T_{\infty}(x)\begin{bmatrix}e^{ia}\\ e^{-ia} \end{bmatrix},\]
		and that
		\[\frac{1}{\sq{\pi}}T_{-1}(x)\begin{bmatrix}|f_{-1}(x)|^{-1/4}\sin(a+\pi/4)\\
			|f_{-1}(x)|^{1/4}\cos(a+\pi/4)
		\end{bmatrix}=T_{\infty}(x)e^{-i\pi\sum_{j=1}^{m}\alpha_j\sigma}\begin{bmatrix}e^{ia}\\ -e^{-ia} \end{bmatrix}.\]
	\end{lem}
	\begin{proof}
		We observe that
		\[e^{-i\frac{\pi}{4}}e^{i\frac{\pi}{4}\sigma}\begin{bmatrix}1&-1\\1&1\end{bmatrix}e^{i\frac{\pi}{4}\sigma}\begin{bmatrix}\sin(a+\pi/4)\\-\cos(a+\pi/4)
		\end{bmatrix}=e^{i\frac{\pi}{4}}\begin{bmatrix}1&i\\-i&-1\end{bmatrix}\begin{bmatrix}\sin(a+\pi/4)\\-\cos(a+\pi/4)
		\end{bmatrix}=\begin{bmatrix}e^{ia}\\
			e^{-ia}
		\end{bmatrix}.\]
		Applying $T_{\infty}$ and recalling the definition of $T_1$ given in (\ref{eqn:T1-def}) then achieves the first equality. The proof of the result for $T_{-1}$ follows similarly.
	\end{proof}
	
	We will also need the following computation.
	\begin{lem}
		\label{T1 lemma}
		We have that
		\[[T_1(1)]_{11}=\sq{2\pi},\;\;[T_{-1}(-1)]_{11}=\sq{2\pi}.\]
	\end{lem}
	\begin{proof}
		We note that $D(x)=\exp(\cal{W}(x))(1+O(|x-1|^{1/2}))$ (see the proof of Lemma 4.15 in \cite{GUEGMC}) and as in addition $f_1'(1)=2$, we see that for $x>1$
		\[T_1(x)=\sq{\pi}e^{-i\frac{\pi}{4}}A(x)e^{i\frac{\pi}{4} \sigma}\begin{bmatrix}1&-1\\
			1&1\end{bmatrix}2^{\sigma/4}(x-1)^{\sigma/4}+O((x-1)^{1/4}).\]
		We note that for $x>1$ we have that
		\[A(x)=\frac{1}{2^{3/4}(x-1)^{1/4}}\begin{bmatrix}1& i\\-i&1 \end{bmatrix}+\frac{(x-1)^{1/4}}{2^{5/4}}\begin{bmatrix}1& -i\\i&1 \end{bmatrix}+O((x-1)^{3/4}).\]
		
		From this, one may routinely obtain that
		\[\sq{\pi}e^{-i\frac{\pi}{4}}A(x)e^{i\frac{\pi}{4} \sigma}\begin{bmatrix}1&-1\\
			1&1\end{bmatrix}2^{\sigma/4}=\sq{2\pi}(x-1)^{-1/4}\begin{bmatrix}1& 0\\
			-i&0\end{bmatrix}+i\sq{\frac{\pi}{2}}(x-1)^{1/4}\begin{bmatrix}0& -1\\
			0&-i\end{bmatrix}+O((x-1)^{3/4}).\]
		Together these are sufficient to establish that $[T_{1}(1)]_{11}=\sq{2\pi}$. The case of $-1$ is similar.
	\end{proof}
	
	With these results established, we are now able to compute the relevant integrals in the regions around $\{\pm 1\}$. This result will follow from similar methods to Lemma \ref{both Bessel Lemma}, though we will contend with a variety of complications as the integrand switches from being oscillatory to exponentially decaying. Particularly, this transition will cause a variety of boundary terms to appear, which requires us to carefully analyze each term carefully to get the exact contribution. In addition, the non-smooth nature of many of the functions composing $T_{\pm 1}$ will require us to carefully Taylor expand many quantities to show cancellation. 
	
	We note that we will only need this level of detail for one integral of Proposition \ref{W-Integral-Proposition}, which is recalled as (\ref{ig-r2}) below, and showing simply that (\ref{airy-2-lem}) is simply $O(N^{-1})$ is significantly simpler.

	\begin{lem}
		\label{both airy lemma}
		There is $\delta>0$, such that if $\phi$ and $\varphi$ are smooth functions supported on $(\pm 1-\delta,\pm 1+\delta)$, then we have for any $n$ and $m$ that
		\[\int \phi(x)p_{N-n}(x)w_N(x)^{1/2}dx=\frac{1}{(2N)^{1/2}}(\pm 1)^{N-n}\phi(\pm 1)+O(N^{-5/6}),\label{airy-1-lem}\]
		\[\frac{1}{2}\int\int  \phi(x)\varphi(y)p_{N-n}(x)p_{N-m}(y)\sign(x-y)w_N(x)^{1/2}w_N(y)^{1/2}dxdy=\]\[\frac{1}{2\pi N}\int_{-1}^{1} \frac{\phi(x)\varphi(x)}{(1-x^2)}\sin((m-n)\arcos(x))dx+O(N^{-7/6}).\label{airy-2-lem}\]
	\end{lem}
	
	\begin{proof}
		Fix $\delta$ as in Lemma \ref{both Bessel Lemma}. We demonstrate the case of $+1$, the case of $-1$ being identical. We first recall that there is $C$ such that $|\Ai(x)|,|\Ai'(x)|\le Ce^{-x}$ (see Chapter 10 of \cite{AS}). From the asymptotics of Proposition \ref{Asym p_N Airy} we see that
		\[\int_1^{\infty}\phi(x)p_{N-n}(x)w_N(x)^{1/2}dx=\frac{1}{\pi^{1/2}}\int_{0}^{\infty}(N-n)^{1/6}h_{N,1}(x)\Ai((N-n)^{2/3}x)dx+\]
		\[\frac{1}{\pi^{1/2}}\int_{0}^{\infty}(N-n)^{-1/6}h_{N,2}(x)\Ai'((N-n)^{2/3}x)dx,\label{RHS-airy-eq}\]
		where here $h_{N,i}(x)=(f_{N,n,1}^{-1})'(x)[(I+R_N(f_{N,n,1}^{-1}(x)))T_1(f_{N,n,1}^{-1}(x))]_{1i}\phi(f_{N,n,1}^{-1}(x))$. Noting that $h_{N,2}$ is uniformly bounded on $(1,\infty)$, we see that by the exponential bound on $\Ai'$ the second integral on the right-hand side of (\ref{RHS-airy-eq}) is $O(N^{-5/6})$, so we may focus on the first integral on the right-hand side. As $h_{N,1}(x)$ is smooth and of compact support, we see that there is $C$ such that $|h_{N,1}(x)-h_{N,1}(0)|\le C|x|$. In particular,
		\[|\int_{0}^{\infty}(h_{N,1}(x)-h_{N,1}(0))\Ai((N-n)^{2/3}x)dx|\le C\int^{\infty}_{0}|x\Ai((N-n)^{2/3}x)|dx.\]
		Using that $\Ai$ is subexponential, we see there are $C>0$ such that
		\[\int^{\infty}_{0}|x\Ai((N-n)^{2/3}x)|dx\le C(N-n)^{-4/3}\int_0^{\infty}xe^{-x}dx=O(N^{-4/3}),\]
		so that
		\[\int_{0}^{\infty}h_{N,1}(x)\Ai((N-n)^{2/3}x)dx=h_{N,1}(0)\int_{0}^{\infty}\Ai((N-n)^{2/3}x)dx+O(N^{-4/3}).\]
		As $f'_1(1)=2$, we see that $h_{N,1}(0)=2^{-1}\phi(1)[T_1(1)]_{11}+O(N^{-1})$, and thus (\ref{RHS-airy-eq}) may be rewritten as
		\[\frac{1}{2\pi^{1/2}}N^{-1/2}\phi(1)[T_{1}(1)]_{11}\int_0^{\infty}\Ai(x)dx+O(N^{-5/6}).\]
		Applying Lemma \ref{T1 lemma} and the fact that $\int_0^{\infty}\Ai(x)dx=\frac{1}{3}$ (see 10.4.82 of \cite{AS}) we see that we may further reduce this to
		\[\phi(1)\frac{1}{(2N)^{1/2}}(\frac{1}{3})+O(N^{-5/6})\label{Airy-Single-Integral-Outer}.\]
		To compute the integral in the region $(1-\delta,1)$ we observe that we have that
		\[\Ai(-x)=\frac{\sq{x}}{3}(J_{1/3}(\frac{2}{3}x^{3/2})+J_{-1/3}(\frac{2}{3}x^{3/2})),\label{Airy-Bessel-1}\]\[\Ai'(-x)=\frac{x}{3}(J_{2/3}(\frac{2}{3}x^{3/2})-J_{-2/3}(\frac{2}{3}x^{3/2})),\label{Airy-Bessel-2}\]
		which occur as (10.4.15) and (10.4.17) of \cite{AS}. By applying Lemma \ref{SingleBesselLemma} to the terms in this expression, we obtain that for smooth, compactly-supported $h$, we have that
		\[\int_0^{\infty} h(t)\Ai(-N^{2/3}t)dt=\frac{h(0)}{N^{2/3}}\frac{2}{3}+O(N^{-4/3}),\;\;\;\int_0^{\infty} h(t)\Ai'(-N^{2/3}t)dt=\frac{h(0)}{3^{2/3}\Gamma(2/3)N^{2/3}}+O(N^{-4/3}),\label{airy-single-int}\]
		where the error can similarly be expressed in terms of the derivatives of $h$. Proceeding similarly to the case of $(1,1+\delta)$, and noting that
		\[\int_{-\infty}^{1}\phi(x)p_{N-n}(x)w_N(x)^{1/2}dx=\frac{\phi(1)}{(2N)^{1/2}}(\frac{2}{3})+O(N^{-5/6}). \label{Airy-Single-Integral-Inner}\]
		Together these yield (\ref{airy-1-lem}).
		
		We now proceed to derive the double integral result. We observe that in view of (\ref{Airy-Single-Integral-Inner}), (\ref{Airy-Single-Integral-Outer}) and (\ref{disjoint-support}), we have that
		\[\int_{\mp \infty}^1\int_1^{\pm\infty}\phi(x)\varphi(y)p_{N-n}(x)p_{N-m}(y)\Sign(x-y)w_N(x)^{1/2}w_N(y)^{1/2}dxdy=\pm\frac{\phi(1)\varphi(1)}{9N}+O(N^{-4/3}),\]
		so that splitting the domain of the double integral into four regions, we see that
		\[\int\int\phi(x)\varphi(y)p_{N-n}(x)p_{N-m}(y)\Sign(x-y)w_N(x)^{1/2}w_N(y)^{1/2}dxdy=\]\[\int_{1}^{\infty}\int_{1}^{\infty}\phi(x)\varphi(y)p_{N-n}(x)p_{N-m}(y)\Sign(x-y)w_N(x)^{1/2}w_N(y)^{1/2}dxdy+\]
		\[\int^{1}_{-\infty}\int^{1}_{-\infty}\phi(x)\varphi(y)p_{N-n}(x)p_{N-m}(y)\Sign(x-y)w_N(x)^{1/2}w_N(y)^{1/2}dxdy+O(N^{-4/3}).\]
		To evaluate these integrals, we will denote $h_{N,i}(x)=[(I+R_N(x))T_1(x)]_{1i}\phi(x)$ and $l_{N,i}(x)=[(I+\bar{R}_N(x))T_1(x)]_{1i}\varphi(x)$, where $R_N$ and $\bar{R}_N$ denote the error terms in Proposition \ref{Asym p_N Airy} for case of $k=n$ and $k=m$, respectively. Let us denote
		\[\cal{A}_{N,k}(x)=\begin{bmatrix}(N-k)^{1/6}\Ai(x)\\ (N-k)^{-1/6}\Ai'(x)\end{bmatrix},\]
		and $g_{N,k}(x)=N^{2/3}f_{N,k,1}(x)$. With this notation, we may write the integral over the region $(1,\infty)^2$ as
		\[\frac{1}{2\pi}\sum_{i,j=1,2}\int_{1}^{\infty}\int_{1}^{\infty}h_{N,i}(x)l_{N,j}(y)[\cal{A}_{N,n}(g_{N,n}(x))]_i[\cal{A}_{N,m}(g_{N,m}(y))]_j\Sign(x-y)dxdy.\label{airy-easy-2}\]
		We note that for $l,k\in \{1,2\}$, there is $C,c>0$ such that for $x\ge 1$, 
		\[|[\cal{A}_{N,l}(g_{N,m}(x))]_{k}|\le CN^{-(-1)^k/6}e^{-cN^{2/3}x}.\]
		Applying this and noting that $h_{N,i},l_{N,i}=O(1)$, we see that we see that $(i,j)$-th term in (\ref{airy-easy-2}) is of order $O(N^{-4/3-(-1)^{i}/6-(-1)^{j}/6})$, so that it suffices to treat the $i=j=1$ term. Arguing with Taylor's theorem as in the single integral case, we see that the integral
		\[\int_{1}^{\infty}\int_{1}^{\infty}h_{N,1}(x)(g_{N,1}(y)-\frac{g_{N,1}(1)}{h_{N,1}(1)}h_{N,1}(y))[\cal{A}_{N,n}(g_{N,n}(x))]_1[\cal{A}_{N,m}(g_{N,m}(y))]_1\sign(x-y)dxdy\]
		is of order $O(N^{-4/3})$. Noting that $|g_{N,n}(x)-g_{N,m}(x)|=O(N^{-1/3})$, we see by the above estimate on $\Ai'$, as well as Taylor's theorem again, that we see that there is $C,c>0$ such that for $x\in(1,1+\delta)$
		\[|[\cal{A}_{N,n}(g_{N,n}(x))]_1-[\cal{A}_{N,m}(g_{N,m}(x))]_1|\le CN^{-1/6}e^{-cN^{2/3}x},\]
		so that
		\[\int_{1}^{\infty}\int_{1}^{\infty}h_{N,1}(x)h_{N,1}(y)[\cal{A}_{N,n}(g_{N,n}(x))]_1([\cal{A}_{N,m}(g_{N,n}(y))]_1-[\cal{A}_{N,m}(g_{N,m}(y))]_1)\sign(x-y)dxdy\]
		is of order $O(N^{-4/3})$. Now noting that by symmetry, we have that
		\[\frac{1}{2}\int_{1}^{\infty}\int_{1}^{\infty}h_{N,1}(x)h_{N,1}(y)[\cal{A}_{N,n}(g_{N,n}(x))]_1[\cal{A}_{N,n}(g_{N,n}(y))]_1\sign(x-y)dxdy=0,\]
		we see combining these results that the $i=j=1$ term in (\ref{airy-easy-2}) is of order $O(N^{-4/3})$, so that in total we have that
		\[\int_{1}^{\infty}\int_{1}^{\infty}\phi(x)\varphi(y)p_{N-n}(x)p_{N-m}(y)\Sign(x-y)w_N(x)^{1/2}w_N(y)^{1/2}dxdy=O(N^{-4/3}).\]
		Combining these to show (\ref{airy-2-lem}), we see that it suffices to show that
		\[\frac{1}{2}\int^{1}_{-\infty}\int^{1}_{-\infty}\phi(x)\varphi(y)p_{N-n}(x)p_{N-m}(y)\Sign(x-y)w_N(x)^{1/2}w_N(y)^{1/2}dxdy=\]\[\frac{1}{2\pi N}\int_{-1}^{1} \frac{\phi(x)\varphi(x)}{(1-x^2)}\sin((m-n)\arcos(x))dx+O(N^{-7/6}).\label{abe-hat}\]
		
		As above, we write the integral over $(-\infty,1)^2$ as
		\[\frac{1}{2\pi}\sum_{i,j=1,2}\int_{-\infty}^{1}\int_{-\infty}^{1}h_{N,i}(x)l_{N,j}(y)[\cal{A}_{N,n}(g_{N,n}(x))]_i[\cal{A}_{N,m}(g_{N,m}(y))]_j\Sign(x-y)dxdy.\label{airy-4-ignore}\]
		
		Except for the case of $i=j=1$, all of these terms may be dealt with essentially as in the proof of Lemma \ref{both Bessel Lemma}. On the other hand, the $i=j=1$ term will require an amount of technical care, as did for the integral over $(1,\infty)^2$. For this reason, instead of employing Lemma \ref{integration-by-parts-lemma} directly, we will need to modify its proof by integrating by parts using $\int_{-y}^{\infty}$ rather than $\int_0^y$. This will be convenient as otherwise, each integral in Lemma \ref{integration-by-parts-lemma} would still be of leading order for the $i=j=1$ term.
		
		To begin, we note that by integration by parts, we may write
		\[\frac{1}{2}\int_{-\infty}^{1}l_{N,j}(y)[\cal{A}_{N,m}(g_{N,m}(y))]_j\Sign(x-y)dy=\frac{l_{N,j}(x)}{g_{N,m}'(x)}(\int_{-g_{N,m}(x)}^{\infty}[\cal{A}_{N,m}(-z)]_jdz)\]
		\[-\frac{1}{2}\frac{l_{N,j}(1)}{g'_{N,m}(1)}(\int_{-g_{N,m}(1)}^{\infty}[\cal{A}_{N,m}(-z)]_jdz)-\frac{1}{2}\int_{-\infty}^1 (\int_{-g_{N,m}(y)}^{\infty}[\cal{A}_{N,m}(-z)]_jdz)\frac{d}{dy}\left(\frac{l_{N,j}(y)}{g'_{N,m}(y)}\right)\sign(x-y)dy.\]
		Employing this, we see that we may write the $(i,j)$-th term of (\ref{airy-4-ignore}) as the sum $I_{ij}+II_{ij}+III_{ij}'$ where
		\[I_{ij}=\frac{1}{\pi}\int_{-\infty}^1 [\cal{A}_{N,n}(g_{N,n}(x))]_i(\int^\infty_{-g_{N,m}(x)}[\cal{A}_{N,m}(-z)]_jdz)\frac{h_{N,i}(x)l_{N,j}(x)}{g_{N,m}'(x)}dx,\label{john}\]
		\[II_{ij}=-\frac{1}{2\pi}(\int^1_{-\infty}[\cal{A}_{N,n}(g_{N,n}(x))]_ih_{N,i}(x)dx)\frac{l_{N,j}(1)}{g_{N,m}'(1)}\int_{-g_{N,m}(1)}^{\infty}[\cal{A}_{N,m}(-z)]_jdz,\]
		\[III'_{ij}=-\frac{1}{2\pi}\int_{-\infty}^1\int_{-\infty}^1[\cal{A}_{N,n}(g_{N,n}(x))]_i(\int^\infty_{-g_{N,m}(y)}[\cal{A}_{N,m}(-z)]_jdz)h_{N,i}(x)\frac{d}{dy}(\frac{l_{N,j}(y)}{g_{N,m}'(y)})\sign(x-y)dxdy.\]
		By a similar integration by parts we may write $III'_{ij}$ itself as a sum $III_{ij}+IV_{ij}+V_{ij}$ with
		\[III_{ij}=\frac{1}{\pi}\int_{-\infty}^1(\int^\infty_{-g_{N,n}(x)}[\cal{A}_{N,n}(-z)]_idz)(\int^\infty_{-g_{N,m}(x)}[\cal{A}_{N,m}(-z)]_jdz)\frac{h_{N,i}(x)}{g_{N,n}'(x)}\frac{d}{dx}(\frac{l_{N,j}(x)}{g_{N,m}'(x)})dx,\]
		\[IV_{ij}=-\frac{1}{2\pi}\frac{h_{N,i}(1)}{g_{N,n}'(1)}\int_{-g_{N,n}(1)}^{\infty}[\cal{A}_{N,n}(-z)]_idz\int_{-\infty}^1 \frac{d}{dy}(\frac{l_{N,j}(y)}{g'_{N,m}(y)})(\int_{-g_{N,m}(y)}^{\infty}[\cal{A}_{N,m}(-z)]_jdz)dy.\]
		\[V_{ij}=\frac{1}{2\pi}\int_{-\infty}^1\int_{-\infty}^1(\int^\infty_{-g_{N,n}(x)}[\cal{A}_{N,n}(-z)]_idz)(\int^\infty_{-g_{N,m}(y)}[\cal{A}_{N,m}(-z)]_jdz)\times \]\[\frac{d}{dx}(\frac{h_{N,i}(x)}{g_{N,n}'(x)})\frac{d}{dy}(\frac{l_{N,j}(y)}{g_{N,m}'(y)})\sign(x-y)dxdy.\]
		We now collect the following asymptotics for $x\in (1,\infty)$ (see \cite{AS})
		\[\sq{\pi}\Ai(-x)=\frac{1}{x^{1/4}}\sin(\frac{2}{3}x^{3/2}+\pi/4)+O(x^{-7/4}),\;\;\;\;\sq{\pi}\Ai'(-x)=-x^{1/4}\cos(\frac{2}{3}x^{3/2}+\pi/4)+O(x^{-5/4}),\]
		\[\sq{\pi}\int^{\infty}_x\Ai(-y)dy=-\frac{1}{x^{3/4}}\sin(\frac{2}{3}x^{3/2}-\frac{\pi}{4})+O(x^{-9/4}),\]\[\sq{\pi}\int_x^{\infty} \Ai'(-y)dy=-\frac{1}{x^{1/4}}\cos(\frac{2}{3}x^{3/2}-\frac{\pi}{4})+O(x^{-7/4}).\label{airy-a}\]
		For now we only take from these that these functions are bounded, so that for fixed $k\in \Z$ and $x\in [1-\delta,1]$ \[[\cal{A}_{N,k}(x)]_l=O(N^{-(-1)^l/6}),\;\;\;\;\int_{-g_{N,k}(x)}^{\infty}[\cal{A}_{N,k}(-z)]_ldz=O(N^{-(-1)^l/6}).\label{ignore-v1}\]
		Employing these bounds, and recalling that $g_{N,k}(x)=O(N^{2/3})$, we see that for any $i,j=1,2$ the integrals $III_{ij},$ $IV_{ij}$ and $V_{ij}$ have order $O(N^{-4/3-(-1)^{i}/6-(-1)^j/6})$. Further employing (\ref{airy-single-int}) we see that
		\[\int_{-\infty}^1 [\cal{A}_{N,n}(g_{N,n}(x))]_ih_{N,i}(x)dx=O(N^{-(-1)^i/6-2/3}).\]
		Thus we see as well that $II_{ij}$ has order $O(N^{-4/3-(-1)^{i}/6-(-1)^j/6})$. In particular, if $(i,j)\neq (1,1)$ all integrals except for $I_{ij}$ are $O(N^{-4/3})$. 
		
		If $i=j=1$ though, none of these bounds are sufficient. Instead we will need to note that for $k\in \Z$ and $x\in (1-\delta,1]$
		\[[\cal{A}_{N,k}(x)]_1=O((1-x)^{-1/4}),\;\;\;\; \int_{-g_{N,k}(x)}^{\infty}[\cal{A}_{N,k}(-z)]_1dz=O(N^{-1/3}(1-x)^{-3/4}).\label{ignore-v2}\]
		Employing (\ref{ignore-v1}) for $x\in (1-N^{-2/3},1]$ and (\ref{ignore-v2}) elsewhere we see that there is $C$ such that
		\[|III_{11}|\le \frac{C}{N^{4/3}}\left(\frac{1}{N^{2/3}}\int_{N^{-2/3}}^\infty \frac{dx}{x^{6/4}}+N^{1/3}\int_0^{N^{-2/3}}dx\right)=\frac{3C}{N^{5/3}}.\]
		Similarly we see that $IV_{11},V_{11}$ are $O(N^{-4/3})$.
		
		Lastly, by applying (\ref{airy-single-int}) to $II_{11}$, we see that
		\[II_{11}=-\frac{N^{1/3}}{2\pi}\frac{h_{N,1}(1)}{g_{N,n}'(1)}\frac{l_{N,1}(1)}{g_{N,m}'(1)}\left(\frac{2}{3}\right)^2+O(N^{-4/3})=-\frac{\phi(1)\varphi(1)}{18\pi N}[T_{1}(1)]_{11}^2+O(N^{-4/3}).\]
		
		Now we are left to consider $I_{ij}$. As before, the case of $i=j=1$ will present an additional subtlety. Let us denote
		\[\cal{S}_{N}(x)=\begin{bmatrix}|f_{N,n}(x)|^{-1/4}\sin((N-n)s_{N,n}(x)+\pi/4)\\-|f_{N,n}(x)|^{1/4}\cos((N-n)s_{N,n}(x)+\pi/4)\end{bmatrix},\]
		\[\cal{C}_{N}(x)=\begin{bmatrix}-|f_{N,m}(x)|^{-1/4}\sin((N-m)s_{N,m}(x)-\pi/4)\\-|f_{N,m}(x)|^{1/4}\cos((N-m)s_{N,m}(x)-\pi/4)\end{bmatrix}.\]
		We observe that by applying the asymptotics of (\ref{airy-a}), for $1-\delta\le x\le 1-N^{-2/3}$ and $i=1,2$
		\[[\cal{A}_{N,n}(g_{N,n}(x))]_i=[\cal{A}_{N,n}((N-n)^{2/3}f_{N,n,1}(x))]_i=\pi^{-1/2}[\cal{S}_N(x)]_i+O(N^{-1}(1-x)^{-7/4}),\] \[\int_{-g_{N,m}(x)}^{\infty}[\cal{A}_{N,m}(-z)]_idz=(N-m)^{1/2}|f_{N,m,1}(x)|^{1/2}\int_{-(N-m)^{2/3}f_{N,m,1}(x)}^{\infty}[\cal{A}_{N,m}(-z)]_idz=\]\[\pi^{-1/2}(N-m)^{-1/3}|f_{N,m,1}(x)|^{-1/2}[\cal{C}_N(x)]_i+O(N^{-4/3}(1-x)^{-9/4}).\label{ignore-v3}\]
		The subtlety that arises is that direct substitution of these asymptotic into $I_{11}$, even neglecting the error terms, leads to a quantity that is not integrable. To deal with this, we will need to treat the integral of $I_{ij}$ over the regions $[1-N^{-1/3},1]$ and $[1-\delta,1-N^{-1/3}]$ separately. In the latter integral, we will be able to apply (\ref{ignore-v3}), while the integral over the prior will be negligible unless $i=j=1$, where it will surprisingly cancel the contribution from $II_{11}$ up to lower order terms.
		
		To begin, we note that by (\ref{ignore-v3}) for $i,j=1,2$ and $1-\delta\le x\le 1-N^{-2/3}$
		\[[\cal{A}_{N,n}(g_{N,n}(x))]_i\int^\infty_{-g_{N,m}(x)}[\cal{A}_{N,m}(-z)]_jdz-\pi^{-1}N^{-1/3}|f_{N,m,1}(x)|^{-1/2}[\cal{S}_N(x)]_i[\cal{C}_N(x)]_j=\]\[O(N^{-7/3}(1-x)^{-4}).\label{ignore-j1}\]
		We
		now note that
		\[N^{-7/3}\int_{N^{-1/3}}^{\infty}x^{-4}dx=O(N^{-4/3}).\]
		Employing this and (\ref{ignore-j1}) we see that
		\[\frac{1}{\pi}\int_{-\infty}^{1-N^{-1/3}} [\cal{A}_{N,n}(g_{N,n}(x))]_i(\int^\infty_{-g_{N,m}(x)}[\cal{A}_{N,m}(-z)]_jdz)\frac{h_{N,i}(x)l_{N,j}(x)}{g_{N,m}'(x)}dx=\]\[\frac{1}{\pi^2(N-m)}\int_{-\infty}^{1-N^{-1/3}} [\cal{S}_N(x)]_i[\cal{C}_N(x)]_j \frac{h_{N,i}(x)l_{N,j}(x)}{f_{N,m,1}(x)^{1/2}f_{N,m,1}'(x)}dx+O(N^{-4/3}). \label{kyle-john}\]
		Now if $(i,j)=(1,2)$ or $(i,j)=(2,1)$ then for $1-\delta\le x\le 1-N^{-2/3}$
		\[[\cal{A}_{N,n}(g_{N,n}(x))]_i\int^\infty_{-g_{N,m}(x)}[\cal{A}_{N,m}(-z)]_jdz=O(N^{-1/3}x^{-1/2}),\]
		and $O(1)$ for $1-N^{-2/3}\le x\le 1$. Combining both of these, we see that
		\[\frac{1}{\pi}\int_{1-N^{-1/3}}^{1} [\cal{A}_{N,n}(g_{N,n}(x))]_i(\int^\infty_{-g_{N,m}(x)}[\cal{A}_{N,m}(-z)]_jdz)\frac{h_{N,i}(x)l_{N,j}(x)}{g_{N,m}'(x)}dx=O(N^{-4/3}).\label{ignore-r1}\] 
		The case of $(i,j)=(2,2)$ follows more simply by just employing (\ref{ignore-v1}) to obtain (\ref{ignore-r1}). Now we will focus on the case of $(i,j)=(1,1)$. We note that by Taylor's Theorem $f_{N,n,1}(x)-f_{N,m,1}(x)=O(N^{-1})$ uniformly for $x\in (1-\delta,1+\delta)$. Moreover we note that by the above asymptotic for $\Ai'$ there is $C$ such that for $x\in (-\infty,0)$ we have that $|\Ai'(x)|\le C(1+|x|^{1/4})$. Thus by again applying Taylor's Theorem we see that for $1-N^{-1/3}\le x\le 1$, \[|[\cal{A}_{N,n}(g_{N,n}(x))]_1-[\cal{A}_{N,m}(g_{N,m}(x))]_1|=O(N^{-1/6}+(1-x)^{1/4}).\] From this, we see that there is $C$ such that
		\[|\int_{1-N^{-1/3}}^{1}([\cal{A}_{N,n}(g_{N,n}(x))]_1-[\cal{A}_{N,m}(g_{N,m}(x))]_1)(\int^\infty_{-g_{N,m}(x)}[\cal{A}_{N,m}(-z)]_1dz)\frac{h_{N,1}(x)l_{N,1}(x)}{g_{N,m}'(x)}dx|\le\]
		\[CN^{-1}\int_{1-N^{-1/3}}^1\frac{N^{-1/6}+(1-x)^{1`/4}}{(1-x)^{3/4}}dx=O(N^{-7/6}).\]
		In addition, we note that by integration by parts we may write
		\[\frac{1}{\pi}\int_{1-N^{-1/3}}^{1}[\cal{A}_{N,m}(g_{N,m}(x))]_1(\int^\infty_{-g_{N,m}(x)}[\cal{A}_{N,m}(-z)]_1dz)\frac{h_{N,1}(x)l_{N,1}(x)}{g_{N,m}'(x)}dx=I'+II',\]
		\[I'=\frac{1}{2\pi}\bigg[\frac{h_{N,1}(x)l_{N,1}(x)}{g_{N,m}'(x)g_{N,m}'(x)}\bigg(\int^\infty_{-g_{N,m}(x)}[\cal{A}_{N,m}(-z)]_1dz\bigg)^2\bigg]_{1-N^{-1/3}}^{1},\]
		\[II'=-\frac{1}{2\pi}\int_{1-N^{-1/3}}^1\frac{d}{dx}(\frac{h_{N,1}(x)l_{N,1}(x)}{g_{N,m}'(x)g_{N,m}'(x)})\bigg(\int^\infty_{-g_{N,m}(x)}[\cal{A}_{N,m}(-z)]_1dz\bigg)^2.\]
		Using (\ref{ignore-v1}) we see that $II'=O(N^{-4/3})$ and using (\ref{ignore-v2}) that
		\[I'=\frac{\phi(1)\varphi(1)}{18\pi N}[T_{1}(1)]_{11}^2+O(N^{-4/3}).\]
		We observe that this is the negation of asymptotic for $II_{11}$ obtained above. In particular, combining all of these results, we see that the left-hand side of (\ref{abe-hat}) may be rewritten as
		\[\frac{1}{\pi^2(N-m)}\sum_{i,j=1,2}\int_{-\infty}^{1-N^{-1/3}} [\cal{S}_N(x)]_i[\cal{C}_N(x)]_j \frac{h_{N,i}(x)l_{N,j}(x)}{f_{N,m,1}(x)^{1/2}f_{N,m,1}'(x)}dx+O(N^{-7/6}).\label{jack}\]
		Combining the relation $s'_{N,m}(x)=f_{N,m,1}(x)^{1/2}f_{N,m,1}'(x)$, Lemma \ref{Airy Matching}, and the bounds on $R_N,\bar{R}_N$, as in the proof of Lemma \ref{both Bessel Lemma}, we see that the integral of (\ref{jack}) may further be rewritten as
		\[\frac{1}{\pi(N-m)}\int_{-\infty}^{1-N^{-1/3}}[T_{\infty}(x)\begin{bmatrix}e^{i(N-n)s_{N,n}(x)}\\ e^{-i(N-n)s_{N,n}(x)}\end{bmatrix}]_1[T_{\infty}(x)e^{-i\frac{\pi}{2}\sigma}\begin{bmatrix}e^{i(N-m)s_{N,m}(x)}\\ e^{-i(N-m)s_{N,m}(x)}\end{bmatrix}]_1\frac{\phi(x)\varphi(x)}{s_{N,m}'(x)}dx+O(N^{-7/6})=\]
		\[\frac{-i}{\pi(N-m)}\sum_{k,l=1,2}\int_{-\infty}^{1-N^{-1/3}}[T_{\infty}(x)]_{1k}[T_{\infty}(x)]_{1l}e^{-(-1)^k i(N-n)s_{N,n}(x)-(-1)^l i(N-m)s_{N,m}(x)}(-1)^l\frac{\phi(x)\varphi(x)}{s_{N,m}'(x)}dx\label{last-airy}\]
		We will show that the terms in this equation when $k=l$ are negligible. For this, we observe that by integration by parts, we may write the $(k,l)=(1,1)$ term in (\ref{last-airy}) as the sum $I''+II''$ where
		\[I''=\bigg[\frac{-i}{\pi(N-m)}\bigg(\frac{[T_{\infty}(x)]_{11}^2\phi(x)\varphi(x)}{s'_{N,m}(x)((N-n)s_{N,n}'(x)+(N-m)s_{N,m}'(x))}\bigg)e^{i((N-n)s_{N,n}(x)+(N-m)s_{N,m}(x))}\bigg]_{-\infty}^{1-N^{-1/3}}\]
		\[II''=\frac{i}{\pi(N-m)}\int_{-\infty}^{1-N^{-1/3}}\frac{d}{dx}\bigg(\frac{[T_{\infty}(x)]_{11}^2\phi(x)\varphi(x)}{s'_{N,m}(x)((N-n)s_{N,n}'(x)+(N-m)s_{N,m}'(x))}\bigg)e^{i((N-n)s_{N,n}(x)+(N-m)s_{N,m}(x))}dx.\]
		Observing that $[T_{\infty}(x)]_{11}=O((1-x)^{-1/4})$ and that $s'(x)=O((1-x)^{1/2})$, we see that there is $C$ such that
		\[|I'|\le N^{-2}C|[x^{-3/2}]_{N^{-1/3}}^{\infty}|=CN^{-3/2},\;\;|II''|\le \frac{C}{N^2}\int_{N^{-1/3}}^{\infty}\frac{dx}{x^{5/2}}\le CN^{-3/2}.\]
		Together these show that the $(k,l)=(1,1)$ term in (\ref{last-airy}) is $O(N^{-3/2})$. Showing that the $(k,l)=(2,2)$ is $O(N^{-3/2})$ is similar. Now lastly, we will need a sharp form of Lemma \ref{sk expansion lem}. For this we note that for $|x|<1$, $s'(x)=-\sq{1-x^2}$, and so for $x\in (1-\delta,1-N^{-1/3})$ we have that $s''(x)=O(N^{1/6})$. From this, we see that for fixed $k$ and $x\in (1-\delta,1-N^{-1/3})$, we have that
		\[(N-k)s_{N,k}(x)=Ns(x)-k\arcos(x)+O(N^{-5/6}).\]
		Employing this and proceeding as in the proof of Lemma \ref{both Bessel Lemma}, we may show now that (\ref{last-airy}) coincides with 
		\[\frac{1}{2\pi N}\int_{-1}^{1-N^{-1/3}} \frac{\phi(x)\varphi(x)}{(1-x^2)}\sin((m-n)\arcos(x))dx+O(N^{-7/6})=\]\[\frac{1}{2\pi N}\int_{-1}^{1} \frac{\phi(x)\varphi(x)}{(1-x^2)}\sin((m-n)\arcos(x))dx+O(N^{-7/6}).\]
		Altogether, these statements complete the proof of (\ref{abe-hat}).
	\end{proof}
	
	Taking a partition of unity with respect to all of the above regions, we may conclude the following proposition.
	\begin{prop}
		\label{Analytic Integral Double Integral}
		Let $\phi$ and $\varphi$ be smooth functions of subexponential growth. Then we have that for any $n$ and $m$ that
		\[\frac{1}{2}\int\int  \phi(x)\varphi(y)p_{N-n}(x)p_{N-m}(y)\sign(x-y)w_N(x)^{1/2}w_N(y)^{1/2}dx=\]\[\frac{1}{2\pi N}[\int_{-1}^{1} \frac{\phi(x)\varphi(x)}{(1-x^2)}\sin((m-n)\arcos(x))dx]+\]\[\frac{1}{4N}[(-1)^{N-m}\phi(1)\varphi(-1)-(-1)^{N-n}\phi(-1)\varphi(1)]+O(N^{-7/6}).\]
	\end{prop}
	We see that as \[\int_{-1}^1 \frac{\sin(\arcos(x))}{(1-x^2)}dx=\pi,\] Proposition \ref{Analytic Integral Double Integral} shows that when $N$ is even that
	\[\int  J_N^{-1}(p_{N})(x)p_{N-1}(x)w_N(x)dx=O(N^{-7/6}).\label{ig-r2}\]
	For $n$ and $m$ arbitrary, we additionally see that
	\[\int  J_N^{-1}(p_{N-n})(x)p_{N-m}(x)w_N(x)dx=O(N^{-1}).\]
	Together these demonstrate two of the integrals required by Proposition \ref{W-Integral-Proposition}. Before proceeding, we also observe the following corollary of Lemma \ref{Bulk Single Integral of pN}, and the single integral cases of Lemmas \ref{both Bessel Lemma} and \ref{both airy lemma}.
	\begin{corr}
		\label{corr Single Integral of pN}
		Assume that $\phi$ is a smooth function on $\R$ of sub-exponential growth. Then for fixed $k$, we have that
		\[\int \phi(x)p_{N-k}(x)w_N(x)^{1/2}dx=O(N^{-1/2}).\]
	\end{corr}
	To complete the proof of Proposition \ref{W-Integral-Proposition}, we must now focus our attention on the integrals involving $\ell_i$ and $q_i$. We will begin with some single integral computations.
	\begin{lem}
		\label{single-integral-lemma-l-q}
		Assume that $\phi$ is a smooth function of subexponential growth. Then for any $1\le i\le m$, we have that
		\[\int \phi(x)\ell_i(x)w_N(x)^{1/2}dx=O(N^{-1/2}),\;\;\;\int \phi(x)q_i(x)w_N(x)^{1/2}dx=O(1).\]
	\end{lem}
	\begin{proof}
		We adopt the notation $\lam=\lam_i$ and $\alpha=\alpha_i$ as before. In view of Proposition \ref{l and q asym} and (\ref{l and q relation}) we see that
		\[\int \phi(x)\begin{bmatrix}-\ell_i(x)\\2\pi iq_i(x)\end{bmatrix}w_N(x)^{1/2}dx=\]\[c_{i}^{-\sigma}e^{-i\frac{\pi}{4}}e^{-\frac{\pi i}{4}\sigma}\sigma_1 T_{N,\lam}(\lam)^{-1}(I+O(N^{-1}))e^{-N\frac{\ell}{2}\sigma}[\int \frac{\phi(x)}{x-\lam}\begin{bmatrix}\kappa_{N}^{-1}p_{N}(x)\\-2\pi i\kappa_{N-1}p_{N-1}(x)\end{bmatrix}w_N(x)^{1/2}dx].\]
		From Remark \ref{Remark-coefficents-new} we see that the right-hand side may further be written as
		\[c_{i}^{-\sigma}e^{-i\frac{\pi}{4}}e^{-\frac{\pi i}{4}\sigma}\sigma_1T_{N,\lam}(\lam)^{-1}(I+O(N^{-1}))[ \int \frac{\phi(x)}{x-\lam}\begin{bmatrix}\pi^{1/2}p_{N}(x)\\-i\pi^{1/2}p_{N-1}(x)\end{bmatrix}w_N(x)^{1/2}dx].\]
		Thus we see that it suffices to show that
		\[\int \frac{\phi(x)}{x-\lam}\begin{bmatrix}\pi^{1/2}p_{N}(x)\\-i\pi^{1/2}p_{N-1}(x)\end{bmatrix}w_N(x)^{1/2}dx=T_{N,\lam}(\lam)\begin{bmatrix}O(1)\\0\end{bmatrix}+O(N^{-1/2}).\label{dore}\]
		Let us pick $\delta_0>0$ such that asymptotics of Proposition \ref{Asym p_N Bessel} hold. By applying Corollary \ref{corr Single Integral of pN} we see that it suffices to assume that $\phi$ is supported on $(\lam-\delta_0,\lam+\delta_0)$. With the asymptotics of Proposition \ref{Asym p_N Bessel}, and a change of coordinates by $f_{\lam}$, we may write
		\[\int \frac{\phi(x)}{x-\lam}\begin{bmatrix}\pi^{1/2}p_{N}(x)\\-i\pi^{1/2}p_{N-1}(x)\end{bmatrix}w_N(x)^{1/2}dx=\]
		\[\int \frac{G_{N,1}(x)}{x}\sign(x)\sq{\pi |Nx|}J_{\alpha+1/2}(N|x|)dx+\int \frac{G_{N,2}(x)}{x}\sq{\pi |Nx|}J_{\alpha-1/2}(N|x|)dx=I+II,\label{irrev-no}\]
		where here $G_{N,i}=(f_{\lam}^{-1})'(x)\phi(f_{\lam}^{-1}(x))[(I+R_N(f_{\lam}^{-1}(x)))T_{N,\lam}(f_{\lam}^{-1}(x))]e_i$. We observe that $G_{N,i}^{(c)}(x)=O(1)$ uniformly for all $c\in \N$. From this, we see by Lemma \ref{SingleBesselLemma} that
		\[I=2\sq{\pi}G_{N,1}(0)D(\alpha+1/2,-1)+O(N^{-1})=\sq{\pi}(f_{\lam}^{-1})'(0)\phi(\lam)[T_{N,\lam}(\lam)]e_1D(\alpha+1/2,-1/2)+O(N^{-1}).\]
		The case of $II$ is similar except that the first-order contribution of the integrals on $(0,\infty)$ and $(-\infty,0)$ cancel so that $II=O(N^{-1})$. Together these computations show that
		\[\int \frac{\phi(x)}{x-\lam}\begin{bmatrix}\pi^{1/2}p_{N}(x)\\-i\pi^{1/2}p_{N-1}(x)\end{bmatrix}w_N(x)^{1/2}dx=T_{N,\lam}(\lam)\begin{bmatrix}\phi(\lam)\sq{\pi}2D(\alpha+1/2,-1/2)\\0 \end{bmatrix}+O(N^{-1}),\]
		which is more than sufficient
		to verify (\ref{dore}).
	\end{proof}
	
	We recall from Remark \ref{l and q remark} that for fixed $1\le i\le m$ there are constants $C_{1,N}, C_{2,N}$ with $C_{k,N}=O(1)$ and such that  
	\[\ell_i(x)=C_{1,N}\frac{p_{N}(x)}{x-\lam_i}+C_{2,N}\frac{p_{N-1}(x)}{x-\lam_i}, \label{first-order-form}\]
	and similarly for $q_i$. Let us take $\delta$ as above. We fix for, each $1\le l\le m$, a choice of smooth function $\phi_l$ supported on $(\lam_l-\delta/2,\lam_l+\delta/2)$, such that $\phi_l(x)=1$ for $x\in (\lam_l-\delta/4,\lam_l+\delta/4)$. In addition, we choose $\varphi_l$ supported on $(\lam_l-\delta,\lam_l+\delta)$, such that $\varphi_{l}(x)=1$ for $x\in (\lam_l-\delta/2,\lam_l+\delta/2)$.
	
	With these selected, we may write
	\[\int J_N^{-1}\ell_i(x) p_{N-n}(x)w_N(x)dx=\int (J_N^{-1}\phi_i\ell_i)(x) p_{N-n}(x)w_N(x)dx+\int (J_N^{-1}(1-\phi_i)\ell_i)(x) p_{N-n}(x)w_N(x)dx.\]
	Employing (\ref{first-order-form}), we may write the second integral as
	\[\sum_{k=0,1}\frac{C_{k+1,N}}{2}\int\int p_{N-n}(x)p_{N-k}(y)\frac{(1-\phi_i(y))}{(y-\lam_i)}w_N(x)w_N(y)dxdy.\label{l-dem}\]
	Observing that the function $(1-\phi_i(x))/(x-\lam_i)$ is smooth and bounded, we may apply Proposition \ref{Analytic Integral Double Integral} to see that the terms in (\ref{l-dem}) are of order $O(N^{-1})$. From this, we see that
	\[\int J_N^{-1}\ell_i(x) p_{N-n}(x)w_N(x)dx=\int (J_N^{-1}\phi_i\ell_i)(x) p_{N-n}(x)w_N(x)dx+O(N^{-1}).\]
	We now also write
	\[\int (J_N^{-1}\phi_i\ell_i)(x) p_{N-n}(x)w_N(x)dx=\]\[\int (J_N^{-1}\phi_i\ell_i)(x) p_{N-n}(x)\varphi_i(x)w_N(x)dx+\int (J_N^{-1}(\phi_i\ell_i))(x)(1-\varphi_i(x)) p_{N-n}(x)w_N(x)dx.\]
	We note that as $\phi_i\ell_i$ and $(1-\varphi)$ have disjoint support, we may expand $J_{N}^{-1}$ and employ (\ref{disjoint-support}) to rewrite the second integral on the right-hand side as
	\[\frac{1}{2}\int \phi_i(y)\ell_i(y)w_N(y)^{1/2}dy\int (1-\varphi_i(x))\sign(x-\lam_i) p_{N-n}(x)w_N(x)^{1/2}dx\]
	which is $O(N^{-1})$ by employing Lemma \ref{single-integral-lemma-l-q} and Corollary \ref{corr Single Integral of pN}. Altogether we see that
	\[\int J_N^{-1}\ell_i(x) p_{N-n}(x)w_N(x)dx=\int J_N^{-1}(\phi_i\ell_i)(x) \varphi_i(x)p_{N-n}(x)w_N(x)dx+O(N^{-1}).\]
	Similarly, we obtain that
	\[\int J_N^{-1}p_{N-n}(x)q_i(x)w_N(x)dx=\int J_N^{-1}(\phi_ip_{N-n})(x)\varphi_i(x)q_i(x) w_N(x)dx+O(N^{-1/2}),\]\[
	\int J_N^{-1}\ell_{i}(x)q_j(x)w_N(x)dx=\delta_{ij}\int J_N^{-1}(\phi_i\ell_{i})(x)\varphi_i(x)q_i(x)w_N(x)dx+O(N^{-1/2}).\]
	From these computations and (\ref{first-order-form}), we see that to complete the proof of Proposition \ref{W-Integral-Proposition}, it suffices to show the following.
	\begin{lem}
		\label{Over Integral Double Integral}
		There is $\delta>0$ such that if $\phi$ and $\varphi$ are smooth functions supported on $(\lam_j-\delta,\lam_j+\delta)$ for some $1\le j\le m$ then for any choice of fixed $n$ and $m$ we have that
		\[\frac{1}{2}\int \int \frac{\phi(x)p_{N-n}(x)}{x-\lam_j}\varphi(y) p_{N-m}(y)w_N(x)^{1/2}w_N(y)^{1/2}dxdy=O(N^{-1}\log(N)).\label{Integral-1-Over}\]
		Moreover, if $|n-m|\le 1$, then we additionally have that
		\[\frac{1}{2}\int\int \frac{\phi(x)p_{N-n}(x)}{x-\lam_j} \frac{\varphi(y)p_{N-m}(y)}{y-\lam_j} w_N(x)w_N(y)dxdy=O(N^{-1}\log(N)).\label{Integral-1-Over-2}\]
	\end{lem}
	\begin{proof}
		Again, for convenience, we write $\lam=\lam_j$ and $\alpha=\alpha_j$, and choose $\delta$ as before. We begin by proving the first statement. With notation as in Lemma \ref{both Bessel Lemma}, we may rewrite the left hand side of (\ref{Integral-1-Over}) as in (\ref{Oct1}) as
		\[\frac{1}{2\pi}\sum_{k,l=1,2}\int\int \frac{1}{x-\lam}h_{N,k}(x)[\cal{J}_{\alpha}(g_{N,n}(x))]_kl_{N,l}(y)[\cal{J}_{\alpha}(g_{N,m}(y))]_l\sign(x-y)dxdy.\label{Oct1-2}\]
		We claim that the summand for each choice of $(k,l)$ is $O(N^{-1}\log(N))$. To begin, we define, for $i=1,2$, the functions $I_i(x)=\frac{1}{2}\int \sign(x-y)[\cal{J}_{\alpha}(y)]_iy^{-1}dy$. We note that by performing integration by parts in $x$ (see the proof of Lemma \ref{integration-by-parts-lemma}) we may write (\ref{Oct1-2}) as the sum $I+II$ where
		\[I=\frac{1}{\pi}\sum_{k,l=1,2}\int [\cal{J}_{\alpha}(g_{N,m}(x))]_lI_k(g_{N,n}(x))\frac{h_{N,k}(x)l_{N,l}(x)g_{N,n}(x)}{g_{N,n}'(x)(x-\lam)}dx,\]
		\[II=-\frac{1}{2\pi}\sum_{k,l=1,2}\int \int [\cal{J}_{\alpha}(g_{N,m}(x))]_lI_k(g_{N,n}(y))l_{N,l}(y)\frac{d}{dx}(\frac{h_{N,k}(x)g_{N,n}(x)}{g_{N,n}'(x)(x-\lam)})\sign(x-y)dxdy.\]
		We observe that as $g_{N,n}(\lam)=0$ we have that $g_{N,n}(x)/((x-\lam)g_{N,n}'(x))$ as well as its derivatives are $O(1)$ for $x\in (\lam-\delta,\lam+\delta)$. We note that $I_1(x)$  ($I_2(x)$) is anti-symmetric (symmetric), respectively. Moreover, for $x>0$, we have the relations
		\[I_1(x)=\int_0^x[\cal{J}_{\alpha}(y)]_1y^{-1}dy=\sq{\pi}J_{\mu+1,-1/2}(x),\]
		\[I_2(x)=\int_x^{\infty}[\cal{J}_{\alpha}(y)]_2y^{-1}dy=\sq{\pi}\int_x^{\infty}y^{-1/2}J_{\alpha+1/2}(y)dy.\]
		Applying the asymptotics (\ref{IntegralAsymptotic}) and (\ref{moss-tow}) above, it follows that for $i=1,2$ we have that $I_i(x)=O(|x|^{-1})$ for $|x|\ge 1$, $I_i(x)=O(1)$ for $|x|\le 1$, and $[\cal{J}_{\alpha}(x)]_l=O(1)$ for all $x$. Applying these bounds pointwise, and noting that $h_{N,k}(x),l_{N,l}(x)=O(1)$, we see that the integrand of $I$ is uniformly $O(N^{-1}|x|^{-1})$ for $|x|\ge 1/N$, and $O(1)$ for $|x|\le 1/N$. Thus contribution of the integral over $\R\setminus(\lam-N^{-1},\lam+N^{-1})$ is of order $O(N^{-1}\log(N))$ and the contribution over $(\lam-N^{-1},\lam+N^{-1})$ is of order $O(N^{-1})$, so that $I=O(N^{-1}\log(N))$. Similarly we see that $II=O(N^{-1}\log(N))$. Together these establish (\ref{Integral-1-Over}).
		
		We now establish (\ref{Integral-1-Over-2}). The key difference, in this case, is that we may employ Proposition \ref{Asym p_N Bessel} only once to obtain the asymptotics of both $p_{N-n}$ and $p_{N-m}$. For notational ease, we will only prove the case when $n=0$ and $m=0,1$. We write $h_{N,k}(x)=\phi(x)[(I+R_N(x))T_{N,\lam}(x)]_{1k}$ and  $l_{N,k}(x)=i^{I(m=1)}\varphi(x)[(I+R_N(x))T_{N,\lam}(x)]_{(1+m)k}$, where $R_N$ is as in Proposition \ref{Asym p_N Bessel}. This is where we use the assumption that $|n-m|\le 1$ so that we may write both terms with the same asymptotic. Using this, we may rewrite the left-hand side of (\ref{Integral-1-Over-2}) as
		\[\frac{1}{2\pi}\sum_{k,l=1,2}\int\int \frac{h_{N,k}(x)l_{N,l}(y)}{(x-\lam)(y-\lam)}[\cal{J}_{\alpha}(Nf_\lam(x))]_k[\cal{J}_{\alpha}(Nf_{\lam}(y))]_l\sign(x-y)dxdy.\]
		If we define \[\bar{h}_{N,i}(x)=h_{N,i}(f_{\lam}^{-1}(x))\frac{(f_{\lam}^{-1})'(x)x}{(f_{\lam}^{-1}(x)-\lam)},\] and similarly for $\bar{l}_{N,i}$, we may further rewrite (\ref{Integral-1-Over-2}) as
		\[\frac{1}{2\pi}\sum_{k,l=1,2}\int\int (\frac{\bar{h}_{N,k}(x)\bar{l}_{N,l}(y)}{xy})[\cal{J}_{\alpha}(Nx)]_k[\cal{J}_{\alpha}(Ny)]_{l}\sign(x-y)dxdy.\label{Integral-2-2-2}\]
		Let us fix $\chi(x)$, a smooth, even, function of compact support, with $\chi(x)=1$ for $x\in(-\delta,\delta)$. Applying Taylor's theorem to $\bar{h}_{N,l}$ and $\bar{l}_{N,l}$ and proceeding as in the proof of (\ref{Integral-1-Over}), we see that that it suffices to show that for $i=1,2$ and $j=1,2$, that
		\[\frac{1}{2}\int\int \frac{\chi(x)\chi(y)}{xy}[\cal{J}_{\alpha}(Nx)]_i[\cal{J}_{\alpha}(Ny)]_j\sign(x-y)dxdy=O(N^{-1}\log(N)).\label{Weird-Integral-Over-1}\]
		This integral is antisymmetric under the interchange of $i$ and $j$. Thus it suffices to prove the case of $(i,j)=(1,2)$. By employing Lemma \ref{integration-by-parts-lemma} we may rewrite (\ref{Weird-Integral-Over-1}) in the case $(i,j)=(1,2)$ as the sum of three integrals $I+II+III$
		\[I=\int x^{-1}[\cal{J}_{\alpha}(Nx)]_1I_2(Nx)\chi(x)\chi(x)dx,\]
		\[II=\int I_1(Nx)I_2(Nx)\chi(x)\chi'(x)dx,\]
		\[III=\frac{1}{2}\int\int I_1(Nx)I_2(Ny)\chi'(x)\chi'(y)\sign(x-y)dxdy.\]
		As $\chi'$ vanishes in a neighborhood of $0$, and $I_i(x)=O(|x|^{-1})$, we see that $II,III=O(N^{-2})$. To deal with $I$, we first rewrite it as
		\[I=\int x^{-1}[\cal{J}_{\alpha}(x)]_1I_2(x)\chi(x/N)^2dx.\]
		As $[\cal{J}_{\alpha}(x)]_1=O(1)$, we note that $x^{-1}[\cal{J}_{\alpha}(x)]_1I_2(x)=O(|x|^{-2})$ for $x\in \R\setminus(-1,1)$, and $O(1)$ for $x\in [-1,1]$ so that $x^{-1}[\cal{J}_{\alpha}(x)]_1I_2(x)$ is integrable. In addition, we see that there is $C$ such that
		\[|\int x^{-1}[\cal{J}_{\alpha}(x)]_1I_2(x)(1-\chi(x/N)^2)dx|\le C\int_{|x|\ge \delta N}x^{-2}dx=2C\delta^{-1}N^{-1}.\]
		Combining these, we see that
		\[I=\int x^{-1}[\cal{J}_\alpha(x)]_1I_2(x)dx+O(N^{-1})=2\sq{\pi}\int_0^{\infty} x^{-1/2}J_{\alpha+1/2}(x)I_2(x)dx+O(N^{-1}).\]
		We write
		\[I'=\frac{1}{\sq{\pi}}\int_0^{\infty} x^{-1/2}J_{\alpha+1/2}(x)I_2(x)dx=\int_0^{\infty}\int_x^{\infty}x^{-1/2}J_{\alpha+1/2}(x)y^{-1/2}J_{\alpha-1/2}(y)dydx,\]
		we now need only show that $I'=0$. Making the change of variables $y=ax$, this integral becomes
		\[\int_0^{\infty}\int_1^{\infty}a^{-1/2}J_{\alpha+1/2}(x)J_{\alpha-1/2}(ax)dadx\label{Bessel-a-int}.\]
		We note that by equation 11.4.41 of \cite{AS}, we have that
		\[\int_0^{\infty}a^{-1/2}J_{\alpha+1/2}(x)J_{\alpha-1/2}(ax)dx=0;\;\;\;a>1\label{Abram}.\]
		This shows that (\ref{Bessel-a-int}) vanishes when the integrals are taken in the opposite order. As the integrals are not absolutely convergent, though, we must proceed with care. Applying the above bounds on $I_2$ and $J_{\alpha+1/2}$, we have that
		\[\int_0^{L}\int_{L}^{\infty}a^{-1/2}J_{\alpha+1/2}(x)J_{\alpha-1/2}(ax)dadx=\frac{1}{\sq{\pi}}\int_0^L x^{-1/2}J_{\alpha+1/2}(x)I_2(Lx)dx=O(L^{-1}\log(L)).\]
		Thus we have that
		\[I'=\lim_{L\to \infty}\int_0^{L}\int_{1}^{\infty}a^{-1/2}J_{\alpha+1/2}(x)J_{\alpha-1/2}(ax)dadx=\lim_{L\to \infty}\int_0^{L}\int_{1}^{L}a^{-1/2}J_{\alpha+1/2}(x)J_{\alpha-1/2}(ax)dadx.\]
		By employing (\ref{BesselAsymptotic}) and the double-angle formula, we may compute that
		\[\int_1^{L}\int_L^{\infty}a^{-1/2}J_{\alpha+1/2}(x)J_{\alpha-1/2}(ax)dxda=\]
		\[-\frac{2}{\pi}\int_1^{L}\int_L^{\infty}\bigg(\frac{\sin(x+\alpha\frac{\pi}{2})\cos(ax+\alpha\frac{\pi}{2})}{xa}+O(x^{-2}a^{-1})\bigg)dxda=\]
		\[-\frac{2}{\pi}\int_1^{L}\frac{1}{2a}\bigg[\int_{L}^{\infty}\
		sin((1+a)x+\alpha\pi)x^{-1}dx-\int_{L}^{\infty}\sin((a-1)x)x^{-1}dx\bigg]da+O(\log(L)L^{-1}).\]
		We note that there is $C$ such that for $y>0$, we have that $\int_y^{\infty}\sin(x)x^{-1}dx\le Cy^{-1}$. Employing this, we see that 
		\[|\int_1^{L}\frac{1}{a}\int_{L}^{\infty}\sin((1+a)x+\alpha\pi)x^{-1}dxda|\le C^2\int_1^{L}\frac{1}{aL}da=O(L^{-1}\log(L)).\]
		We also note that there is $C$ such that for $y>0$
		\[|\int_{y}^{\infty}\sin(x)x^{-1}dx|\le C.\]
		From this, we see that
		\[|\int_1^{L}\frac{1}{a}\int_{L}^{\infty}\sin((a-1)x)x^{-1}dxda|\le |\int_{1}^{1+L^{-1}}\frac{1}{a}\int_{L(a-1)}^{\infty}\sin(x)x^{-1}dxda|+\]
		\[|\int_{1+L^{-1}}^{L}\frac{1}{a}\int_{L(a-1)}^{\infty}\sin(x)x^{-1}dxda|\le C|\int_{1+L^{-1}}^{L}\frac{1}{La(a-1)}dxda|+C|\int_{1+L^{-1}}^{L}\frac{1}{a}da|.\label{tado}\]
		The first integral on the right-most side of (\ref{tado}) is  of order $O(L^{-1})$ and the second is of order $O(L^{-1}\log(L))$. Altogether we see that
		\[\lim_{L\to \infty}\int_1^{L}a^{-1/2}\int_L^{\infty}J_{\alpha+1/2}(x)J_{\alpha-1/2}(ax)dxda=0.\]
		Thus finally
		\[\lim_{L\to \infty}\int_0^{L}\int_{1}^{L}a^{-1/2}J_{\alpha+1/2}(x)J_{\alpha-1/2}(ax)dadx=\int_0^{\infty}\int_1^{\infty}a^{-1/2}J_{\alpha+1/2}(x)J_{\alpha-1/2}(ax)dadx=0.\]
		This shows that $I'=0$ which completes the proof of (\ref{Integral-1-Over}).
	\end{proof}

	\section{Proof of Proposition \ref{merging-W-Integral-Proposition}\label{merging section}}
	
	In this section, we will provide a proof of Proposition \ref{merging-W-Integral-Proposition}. The method of proof will be almost identical in structure to the proof of Proposition \ref{W-Integral-Proposition} given in the last section. The main differences that occur are that we must now choose $N$-dependent partitions of unity to isolate the regions of each specified asymptotics and that now the error asymptotics of $R_N$ are of order $O(N^{\gamma-1})$ in the regions around $\{\lam_1,\lam_2\}$. We will now, for the remainder of this section, assume that we are in the case of Proposition \ref{merging-W-Integral-Proposition}. That is, we will fix $1>\gamma>0$ and $\epsilon>0$, and only consider choices of (possibly $N$-dependant) $\lam_1,\lam_2\in (-1+\epsilon,1-\epsilon)$ such that $(\lam_1-\lam_2)>N^{-\gamma}$.
	
	We will first focus on integral results in the region $(\lam_2-\delta,\lam_1+\delta)$. Let us fix $\epsilon,\delta_0>0$ as in Proposition \ref{Merging Error Bessel}, and $\delta<\min(\epsilon/4, 1/4,\delta_0/2)$. Let us denote $\delta_N=\delta(\lam_1-\lam_2)$. We will fix a symmetric smooth function $\chi$, such that $\chi(x)=1$ for $|x|<1/8$ and $\chi(x)=0$ for $|x|>1/4$. We define $\chi_{N,i}(x)=\chi((x-\lam_i)\delta_N^{-1})$.
	We additionally choose another smooth compactly-supported function $\bar{\chi}$, supported on $(\lam_2-\delta/2,\lam_1+\delta/2)$, such that $\bar{\chi}(x)=1$ for $x\in (\lam_2-\delta/4,\lam_1+\delta/4)$. We define $\chi_{N,0}(x)=\bar{\chi}(x)-\chi_{N,1}(x)-\chi_{N,2}(x)$. A key property of these functions is the following: for $k\in \N$ and $i\in \{0,1,2\}$, there is $C$ such that
	\[\int |\chi_{N,i}^{(k)}(x)|dx\le C (\delta_N^{-k+1}+\delta_{ik}\delta_{k0})\le C^2(N^{\gamma(k-1)}+\delta_{ik}\delta_{k0}),\;\;\sup_{x}|\chi^{(k)}_{N,i}(x)|\le C\delta_{N}^{-k}\le C^2N^{\gamma k}.\label{qi-bounds}\]
	
	In addition, we will define $\chi_{N,i,j}(x)=\chi_{N,i}(x)/(x-\lam_j)$. An important subtlety that arises in the current context is that for $i\neq   j$, the function $\chi_{N,i,j}(x)$ is no longer $O(1)$. On the other hand, we instead observe that for $i\neq j$, and $k\in \N$, there is $C>0$, such that
	\[\int |\chi_{N,i,j}^{(k)}(x)|dx\le C (\delta_N^{-k}+\log(\delta_N)\delta_{ik}\delta_{k0})\le C^2(N^{\gamma k}+\gamma\log(N)\delta_{ik}\delta_{k0}),\]\[\sup_{x}|\chi_{N,i,j}^{(k)}(x)|\le C\delta_{N}^{-k-1}\le C^2N^{\gamma (k+1)}.\label{chi-bounds}\]
	
	The key property we will use for $\chi_{N,i}$ and $\chi_{N,i,j}$ will be (\ref{qi-bounds}) and (\ref{chi-bounds}), respectively, the majority of proofs of the results below are identical for both functions, with slightly worse bounds for the latter. We begin our analysis with the following modification of Lemma \ref{Bulk Single Integral of pN}.
	\begin{lem}
		\label{mer-Bulk Single Integral of pN}
		Let $\phi$ be a smooth function. Then for each fixed $k\in \Z$, $i=1,2$, and choice of $c\in \N$, we have that
		\[\int \phi(x)\chi_{N,0}(x)p_{N-k}(x)w_N(x)^{1/2}dx=O(N^{-c}),\label{sta-1}\]
		\[\int \phi(x)\chi_{N,0,i}(x)p_{N-k}(x)w_N(x)^{1/2}dx=O(N^{-c}).\label{sta-2}\]
	\end{lem}
	\begin{proof}
		We may assume that $c>1$. We first establish (\ref{sta-1}). For notational clarity, we assume that $k=0$. We define $f_{N,l}$ as in Lemma \ref{Bulk Single Integral of pN}. By (\ref{integration-by-parts}) we see that for $c\in \N$,
		\[|\int \phi(x)\chi_{N,0}(x)p_{N-k}(x)w_N(x)^{1/2}dx|\le \frac{1}{\pi^{1/2}}\sum_{l=1,2} \frac{1}{N^c}\int| \frac{d^c}{dy^c}((s^{-1})'(y)\chi_{N,0}(s^{-1}(y))f_{N,l}(s^{-1}(y)))|dy.\]
		We observe that $f_{N,l}^{(c)}(x)=O(1)$ for $c\in \N$. By the product formula, we may inductively write
		\[ \frac{d^c}{dy^c}((s^{-1})'(y)\chi_{N,0}(s^{-1}(y))f_{N,l}(s^{-1}(y)))=\sum_{k=0}^{c}\chi_{N,0}^{(k)}(s^{-1}(y))a_{k,l,N}(y),\]
		where $a_{k,l,N}(y)$ are smooth functions with $a_{k,l,N}(y)=O(1)$ uniformly. Thus we have that there is $C$ such that
		\[\int| \frac{d^c}{dy^c}((s^{-1})'(y)\chi_{N,0}(s^{-1}(y))f_{N,l}(s^{-1}(y)))|dy\le \sum_{k=0}^{c}\int |\chi_{N,0}^{(k)}(y)a_{k,l,N}(s(y))s'(y)|dy\le \]
		\[C\sum_{k=0}^{c}\int |\chi_{N,0}^{(k)}(y)|dy\le C^2 N^{c\gamma},\]
		where in the second inequality, we have used (\ref{qi-bounds}). In particular, we see that 
		\[\int \phi(x)\chi_{N,0}(x)p_{N-k}(x)w_N(x)^{1/2}dx=O(N^{c(\gamma-1)}).\]
		As $c\in \N$ was arbitrary, we see that taking $c\mapsto \ceil{c/(1-\gamma)}$, we obtain (\ref{sta-1}).
		
		To show (\ref{sta-2}), the above proof works, replacing (\ref{qi-bounds}) for (\ref{chi-bounds}) to show that the left-hand side of (\ref{sta-2}) is $O(N^{(c+1)\gamma-c})$. Adjusting $c$ again gives (\ref{sta-2}).
	\end{proof}
	
	We now state the modification of Lemma \ref{Bulk double Integral of pN} that we shall use.
	\begin{lem}
		\label{mer-Bulk double Integral of pN}
		Let $\phi$ and $\varphi$ be smooth functions. Then for fixed choice of $n$ and $m$ and $i,j\in \{1,2\}$, we have that
		\[\frac{1}{2}\int \int \phi(x)\varphi(y)\chi_{N,0}(x)\chi_{N,0}(y)p_{N-n}(x)p_{N-m}(y)\Sign(x-y)w_N(x)^{1/2}w_N(y)^{1/2}dxdy=\]\[\frac{1}{2\pi N}\int_{-1}^{1} \frac{\phi(x)\varphi(x)\chi_{N,0}(x)^2}{(1-x^2)}\sin((m-n)\arcos(x))dx+O(N^{\gamma-2}),\label{stan-1}\]
		\[\int \int \phi(x)\varphi(y)\chi_{N,0,i}(x)\chi_{N,0}(y)p_{N-n}(x)p_{N-m}(y)\Sign(x-y)w_N(x)^{1/2}w_N(y)^{1/2}dxdy=O(N^{-1}),\label{stan-2}\]
		\[\int \int \phi(x)\varphi(y)\chi_{N,0,i}(x)\chi_{N,0,j}(y)p_{N-n}(x)p_{N-m}(y)\Sign(x-y)w_N(x)^{1/2}w_N(y)^{1/2}dxdy=\]\[O(N^{-\min(1,2(1-\gamma))}).\label{stan-3}\]
	\end{lem}
	\begin{proof}
		We begin with the proof of (\ref{stan-1}). As in the proof of Lemma \ref{Bulk double Integral of pN}, we write
		\[\frac{1}{2}\int \int \phi(x)\varphi(y)\chi_{N,0}(x)\chi_{N,0}(y)p_{N-n}(x)p_{N-m}(y)\Sign(x-y)w_N(x)^{1/2}w_N(y)^{1/2}dxdy=\]
		\[\frac{1}{2\pi}\sum_{j,k=1,2}\int \int\chi_{N,0}(x)\chi_{N,0}(y) f_{N,j,k}(x,y)e^{-(-1)^jNis(x)-(-1)^kNis(y)}\sign(x-y)dxdy,\]
		where here $f_{N,j,k}$ is defined as in the proof of Lemma \ref{Bulk double Integral of pN}. We denote \[\bar{f}_{N,j,k}(x,y)=(s^{-1})'(x)(s^{-1})'(y)f_{N,j,k}(s^{-1}(x),s^{-1}(y)).\] We have that for $c\in \N$ that $\bar{f}_{N,j,k}^{(c)}(x)=O(1)$ uniformly in $x$. One may apply Lemma \ref{DoubleWatsonsLemma} as in the proof of Lemma \ref{Bulk double Integral of pN} to write
		\[\frac{1}{2}\sum_{j,k=1,2}\int \int\chi_{N,0}(x)\chi_{N,0}(y) f_{N,j,k}(x,y)e^{-(-1)^jNis(x)-(-1)^kNis(y)}\sign(x-y)dxdy=\]\[-\frac{i}{N}\int \frac{f_{N,1,2}(x,x)\chi_{N,0}(x)^2}{|s'(x)|}dx+\frac{i}{N}\int \frac{f_{N,2,1}(x,x)\chi_{N,0}(x)^2}{|s'(x)|}dx+E_N,\label{mer-bulk-error-ignore}\]
		where $E_N$ is bounded by
		\begin{align}
			\sum_{j,k=1,2}\frac{4}{N^{2}}\bigg(\int\int \|\nabla(\chi_{N,0}(s^{-1}(x))\chi_{N,0}(s^{-1}(y))\bar{f}_{N,j,k}(x,y))\|dxdy+\\
			4\int |\frac{d}{dx}(\chi_{N,0}(s^{-1}(x))\chi_{N,0}(s^{-1}(x))\bar{f}_{N,j,k}(x,x))|dx\bigg).\label{ignore-merging}
		\end{align}
		As in the proof of Lemma \ref{mer-Bulk Single Integral of pN}, and recalling that $\|(x,y)\|\le |x|+|y|$, we see that there is $C$ such that (\ref{ignore-merging}) is bounded by
		\[CN^{-2}(\int|\chi_{N,0}'(x)|dx\int|\chi_{N,0}(y)|dy+\int|\chi_{N,0}(x)|dx\int|\chi_{N,0}(y)|dy+\]
		\[\int |\chi_{N,0}'(x)||\chi_{N,0}(x)|dx+\int |\chi_{N,0}(x)|^2dx).\]
		We further note that
		\[\int |\chi_{N,0}'(x)||\chi_{N,0}(x)|dx+\int |\chi_{N,0}(x)|^2dx\le (\sup_{x}|\chi_{N,0}(x)|)(\int |\chi_{N,0}'(x)|dx+\int |\chi_{N,0}(x)|dx).\]
		All of these terms are $O(1)$ by (\ref{qi-bounds}). Thus we see that we have that $E_N=O(N^{-2})$. As we have that $R_N(x),\bar{R}_N(x)=$ $O(N^{\gamma-1})$ we have that
		\[f_{N,j,k}(x,y)=\phi(x)\varphi(y)[T_{\infty}(x)]_{1j}[T_{\infty}(y)]_{1k}e^{(-1)^ji(n\arcos(x)+\pi (I(x<\lam_1)+I(x<\lam_2))\alpha)}\times\]
		\[e^{(-1)^ki(m\arcos(y)+\pi (I(y<\lam_1)+I(y<\lam_2)))\alpha}+O(N^{\gamma-1}).\label{hiso}\]
		The remainder of the proof of (\ref{stan-1}) consists of inserting (\ref{hiso}) into (\ref{mer-bulk-error-ignore}) and simplifying the resulting expression identically to as in the proof of Lemma \ref{Bulk double Integral of pN}.
		
		To establish (\ref{stan-2}), the entire proof works identically, except that $E_N$ is now bounded by
		\[CN^{-2}[\int|\chi_{N,0,i}'(x)|dx\int|\chi_{N,0}(x)|dx+\int|\chi_{N,0}'(x)|dx\int|\chi_{N,0,i}(x)|dx+\]\[\int|\chi_{N,0}(x)|dx\int|\chi_{N,0,i}(x)|dx+
		(\sup_{x}|\chi_{N,0}(x)|)(\int |\chi_{N,0,i}'(x)|dx+
		\int|\chi_{N,0,i}(x)|dx)+\]\[
		(\sup_{x}|\chi_{N,0}(x)|)\int |\chi_{N,0,i}'(x)|dx],\]
		for some $C>0$. Employing (\ref{qi-bounds}) and (\ref{chi-bounds}) we see that $E_N=O(N^{\gamma-2})$, which establishes (\ref{stan-2}). 
		
		In the case of (\ref{stan-3}),
		we have that $E_N$ is now bounded by
		\[CN^{-2}[\int|\chi_{N,0,i}'(x)|dx\int|\chi_{N,0,j}(x)|dx+\int|\chi_{N,0,j}'(x)|dx\int|\chi_{N,0,i}(x)|dx+\]\[\int|\chi_{N,0,j}(x)|dx\int|\chi_{N,0,i}(x)|dx+
		(\sup_{x}|\chi_{N,0,j}(x)|)(\int |\chi_{N,0,i}'(x)|dx+
		\int|\chi_{N,0,i}(x)|dx)+\]\[
		(\sup_{x}|\chi_{N,0,j}(x)|)\int |\chi_{N,0,i}'(x)|dx],\]
		for some $C$. Again employing (\ref{chi-bounds}) we now see that $E_N=O(N^{2\gamma-2})$, which establishes (\ref{stan-3}).
	\end{proof}
	
	We now will focus on the modifications that need to be made in the region around $\lam_i$. We note that the preparatory Lemma \ref{Bessel Matching Lemma} still holds as stated. 
	
	\begin{lem}
		\label{mer-both Bessel Lemma}
		Let $\phi$ and $\varphi$ be smooth functions. Then for any $i,j\in\{0,1,2\}$, and fixed $n$ and $m$ we have that
		\[\int \phi(x)\chi_{N,i}(x)p_{N-n}(x)w_N(x)^{1/2}dx=O(N^{-1}),\label{mer-bes-1}\]
		\[\frac{1}{2}\int\int  \phi(x)\varphi(y)\chi_{N,i}(x)\chi_{N,j}(y)p_{N-n}(x)p_{N-m}(y)\sign(x-y)w_N(x)^{1/2}w_N(y)^{1/2}dxdy=\]\[\frac{1}{2\pi N}\int_{-1}^{1} \frac{\phi(x)\varphi(x)\chi_{N,i}(x)\chi_{N,j}(x)}{(1-x^2)}\sin((m-n)\arcos(x))dx+O(N^{\gamma-2}).\label{mer-bes-2}\]
		In addition, with $r,t=1,2$, such that $i\neq r$ and $j\neq t$, we have that
		\[\int \phi(x)\chi_{N,i,r}(x)p_{N-n}(x)w_N(x)^{1/2}dx=O(N^{-\min(1,2(1-\gamma))}),\label{mer-bes-1-1}\]
		\[\int\int  \phi(x)\varphi(y)\chi_{N,i,r}(x)\chi_{N,j}(y)p_{N-n}(x)p_{N-m}(y)\sign(x-y)w_N(x)^{1/2}w_N(y)^{1/2}dxdy=O(N^{-1}),\label{mer-bes-2-1}\]
		\[\int\int  \phi(x)\varphi(y)\chi_{N,i,r}(x)\chi_{N,j,t}(y)p_{N-n}(x)p_{N-m}(y)\sign(x-y)w_N(x)^{1/2}w_N(y)^{1/2}dxdy=O(N^{-\min(1,2(1-\gamma))}).\label{mer-bes-2-2}\]
	\end{lem}
	\begin{proof}
		We begin with the proof of (\ref{mer-bes-1}). We have already established the case of $i=0$, so we assume that $i>0$, and denote $\lam_i=\lam$. We observe that proceeding identically to Lemma \ref{both Bessel Lemma} we have that there is $C$ such that
		\[|\int_{\lam}^{\lam\pm \delta}\phi(x)p_{N-n}(x)\chi_{N,i}(x)w_N(x)^{1/2}dx|\le N^{-1}\pi^{1/2}(D(\alpha+1/2,1/2)|g_{N,1}(\lam)|+D(\alpha-1/2,1/2)|g_{N,2}(\lam)|)\times\]\[(f_{N,n,\lam}^{-1})'(0)+\frac{C}{N^2}\sum_{k=1,2}\sum_{l=0}^{2}|\int_{0}^{\pm \infty}\frac{d^l}{dx^{l}}((f_{N,n,\lam}^{-1})'(x)g_{N,k}(f_{N,n,\lam}^{-1}(x))\chi_{N,i}(f_{N,n,\lam}^{-1}(x)))|dx,\label{ignor-ja}\]
		where $g_{N,k}$ is defined as in the proof of Lemma \ref{both Bessel Lemma}. The first term on the right-hand side is clearly $O(N^{-1})$. Additionally, we may show, as before, that there is $C$, such that the second term of the right-hand side of (\ref{ignor-ja}) may be bounded by
		\[\frac{C^2}{N^2}\sum_{l=0}^{2}\int|\chi_{N,i}^{(l)}(x)|dx=O(N^{\gamma-2}).\]
		Together these bounds establish (\ref{mer-bes-1}). The same argument may be used to establish (\ref{mer-bes-1-1}).
		
		To show (\ref{mer-bes-2}), we first observe that the case of $i=j=0$ was established in Lemma \ref{mer-Bulk double Integral of pN}. In the case that $i>0$ and $j=0$, let us write $\hat{\chi}(x)=\chi(4x)$ and $\hat{\chi}_{N,i}(x)=\hat{\chi}((x-\lam_i)\delta_N^{-1})$, so that $\supp(\hat{\chi}_{N,i})\cap \supp(\chi_{N,0})=\varnothing$. Thus we see that
		\[\int\int  \phi(x)\varphi(y)\hat{\chi}_{N,i}(x)\chi_{N,0}(y)p_{N-n}(x)p_{N-m}(y)\sign(x-y)w_N(x)^{1/2}w_N(y)^{1/2}dxdy=\]
		\[\int \phi(x)\hat{\chi}_{N,i}(x)p_{N-n}(x)w_N(x)^{1/2}dx\int  \varphi(y)\chi_{N,0}(y)p_{N-m}(y)\sign(\lam_i-y)w_N(y)^{1/2}dy.\label{die-go}\]
		We note that as $\hat{\chi}_{N,i}$ clearly satisfies the bounds (\ref{qi-bounds}) and has a support contained within that of $\chi_{N,i}$, the above proof of (\ref{mer-bes-1}) with $\chi_{N,i}$ replaced by $\hat{\chi}_{N,i}$ shows that
		\[\int \phi(x)\hat{\chi}_{N,i}(x)p_{N-n}(x)w_N(x)^{1/2}dx=O(N^{-1}).\]
		We also observe that while Lemma \ref{mer-Bulk Single Integral of pN} does not formally apply to the second integral on the right-hand side of (\ref{die-go}), as the function $\phi(x)\sign(\lam_i-x)$ is $N$-dependant though $\lam_i$, this function (as well as its derivatives of up to any finite order) are still uniformly bounded on the support of $\chi_{N,0}$, so that the proof of Lemma \ref{mer-Bulk Single Integral of pN} with $\phi(x)$ replaced by $\phi(x)\sign(\lam_i-x)$ shows that for any $c>0$ we also have that
		\[\int  \varphi(y)\chi_{N,0}(y)p_{N-m}(y)\sign(\lam_i-y)w_N(y)^{1/2}dy=O(N^{-c}).\]
		Together these results show that
		\[\int\int  \phi(x)\varphi(y)\hat{\chi}_{N,i}(x)\chi_{N,0}(y)p_{N-n}(x)p_{N-m}(y)\sign(x-y)w_N(x)^{1/2}w_N(y)^{1/2}dxdy=O(N^{-2}).\label{colum}\]
		We also observe that $\supp(\chi_{N,i}-\hat{\chi}_{N,i})\subseteq (\lam_i-\delta/2,\lam_i-\delta/32)\cup (\lam_i+\delta/32,\lam_i+\delta/2)$, and that $\chi_{N,i}-\hat{\chi}_{N,i}$ satisfies the bounds (\ref{qi-bounds}). From this, we see that the proof of Lemma \ref{mer-Bulk double Integral of pN} with $\chi_{N,0}$ replaced by $\chi_{N, i}-\hat{\chi}_{N, i}$ establishes that
		\[\frac{1}{2}\int\int  \phi(x)\varphi(y)(\chi_{N,i}(x)-\hat{\chi}_{N,i}(x))\chi_{N,0}(y)p_{N-n}(x)p_{N-m}(y)\sign(x-y)w_N(x)^{1/2}w_N(y)^{1/2}dxdy=\]\[\frac{1}{2\pi N}\int_{-1}^{1} \frac{\phi(x)\varphi(x)(\chi_{N,i}(x)-\hat{\chi}_{N,i}(x))\chi_{N,0}(x)}{(1-x^2)}\sin((m-n)\arcos(x))dx+O(N^{\gamma-2})=\]
		\[\frac{1}{2\pi N}\int_{-1}^{1} \frac{\phi(x)\varphi(x)\chi_{N,i}(x)\chi_{N,0}(x)}{(1-x^2)}\sin((m-n)\arcos(x))dx+O(N^{\gamma-2})\label{mer-bes-2-mod}\]
		where in the last step we have used that $\hat{\chi}_{N,i}\chi_{N,0}(x)=0$. Combining (\ref{colum}) with (\ref{mer-bes-2-mod}) establishes (\ref{mer-bes-2}) in this case. The case of $i=0$ and $j>0$ of (\ref{mer-bes-2}) follows symmetrically, so we now assume that $i,j>0$.
		
		We note that if additionally $i\neq j$, then (\ref{mer-bes-2}) follows from (\ref{mer-bes-1}) by employing the disjointness of the support of $\chi_{N,i}$ and $\chi_{N,j}$ (\ref{disjoint-support}) to rewrite this as a product of single integrals. Thus we assume that $i=j$ and write, as before, $\lam_i=\lam$. As before, we additionally note that by the proof of (\ref{mer-bes-1}), we have that 
		\[\int\int \phi(x)\varphi(y)\chi_{N,i}(x)\chi_{N,i}(y)p_{N-n}(x)p_{N-m}(y)\sign(x-y)w_N(x)^{1/2}w_N(y)^{1/2}dxdy=\]
		\[\int_{\lam}^{\infty}\int_{\lam}^{\infty}  \phi(x)\varphi(y)\chi_{N,i}(x)\chi_{N,i}(y)p_{N-n}(x)p_{N-m}(y)\sign(x-y)w_N(x)^{1/2}w_N(y)^{1/2}dxdy+\]
		\[\int_{-\infty}^{\lam}\int_{-\infty}^{\lam}  \phi(x)\varphi(y)\chi_{N,i}(x)\chi_{N,i}(y)p_{N-n}(x)p_{N-m}(y)\sign(x-y)w_N(x)^{1/2}w_N(y)^{1/2}dxdy+O(N^{\gamma-2}).\label{mer split 4}\]
		We will compute the asymptotics of the first integral on the right-hand side, the other being identical. Proceeding as in Lemma \ref{both Bessel Lemma} we may rewrite this integral as the sum $I+II+III$ with
		\[I=\frac{1}{\pi}\sum_{k,l=1,2}\int_\lam^{\infty} [\cal{J}_{\alpha}(g_{N,n}(x))]_k(\int_\lam^{g_{N,m}(x)}[\cal{J}_{\alpha}(z)]_ldz)\frac{h_{N,k}(x)l_{N,l}(x)\chi_{N,i}(x)^2}{g_{N,m}'(x)}dx,\]
		\[II=\frac{1}{\pi}\sum_{k,l=1,2}\int_\lam^{\infty} (\int _\lam ^{g_{N,n}(x)}[\cal{J}_{\alpha}(z)]_kdz)(\int_\lam^{g_{N,m}(x)}[\cal{J}_{\alpha}(z)]_ldz)\frac{\chi_{N,i}(x)h_{N,k}(x)}{g_{N,n}'(x)}\frac{d}{dx}(\frac{\chi_{N,i}(x)l_{N,l}(x)}{g_{N,m}'(x)})dx,\]
		\[III=\frac{1}{2\pi}\sum_{k,l=1,2}\int_\lam^{\infty}\int_\lam^{\infty} (\int _\lam^{g_{N,n}(x)}[\cal{J}_{\alpha}(z)]_kdz)(\int_\lam^{g_{N,m}(y)}[\cal{J}_{\alpha}(z)]_ldz)\times\]\[\frac{d}{dx}(\frac{\chi_{N,i}(x)h_{N,k}(x)}{g_{N,n}'(x)})\frac{d}{dy}(\frac{\chi_{N,i}(y)l_{N,l}(y)}{g_{N,m}'(y)})\sign(x-y)dxdy,\]
		where $l_{N,k},h_{N,k}$ are as in the proof of Lemma \ref{both Bessel Lemma}. We note that for any $c\in \N$, we have that $l_{N,k}^{(c)}(x),h_{N,k}^{(c)}(x)=O(1)$, which when supplemented with the asymptotics of (\ref{ig-ig2}) establish that there is $C$ such that
		\[|III|\le \frac{C}{N^2}\int \int (|\chi_{N,i}'(x)|+|\chi_{N,i}(x)|)(|\chi_{N,i}(y)|+|\chi_{N,i}'(y)|)dxdy,\]
		which shows that $III=O(N^{-2})$. Similarly
		\[|II|\le \frac{C}{N^2}(\sup_{x}|\chi_{N,i}(x)|)(\int|\chi_{N,i}'(x)|dx+\int|\chi_{N,i}(x)|dx),\]
		so that $II=O(N^{-2})$. As before, we may write $I$ as the sum of two terms
		\[I'=-\sum_{k,l=1,2}\int_\lam^{\infty} \sq{g_{N,n}(x)}J_{\alpha-(-1)^k1/2}(g_{N,n}(x))I_{\alpha-(-1)^l1/2,1/2}(g_{N,m}(x))\frac{h_{N,k}(x)l_{N,l}(x)\chi_{N,i}(x)^2}{g_{N,m}'(x)}dx\]
		\[II'=\sum_{k,l=1,2}D(\alpha-(-1)^l1/2,1/2)\int_\lam^{\infty} J_{\alpha-(-1)^k1/2}(g_{N,n}(x))\frac{h_{N,k}(x)l_{N,l}(x)\chi_{N,i}(x)^2}{g_{N,m}'(x)}dx.\]
		Employing Lemma \ref{SingleBesselLemma}, we see that there is $C$ such that
		\[|II'|\le \frac{C}{N^2}+\frac{C}{N^3}\sum_{l=0}^{2}\int_0^{\infty}|\frac{d^l}{dx^l}(\chi_{N,i}(x)^2)|dx\le\]\[ C^2\bigg(N^{-2}+N^{-3}(\sup_{x}|\chi_{N,i}(x)|)\sum_{l=0}^{2}\int_0^{\infty}|\chi_{N,i}^{(l)}(x)|dx+N^{-3}(\sup_{x}|\chi_{N,i}'(x)|)\int_0^{\infty}|\chi_{N,i}'(x)|dx\bigg)=C^3(N^{-2}+N^{\gamma-3}).\label{mer-2-new}\]
		Employing  the asymptotics (\ref{IntegralAsymptotic}) and (\ref{Bessel Asymptotic}), as in the proof of Lemma \ref{both Bessel Lemma}, we see that
		\[I'=\frac{1}{(N-m)\pi}\sum_{k,l=1,2}\int_{\lam}^{\infty} [\cal{I}_\alpha(g_{N,n}(x))]_k[\cal{I}_{\alpha-1}(g_{N,m}(x))]_l\frac{h_{N,k}(x)l_{N,l}(x)\chi_{N,i}(x)^2}{f_{N,m,\lam}'(x)}dx+O(N^{-2}).\label{sti-1}\]
		Now using the fact that $R_N(x),\bar{R}_N(x)=O(N^{\gamma-1})$ for $x\in (\lam-\delta_N,\lam+\delta_N)$, we derive that
		\[I'=\frac{1}{(N-m)\pi}\sum_{k,l=1,2}\int_{\lam}^{\infty} [\cal{I}_\alpha(g_{N,n}(x))]_k[\cal{I}_{\alpha-1}(g_{N,m}(x))]_l\phi(x)
		\varphi(x)\chi_{N,i}(x)^2\times \]\[\frac{[T_{N,n,\lam}(x)]_{1k}[T_{N,m,\lam}(x)]_{1l}}{f_{N,m,\lam}'(x)}dx+O(N^{\gamma-2}).\]
		The remainder of the proof of (\ref{mer-bes-2}) now proceeds identically to that of Lemma \ref{both Bessel Lemma}. 
		
		To show (\ref{mer-bes-2-1}) one proceeds identically to the case of (\ref{mer-bes-2}) above, with the only differences being that now the error in (\ref{mer split 4}) is given by $O(N^{-1-\min(1,2(1-\gamma))})$, $II,III=O(N^{\gamma-2})$ and $II'=O(N^{-\min(3-2\gamma,2)})$. The
		case of (\ref{mer-bes-2-2}) is similar as well, with the error in (\ref{mer-bes-2-mod}) now being given by $O(N^{2(\gamma-1)})$, the error in (\ref{mer split 4}) now being given by $O(N^{-2\min(1,2(1-\gamma))})$, and with the new error bounds $II,III=O(N^{2(\gamma-1)})$ and $II'=O(N^{-\min(3(1-\gamma),2)})$
	\end{proof} 
	
	Now we will now discuss the remaining asymptotics which we need outside of $(\lam_2-\delta,\lam_1+\delta)$. As we will not need $N$-dependant partitions in this region, and as the bound on the error term given in Proposition \ref{Merging Error Bulk} is still of order $O(N^{-1})$, we see that the proofs given for Lemmas \ref{Bulk Single Integral of pN}, \ref{Bulk double Integral of pN}, and \ref{both airy lemma}, require no modifications on their respective regions. Explicitly, together they imply the following result.
	\begin{lem}
		\label{mer-Analytic Integral Double Integral}
		Let $\phi$ and $\varphi$ be smooth functions of subexponential growth, both vanishing on $(\lam_2-\delta,\lam_1+\delta)$. Then we have that for any fixed $n$ and $m$ that
		\[\int \phi(x)p_{N-k}(x)w_N(x)^{1/2}dx=O(N^{-1/2}).\]
		\[\int  J_N^{-1}(\varphi p_{N-m})(x)\phi(x)p_{N-n}(x)w_N(x)dx=\frac{1}{2\pi N}[\int_{-1}^{1} \frac{\phi(x)\varphi(x)}{(1-x^2)}\sin((m-n)\arcos(x))dx]+\]\[\frac{1}{4N}[(-1)^{N-m}\phi(1)\varphi(-1)-(-1)^{N-n}\phi(-1)\varphi(1)]+O(N^{-7/6}).\]
	\end{lem}
	Combining this Lemma with the above results, we achieve the following.
	\begin{lem}
		Let $\phi$ and $\varphi$ be smooth functions of subexponential growth. Then we have that for any fixed $n$ and $m$ that
		\[\int J_N^{-1}(\varphi p_{N-m})(x)\phi(x)p_{N-n}(x)w_N(x)dx=\frac{1}{2\pi N}[\int_{-1}^{1} \frac{\phi(x)\varphi(x)}{(1-x^2)}\sin((m-n)\arcos(x))dx]+\]
		\[\frac{1}{4N}[(-1)^{N-m}\phi(1)\varphi(-1)-(-1)^{N-n}\phi(-1)\varphi(1)]+O(N^{-\min(2-\gamma,7/6)}).\]
	\end{lem}
	Proceeding as in the prior section, we see that this result implies that when $N$ is even that
	\[\int  J_N^{-1}( p_{N})(x)p_{N-1}(x)w_N(x)dx=O(N^{-\min(2-\gamma,7/6)}),\]
	which is the first result required by Proposition \ref{merging-W-Integral-Proposition}.
	
	We will now begin discussing the required integrals involving $\ell_i$ and $q_i$. A technicality present in this case is that the term $c_i$ defined in Proposition \ref{l and q asym} is now $N$-dependant and diverging as $|\lam_1-\lam_2|\to 0$. To avoid confusion, we will use the notation $c_{i, N}$ for the remainder of the section to emphasize this $N$-dependence. We also observe that for $i\in \{1,2\}$, both $c_{i,N}|\lam_1-\lam_2|^{\alpha}$ and $c_{i,N}^{-1}|\lam_1-\lam_2|^{-\alpha}$ are uniformly bounded for $\lam_1,\lam_2\in (-1+\epsilon,1-\epsilon)$.
	
	\begin{lem}
		\label{mer-single-integral-lemma-l-q}
		For $\phi$ a smooth function supported on $(-1+\delta,1-\delta)$, and any choice of $i\in {1,2}$ and $j\in \{1,2,3\}$, we have that
		\[ \int \phi(x)\chi_{N,j}(x)\ell_i(x)w_N(x)^{1/2}dx=c_{i,N}^{-1}O(N^{\gamma-1}),\;\;\int \phi(x)\chi_{N,j}(x)q_i(x)w_N(x)^{1/2}dx=c_{i,N}O(1).\]
	\end{lem}
	\begin{proof}
		In sight of Proposition \ref{Merging Error Bessel}, and proceeding as in the proof of Lemma \ref{single-integral-lemma-l-q}, we are reduced to showing that
		\[\int \phi(x)\chi_{N,j,i}(x)\begin{bmatrix}\pi^{1/2}p_{N}(x)\\-i\pi^{1/2}p_{N-1}(x)\end{bmatrix}w_N(x)^{1/2}dx=T_{N,\lam_i}(\lam_i)\begin{bmatrix}O(1)\\0\end{bmatrix}+O(N^{\gamma-1}).\label{zoltron}\]
		We will actually show that the error is of order $O(N^{-1})$, but as the error terms in Proposition \ref{Merging Error Bessel} are of order $O(N^{\gamma-1})$, this does not improve the result. Now in the case that $i\neq j$, (\ref{zoltron}) follows from (\ref{mer-bes-1-1}), and thus we will focus on the case of $i=j$, and denote $\lam=\lam_i$ and $\alpha=\alpha_i$ as before. Applying Proposition \ref{Merging Error Bessel} we may write
		\[\int \phi(x)\chi_{N,i,i}(x)\begin{bmatrix}\pi^{1/2}p_{N}(x)\\-i\pi^{1/2}p_{N-1}(x)\end{bmatrix}w_N(x)^{1/2}dx=\int \frac{\xi_{N}(x)G_{N,1}(x)}{x}\sign(x)\sq{\pi |Nx|}J_{\alpha+1/2}(N|x|)dx+\]\[\int \frac{\xi_{N}(x)G_{N,2}(x)}{x}\sq{\pi |Nx|}J_{\alpha-1/2}(N|x|)dx=I+II,\label{mer-irrev-no}\]
		where $G_{N,i}$ is defined as in the proof of Lemma \ref{single-integral-lemma-l-q}, and $\xi_{N}(x)=\chi_{N,i,i}(f_{\lam}^{-1}(x))x$. By Lemma \ref{SingleBesselLemma} we have that
		\[I=N^{-1}2\pi^{1/2}D(\alpha+1/2,-1/2)G_{N,1}(0)\xi_{N}(0)+E_N,\]
		where there is $C$ such that
		\[|E_N|\le \frac{C}{N^2}\sum_{l=0}^{2}\int|\chi_{N,i}^{(l)}(x)|dx\le C^2N^{\gamma-2}.\]
		Thus $I=O(N^{-1})$. Proceeding similarly, one obtains that $II=O(N^{-1})$.
	\end{proof}
	
	As in (\ref{first-order-form}), we see from the asymptotics provided in Proposition \ref{Merging Error Bessel} that for $i\in \{1,2\}$, there are constants $C_{k,N}$, for $k\in \{1,2,3,4\}$, such that
	\[c_{i,N}\ell_{i}(x)=C_{1,N}\frac{p_{N}(x)}{x-\lam_i}+C_{2,N}\frac{p_{N-1}(x)}{x-\lam_i} \label{mer-first-order-form-1},\]
	\[c_{i,N}^{-1}q_{i}(x)=C_{3,N}\frac{p_{N}(x)}{x-\lam_i}+C_{4,N}\frac{p_{N-1}(x)}{x-\lam_i}, \label{mer-first-order-form-2}\]
	and such that $C_{k,N}=O(1)$ for $k\in \{1,2,3,4\}$.
	
	We recall the functions $\hat{\chi}_{N,i}$ introduced in the proof of Lemma \ref{mer-both Bessel Lemma}, which are similar to $\chi_{N,i}$, but defined so that $\supp(\hat{\chi}_{N,i})\cap \supp(\chi_{N,0})=\varnothing$. Proceeding as in the previous section, we see that
	\[\int J_N^{-1}\ell_l(x) p_{N-1}(x)w_N(x)dx=\int (J_N^{-1}\chi_{N,l}\ell_l)(x) \hat{\chi}_{N,l}(x)p_{N-1}(x)w_N(x)dx+c_{l,N}^{-1}O(N^{-\min(1,(1-\gamma)-1/2)}),\]
	\[\int J_N^{-1}p_{N}(x)q_k(x) w_N(x)dx=\int (J_N^{-1}\chi_{N,l}p_{N})(x)\hat{\chi}_{N,k}(x)q_k(x) w_N(x)dx+c_{k,N}O(N^{-1/2}),\]\[
	\int J_N^{-1}\ell_l(x)q_k(x)w_N(x)dx=\delta_{k,l}\int J_N^{-1}(\chi_{N,l}\ell_{l})(x)\hat{\chi}_{N,k}(x)q_k(x)w_N(x)dx+c_{k,N}c^{-1}_{l,N}O(N^{-\min(1/2,2(1-\gamma))}).\]
	We observe that $c_{k,N}c^{-1}_{l,N}=O(1)$. By these formulas, and (\ref{mer-first-order-form-1}-\ref{mer-first-order-form-2}), we see that to establish the remaining cases of Proposition \ref{merging-W-Integral-Proposition}, it suffices to show the following.
	
	\begin{lem}
		\label{mer-Over Integral Double Integral}
		Let $\phi$ and $\varphi$ be smooth. Then for any choice of $i=1,2$, and any $n$ and $m$, we have that
		\[\int \int \phi(x)\varphi(y)\chi_{N,i,i}(x)\chi_{N,i}(y) p_{N-n}(x)p_{N-m}(y)\sign(x-y)w_N(x)^{1/2}w_N(y)^{1/2}dxdy=O(N^{-1}\log(N)).\label{mer-Integral-1-Over}\]
		Moreover, if $|n-m|\le 1$, then additionally
		\[\int \int \phi(x)\varphi(y)\chi_{N,i,i}(x)\chi_{N,i,i}(y) p_{N-n}(x)p_{N-m}(y)\sign(x-y)w_N(x)^{1/2}w_N(y)^{1/2}dxdy=O(N^{\gamma-1}\log(N)). \label{mer-Integral-1-Over-2}\]
	\end{lem}
	\begin{proof}
		We begin with the proof of (\ref{mer-Integral-1-Over}). We may proceed identically to the proof of (\ref{Integral-1-Over}) to write the left-hand side of (\ref{mer-Integral-1-Over}) as the sum $I+II$ where
		\[I=\frac{1}{\pi}\sum_{k,l=1,2}\int [\cal{J}_{\alpha}(g_{N,m}(x))]_lI_k(g_{N,n}(x))\frac{h_{N,k}(x)l_{N,l}(x)g_{N,n}(x)\chi_{N,i}(x)^2}{g_{N,n}'(x)(x-\lam)}dx,\]
		\[II=-\frac{1}{2\pi}\sum_{k,l=1,2}\int \int [\cal{J}_{\alpha}(g_{N,m}(x))]_lI_k(g_{N,n}(y))\chi_{N,i}(y)l_{N,l}(y)\frac{d}{dx}(\frac{\chi_{N,i}(x)h_{N,k}(x)g_{N,n}(x)}{g_{N,n}'(x)(x-\lam)})\sign(x-y)dxdy,\]
		with notation as in the proof of Lemma \ref{Over Integral Double Integral}. As $|\chi_{N,j}(x)|$ is bounded, and as the integrand of $I$ remains of order $O(\min(N|x|^{-1},1))$, we see that $I=O(N^{-1}\log(N))$, as before. In addition, we see there is $C$ such that
		\[|II|\le \frac{1}{2\pi}\sum_{k,l=1,2}\int  |\chi_{N,i}(y)I_k(g_{N,n}(y))l_{N,l}(y)|dy\int |[\cal{J}_{\alpha}(g_{N,m}(x))]_l\frac{d}{dx}(\frac{\chi_{N,i}(x)h_{N,k}(x)g_{N,n}(x)}{g_{N,n}'(x)(x-\lam)})|dx \le \]
		\[\frac{C}{N}\int \min(|y|^{-1},N)dy\int (|\chi_{N,i}(x)|+|\chi_{N,i}'(x)|)dx=O(N^{-1}\log(N)).\]
		This completes the proof of (\ref{mer-Integral-1-Over}).
		
		We now discuss the proof of (\ref{mer-Integral-1-Over-2}). The proof of (\ref{mer-Integral-1-Over-2}) may also be reduced, by modifying the proof of (\ref{Integral-1-Over-2}) as above, to showing that for $k,l=1,2$,
		\[\int\int \frac{\chi_{N,i}(x-\lam)\chi_{N,i}(y-\lam)}{xy}[\cal{J}_{\alpha}(Nx)]_k[\cal{J}_{\alpha}(Ny)]_l\sign(x-y)dxdy=O(N^{\gamma-1}\log(N)).\label{mer-Weird-Integral-Over-1}\]
		On the other hand, we may rescale the left-hand side of  (\ref{mer-Weird-Integral-Over-1}) to become
		\[\int\int \frac{\chi(x)\chi(y)}{xy}[\cal{J}_{\alpha}(N\delta_Nx)]_k[\cal{J}_{\alpha}(N\delta_Ny)]_l\sign(x-y)dxdy.\]
		This coincides with a rescaled case of (\ref{Weird-Integral-Over-1}), which we have already shown to be of order \[O(\delta_N^{-1} N^{-1}\log(\delta_N^{-1} N^{-1})),\] and thus is further of order $O(N^{\gamma-1}\log(N))$. This completes the proof of (\ref{mer-Weird-Integral-Over-1}).
	\end{proof}

	%\bibliographystyle{acm}
	%\bibliography{main}
\end{document}